\title{The topology of the space of J-holomorphic maps to $\p$}
\author{
 Jeremy Miller \\
                Department of Mathematics\\
        Stanford University\\
        Building 380, Stanford, CA
}
\date{\today}

\documentclass[12pt]{article}
\usepackage{mathtools}
\usepackage[section]{placeins}

\usepackage {amsmath, amscd, amsbsy, amsfonts,amsthm}
\usepackage{latexsym,amssymb,mathrsfs,bbm}
\usepackage[all]{xy}
\usepackage{graphicx}
\usepackage{color}
\usepackage{tikz}
\usetikzlibrary{arrows,calc,positioning,decorations.pathreplacing}
\usepackage[noadjust]{cite}
\usepackage[left=4cm,right=4cm]{geometry}
\usepackage{mhsetup}
\usepackage{mathtools}

\usepackage[pagebackref,colorlinks,citecolor=blue,linkcolor=blue,urlcolor=blue,filecolor=blue]{hyperref}

\DeclareMathOperator{\Hol}{Hol}
\DeclareMathOperator{\Map}{Map}

\newcounter{prob}
\newcommand{\N}{\mathbb{N}_0}
\newcommand{\Z}{\mathbb{Z}}
\newcommand{\C}{\mathbb{C}}
\newcommand{\M}{\mathcal{M}}

\newcommand{\p}{\mathbb{C} P^2}
\newcommand{\s}{\mathbb{C} P^1}
\newcommand{\R}{\mathbb{R}}

\newcommand{\m}{\longrightarrow}

\newcommand{\n}{\bar \partial_\nu}
\newcommand{\ft}{Fr\'echet }

\newcommand{\U}{\tilde U}
\newcommand{\V}{\tilde V}
\newcommand{\D}{\mathbb{D}}
\newcommand{\dbar}{\bar \partial}
\newcommand{\z}{\mathfrak z}
\newcommand{\Simga}{\Sigma}

\newcommand{\MM}{\mathfrak M}

\theoremstyle{plain}
\newtheorem{theorem}{Theorem}[section]
\newtheorem{lemma}[theorem]{Lemma}
\newtheorem{proposition}[theorem]{Proposition}
\newtheorem{corollary}[theorem]{Corollary}

\theoremstyle{definition}
\newtheorem{definition}[theorem]{Definition}
\newtheorem{remark}[theorem]{Remark}

\begin{document}
\maketitle

\begin{abstract}

The purpose of this paper is to generalize a theorem of Segal from \cite{Se} proving that the space of holomorphic maps from a Riemann surface to a complex projective space is homology equivalent to the corresponding space of continuous maps through a range of dimensions increasing with degree. We will address if a similar result holds when other almost complex structures are put on a projective space. For any compatible almost complex structure $J$ on $\C P^2$, we prove that the inclusion map from the space of $J$-holomorphic maps to the space of continuous maps induces a homology surjection through a range of dimensions tending to infinity with degree. The proof involves comparing the scanning map of topological chiral homology (\cite{Lu5}, \cite{An}, \cite{Mi2}) with analytic gluing maps for J-holomorphic curves (\cite{MS}, \cite{S}). This is an extension of the work done in \cite{Mi1} regarding genus zero case.

\end{abstract}

\tableofcontents

\section{Introduction}

This paper represents a continuation of the work done in \cite{Mi1} generalizing the results of Segal in \cite{Se}. To state these results, let us first fix some notation.

\begin{definition}

For a complex manifold $M$, complex curve $(\Sigma_g,j)$ of genus $g$ and $A\in H_2(M)$, let $\Hol_A((\Sigma_g,j),M)$ denote the space of holomorphic maps $u:\Sigma_g \m M$ with $u_* [\Sigma_g ]=A$. Fix points $pt \in \Sigma_g$ and $m_0 \in M$. Let $\Hol^*_A((\Sigma_g,j),M)$ denote the subspace of maps $u$ with $u(pt)=m_0$. Likewise, let $\Map_A((\Sigma_g,j),M)$ and $\Map_A^*((\Sigma_g,j),M)$ denote the corresponding spaces of continuous maps.

\end{definition}

The complex structure on $\Sigma_g$ and $\C P^m$ respectively give a preferred choice of fundamental class on $\Sigma_g$ and a natural isomorphism $H_2(\C P^m) = \Z$. The integer corresponding to a homology class is called degree and agrees with the algebraic notion of degree for algebraic maps. In \cite{Se}, Segal proved a theorem comparing the topology of $\Hol_d^*((\Sigma_g,j), \C P^m)$ and  $\Map_d^*((\Sigma_g,j), \C P^m)$.

\begin{theorem}
The inclusion map $i:\Hol_d^*((\Sigma_g,j),\C P^m) \m \Map_d^*(\Sigma_g,\C P^m)$ induces an isomorphism on homology groups $H_i$ for $i<(d-2g)(2m-1)$ and a surjection for $i=(d-2g)(2m-1)$.
\end{theorem}

This theorem has been generalized in many ways. For example, Guest proved a similar result when complex projective space is replaced by an arbitrary toric variety \cite{Gu}. We consider a different generalization motivated by symplectic geometry. In personal communication with Ralph Cohen, Yakov Eliashberg asked the following question:

\paragraph{Question} Is Segal's theorem still true if the standard complex structure on $\C P^m$ is replaced by some other almost complex structure $J$ on $\mathbb{C} P^m$ compatible with the standard symplectic form?
$ \\$

Even for complex manifolds, it is still interesting to consider $J$-holomorphic curves for non-integrable almost complex structures $J$. By allowing the almost complex structure to vary, one can show that various moduli spaces are manifolds or smooth orbifolds. Also, since embedded $J$-holomorphic curves are symplectic submanifolds, the topology of spaces of $J$-holomorphic curves is relevant to the symplectic isotopy problem \cite{S} \cite{ST}. Although the space of almost complex structures compatible with a symplectic form is contractible, the topology of the space of $J$-holomorphic curves can depend on choice of $J$. For example, in \cite{A}, Abreu considered the case of a symplectic form on $S^2 \times S^2$ where the two spheres have different areas. He proved that the space of $J$-holomorphic curves representing the homology class of the larger sphere can be empty or non-empty depending on the choice of compatible almost complex structure. 

The question of generalizing Segal's theorem to the case of $J$-holomorphic curves has been discussed in \cite{CJS}, \cite{Sav} and \cite{Mi1}. In \cite{CJS}, they discuss a possible Morse theoretic approach and in \cite{Sav}, they describe a related problem involving the notion of $q$-complete symplectic manifolds. In \cite{Mi1}, a partial answer to Eliashberg's question was given for genus zero curves in $\p$. Namely, the author proved the following theorem.

\begin{theorem}
Let $J$ be any almost complex structure on $\p$ compatible with the standard symplectic form. The inclusion map $$i : \Hol_d^*(\s,(\p,J)) \m \Map_d^*(\s,\p)$$ induces a surjection on $H_i$ for $i \leq 3d-3$.

\label{main}
\end{theorem}

The purpose of this paper is to generalize the above theorem to higher genus $J$-holomorphic curves in $\p$. Specifically, we prove the following theorem.

\begin{theorem}
Let $J$ be any compatible almost complex structure on $\p$. The inclusion map $$i: \Hol_d^*(\Sigma_g,(\p,J)) \m \Map_d^*(\Sigma_g,\p)$$ induces a surjection on homology groups $H_i$ for $i \leq (d-g-1)/3$.\label{theorem:big}
\end{theorem}

Here the symbol $\Hol(\Sigma_g, \cdot)$ denotes the space of all maps from $\Sigma_g$ which are holomorphic for some complex structure on $\Sigma_g$. In Remark \ref{oneJ}, we note that this theorem can be strengthened to show that there are complex structures $j$ on $\Sigma_g$ such that $\Hol_d^*((\Sigma_g,j),(\p,J)) \m \Map_d^*(\Sigma_g,\p)$ is a homology surjection in range. However, our methods do not suffice to establish this strengthened result for all complex structures on $\Sigma_g$.

These results are a generalization of Segal's theorem in \cite{Se} in the sense that general almost complex structures on the range are considered, but also are a weakening of Segal's theorem in the sense that homology isomorphism is replaced by homology surjection. The range of dimensions is also worse than that of Segal and we only consider very special complex structures on the domain whereas Segal's result is true for any fixed complex structure on the domain. 

Segal's proof involves first noting that all holomorphic maps between $\Sigma_g$ and $\C P^m$  are algebraic. Then he relates spaces of polynomials to configuration spaces of labeled particles which can be compared to mapping spaces via scanning maps in a manner similar to \cite{Se3}, \cite{Mc1} and \cite{K3}. However, for general almost complex structures, it is completely unclear how to produce a configuration space model for $J$-holomorphic mapping spaces.

The study of $J$-holomorphic curves is very different in dimensions $2$, $4$, and above. All almost complex structures on real 2 dimensional manifolds are integrable and $\s$ admits a unique complex structure up to diffeomorphism. Thus Segal's original theorem applies directly to all almost complex structures on $\s$. We focus on the case of real $4$ dimensional symplectic manifolds because of the phenomenon of automatic transversality discovered by Gromov in \cite{G}. Using Gromov compactness and automatic tranversality, Gromov proved that the topology of the space of degree one rational $J$-holomorphic maps to $\p$ is independent of choice of almost complex structure. We extend this result to degree two rational $J$-holomorphic maps. Because automatic transversality does not hold for manifolds of dimension $6$ or above, we have so far been unable to make progress generalizing Segal's theorem in higher dimensions.

Even in dimension $4$, complications arise preventing us from simply proving that the topology of $J$-holomorphic mapping spaces is independent of $J$ in higher degrees and genera. However, we leverage these results about low degree genus zero $J$-holomorphic maps to draw conclusions about higher degree and genus mapping spaces using a gluing argument. This is similar in spirit to the arguments used by Atiyah and Jones in  \cite{AJ} to study the topology of spaces of solutions to the Yang-Mills equation. 

For $J_0$ the standard integrable complex structure, Frederick Cohen introduced a little 2-disks ($D_2$) algebra structure on $\Hol^*(\s,(\p,J_0))$ given by juxtaposition of roots of polynomials \cite{BM}. This gives a way of constructing high degree $J_0$-holomorphic rational maps out of low degree holomorphic rational maps. In \cite{Mi1}, this construction was generalized using a gluing construction to construct a partial $D_2$-algebra structure on $\Hol^*(\s,(\p,J))$ for any compatible almost complex structure $J$. To study higher genus mapping spaces, we use ideas from topological chiral homology \cite{Sa} \cite{Lu5} \cite{An} \cite{Mi2}.

Topological chiral homology gives a construction which takes in an $E_n$-algebra in spaces $A$ and parallelized
 $n$-manifold $M$ as inputs and produces a space $\int_{M} A\,$. Intuitively, this space can be thought of as a space of collections of points in $M$ with labels in $A$, topologized in such a way that when points collide, their labels add using the $E_n$ structure. The space $\int_{M} A\,$ comes with a natural map $s: \int_{M} A\, \m \Map^c(M,B^nA)$, with $B^n A$ being the $n$-fold delooping of $A$ and $\Map^c$ denoting the space of compactly supported maps. This map is called the scanning map and is often a homotopy equivalence \cite{Lu5} and always a stable homology equivalence \cite{Mi2}. The term stable homology equivalence is a slight weakening of the condition that a map induces a homology equivalence though a range of dimensions tending to infinity. See Section \ref{secTCH} for a definition. To illustrate how ideas from topological chiral homology could be useful, we give a hypothetical proof of Theorem \ref{theorem:big} under the (still unknown) hypothesis that the topology of the space $\Hol^*(\s, (\C P^m, J))$ is independent of $J$. Afterwards, we will explain what modifications are needed to make this hypothetical proof into an actual proof.

\paragraph{Hypothetical proof}

First suppose that we could prove that the topology of the space $\Hol^*(\s, (\C P^m, J))$ is independent of $J$ and that this space has a natural $D_2$-action. Furthermore, suppose that we could construct a gluing map $g$ making the following diagram homotopy commute: $$
\begin{array}{ccccccccl}
\int_{\Sigma_g -pt} \Hol^*(\s, (\C P^m,J))\,  &\overset{g}{\m} & \Hol^*(\Sigma_g, (\C P^m,J)) \\
  \downarrow s  &   & \downarrow i    &            \\

 \Map^c(\Sigma_g-pt,  B^2 \Hol^*(\s, (\C P^m,J)) & \m  & \Map^*(\Sigma_g, \C P^m) \\

\end{array}$$ If we assume that the topology of $\Hol^*(\s , (\C P^m,J))$ is independent of $J$, then Segal's theorem implies that $B^2 \Hol^*(\s , (\C P^m,J)) \simeq \C P^m$ and hence the bottom row is a homotopy equivalence. Since the scanning map $s$ is a stable homology equivalence, we could conclude that $i$ is stably a homology surjection.
$$ $$

Unfortunately, we have been unable to prove that the topology of the spaces $\Hol^*(\s, (\C P^m, J))$ are independent of choice of almost complex structure $J$. However, it is true that the topology of $\Hol_k^*(\s, (\C P^2, J))$ is independent of $J$ for $k \leq 2$. We will define a partial $D_2$-algebra which we call $R_2$ which naturally maps to a neighborhood of $\bigsqcup_{k \leq 2} \Hol_k^*(\s,(\p,J))$ in $\Map^*(\s,\p)$ and has the property that the $2$-fold delooping of its completion as a $D_2$-algebra is homotopy equivalent to $\p$. Moreover, we show that the scanning map $s:\int_{\Sigma_g-pt} R_2 \m \Map^*(\Sigma_g ,\p)$ induces a homology equivalence on connected components in a range tending to infinity with degree.  Thus, our hypothetical proof becomes an actual proof in the case of $\p$ after replacing $\Hol^*(\s, (\C P^2,J))$ by the yet to be defined partial $D_2$-algebra $R_2$. The definition of the map $R_2 \m \bigsqcup_{k \leq 2} \Map_k^*(\s,\p)$ uses the fact that $\bigsqcup_{k \leq 2} \Hol_k^*(\s,(\p,J))$ is homotopy equivalent to $\bigsqcup_{k \leq 2} \Hol_k^*(\s,(\p,J_0))$. This is the primary difficulty in extending this argument to higher dimensions. 

We prove these properties of $\int_{\Sigma_g-pt} R_2$ by proving that it is homotopy equivalent to the so called bounded symmetric product of $\Sigma_g-pt$. The work of Kallel in \cite{K2} and Yamaguchi in \cite{Y} show that these bounded symmetric products give models of the space of continuous maps from a surface to a complex projective space. These results imply that $H_*(\int_{\Sigma_g} R_2)$ approximates $H_*( \Map^*(\Sigma_g,\p))$.

To construct the gluing map $g$, we follow \cite{S}, which is a generalization of the techniques of Appendix A of \cite{MS}. As in \cite{S}, we use automatic transversality arguments to prove the transversality conditions needed to construct gluing maps.

The organization of the paper is as follows: In Section \ref{secTCH}, we discuss scanning theorems for topological chiral homology. In Section \ref{KY}, we define $R_2$ and compare its topological chiral homology to bounded symmetric products. In Section \ref{secorb}, we review some aspects of the theory of orbifolds which will be needed when analyzing various moduli spaces. In Section \ref{secauto}, we review the basic theory of $J$-holomorphic curves and automatic transversality. In Section \ref{secdeg}, we review Gromov's proof that the space of degree one rational $J$-holomorphic maps is independent of $J$ and extend these arguments to the case of degree two maps. In Section \ref{secgluing}, we construct the gluing map $g$. Finally, in Section \ref{sechomotopy}, we compare $g$ to scanning maps and prove the main theorem.

\paragraph{Acknowledgments}
This paper is based on work done by the author as part of his doctoral dissertation at Stanford University, conducted under the supervision of Ralph Cohen. The author would like to thank him as well as Ricardo Andrade, Eleny Ionel, Alexander Kupers and John Pardon for many helpful discussions. Additionally, I am grateful for the useful and important comments from the referee which, in particular, led to the discovery of a major error in an earlier version of this paper.

\section{Topological chiral homology}\label{secTCH}

Topological chiral homology refers to a family of equivalent constructions introduced by \cite{Sa}, \cite{Lu5}, \cite{An} and others. Given a $D_n$-algebra $A$ and a parallelized $n$-manifold $M$, one constructs a space $\int_M A$ and a map $s : \int_M A \m \Map^c(M,B^n A)$, the space of compactly supported maps. This map is called the scanning map and induces an isomorphism in homology after a procedure called stabilizing which we recall in this section. We will be interested in the case when the manifold is $\Sigma_g-pt$ since we are interested in the space $\Map^*(\Sigma_g,\p) \simeq \Map^c(\Sigma_g-pt,\p) $. In this section we review Andrade's model of topological chiral homology \cite{An} and state a scanning theorem from \cite{Mi2}. This model of topological chiral homology is constructed via May's two-sided bar-construction \cite{M}. See \cite{Mi2} for the proof of the main theorem of this section. See \cite{Fr1} for a proof that Andrade's model of topological chiral homology is equivalent to that of Lurie in \cite{Lu5}.

\subsection{Operads, algebras and modules}\label{secOAM}
In this subsection we review the definition of operads and their modules and algebras. In this paper, we only work in the category of spaces. However, most constructions work more generally, and for other applications, it is interesting to work with chain complexes or spectra. The data of an operad includes the data of $\Sigma$-space (called symmetric sequences in spoken language).

\begin{definition}
A $\Sigma$-space is a collection of spaces $\{ X(k) \}_{k \in \N}$ such that $X(k)$ has an action of the symmetric group $\Sigma_k$. A map between $\Sigma$-spaces $f:X \m Y$ is a collection of equivariant maps $\{f_k\}_{k \in \N}$ with $f_k:X(k) \m Y(k)$.
\end{definition}

The category of $\Sigma$-spaces has a (non-symmetric) monoidal structure defined as follows.

\begin{definition}
For $X$ and $Y$ $\Sigma$-spaces, $X \otimes Y$ is the $\Sigma$-space such that $(X \otimes Y)(k) = \bigsqcup_{j=0}^{\infty} X(j) \times_{\Sigma_j} \bigsqcup_{f \in \Map(k,j)} \prod_{i=1}^j Y(|f^{-1}(i)|)$. Here $\Map(k,j)$ is the set of maps from $\{1, \ldots, k\}$ to  $\{1, \ldots, j\}$ and $\Sigma_k$ acts via precomposition.

\end{definition}

Note that the unit with respect to this product is given by the $\Sigma$-space $\iota$ with:  $$\iota(n) = \begin{cases}
 \{id\} & \text{if }  n=1 \\
\varnothing & \text{if } n \neq 1
\end{cases}$$

\begin{definition}
An operad is a monoid in the category of $\Sigma$-spaces.
\end{definition}

In other words, an operad $\mathcal O$ is a $\Sigma$-space with maps $m : \mathcal O \otimes \mathcal O \m \mathcal O$ and $i: \iota \m \mathcal O$ satisfying the obvious compatibility relations. Denote the image of $\iota$ by $id \in \mathcal O(1)$. For the purposes of this paper, we will be primarily interested in the little $2$-disks operad $D_2$. Let $\D^n$ denote the unit $n$-disk in $\R^n$.

\begin{definition}
Let $D_n$ be the $\Sigma$-space with $D_n(k)$ being the space of disjoint embeddings of $\bigsqcup_{i=1}^k  \D^n$ into $\D^n$ which are compositions of dilations and translations. Topologize this space with the subspace topology inside the space of all continuous maps with the compact open topology. This forms an operad via composition of embeddings.
\end{definition}

\begin{definition}
Let $\mathcal O$ be an operad. A left $\mathcal O$-module is a $\Sigma$-space $\mathcal L$ and a map $p:\mathcal O \otimes \mathcal L \m \mathcal L $ such that the following diagrams commute:

$$
\begin{array}{ccccccccl}
 \mathcal O \otimes \mathcal O \otimes \mathcal L  &\overset{id \otimes p}{\m} & \mathcal O \otimes \mathcal L & & \iota \otimes \mathcal L  &\overset{i}{\m} & \mathcal O \otimes \mathcal L  \\
m \otimes id \downarrow  & & p \downarrow  &       &    &  \searrow & p \downarrow  \\

\mathcal O \otimes \mathcal L  &\overset{p}{\m} & \mathcal L  & &  & & \mathcal L\\

\end{array}$$

\label{definition:module}

\end{definition}

\begin{definition}
A partial left module $\mathcal P$ over an operad $\mathcal O$ is the data of sub $\Sigma$-space $Comp \hookrightarrow \mathcal O \otimes \mathcal P$ and a map $p:Comp \m \mathcal P$ such that $$(m \otimes id)^{-1}(Comp)=(id \otimes p)^{-1}(Comp) \subset \mathcal O \otimes \mathcal O \otimes \mathcal P$$ and the diagrams analogous to the ones above commute (replace $ \mathcal O \otimes \mathcal O \otimes \mathcal L$ with $(m \otimes id)^{-1}(Comp)$).

\end{definition}

The name $Comp$ is short for the space of composable elements and operations. We likewise define right modules and partial right modules over operads. There is a functor from spaces to $\Sigma$-spaces which sends a space $X$ to the $\Sigma$-space with: $$X(n) = \begin{cases}
 X, & \text{if }  n=0 \\
\varnothing & \text{if } n \neq 0
\end{cases}$$ We will ignore the distinction between a space and its image as a $\Sigma$-space.

\begin{definition}
For an operad $\mathcal{O}$, an $\mathcal{O}$-algebra is a space $A$ with the structure of a left $\mathcal{O}$-module.
\end{definition}

Note that if $X$ is a $\Sigma$-space and $Y$ is a space, then the formula for $X \otimes Y$ simplifies to:
$X \otimes Y=\bigsqcup_j X(j) \times_{\Sigma_j} Y^j$. One can build a right $D_n$-module out of embeddings of disks in a parallelizable $n$-manifold as follows. For oriented manifolds $M$ and $N$, let $Emb^+(M,N)$ denote the space of all smooth orientation preserving embeddings of $M$ into $N$.

\begin{definition}
Let $M$ be a parallelized $n$-manifold. Using the parallelization, the derivative of a map $f \in Emb^+(\D^n,M)$ at each point can be identified with a matrix in $GL^+_n(\R)$, the group of matrices with positive determinant. Let $P$ denote the path space of $\Map(\D^n,SL_n(\R))$ based at the map which sends every point in $\D^n$ to the identity matrix. Let $ev:P \m \Map(\D^n,SL_n(\R))$ be the evaluation at $1$ map and let $D: Emb(\bigsqcup_{i=1}^k \D^n,M) \m \Map(\D^n,SL_n(\R))^k$ be the map that records the derivatives of the embeddings at every point divided by the $n$th root of its determinant. Let $D(M)(k)$ denote pull back of the following diagram: $$
\begin{array}{ccccccccl}
D(M)(k) & \m & P^k   \\
\downarrow & & ev^k \downarrow  &   \\

Emb^+(\bigsqcup_{i=1}^k \D^n,M)  & \overset{D}{\m} & \Map(\D^n,SL_n(\R))^k.\\

\end{array}$$ 
\end{definition}

The spaces $D(M)(k)$ assemble to form a $\Sigma$-space denoted $D(M)$. The paths of matrices serve to make $D(M)$ homotopy equivalent to the space of configurations of ordered  distinct points in $M$. That is, the map which sends a collection of embeddings to the image of the centers of each disk induces a homotopy equivalence between $D(M)(k)$ and $ M^k-\Delta_{fat}$ where $\Delta_{fat}$ is the fat diagonal. The space $D(M)$ has the structure of a right $D_n$-module as follows. The embeddings transform by composition of embeddings and the paths of matrices transform by restriction to smaller disks. The reason we divide by the $n$th root of the determinate is to remove the dependence of the derivative of $e:\D^n \m M$ on dilations of the domain\footnote{In \cite{Mi2} a different but homotopic definition of $D(M)(k)$ was given. However, in \cite{Mi2} the description of the the right $D_n$-module structure for $D(M)$ has an error in it. The definition used in this paper corrects that error.}.

\subsection{The two sided bar construction}\label{subsectwosided}

In this section, we will describe a definition of topological chiral homology due to Andrade in \cite{An}. It is based on the two sided bar construction introduced by May in \cite{M}. The purpose of this subsection is to describe May's two-sided bar construction. From now on, we will also assume that all operads $\mathcal O$ have $\mathcal O(0) =\{0\}$. All algebras that we will consider will have a base point $a_0 \in A$ and we require that $p(0) =a_0$. We likewise require that all right modules $\mathcal R$ have $\mathcal R(0) = \{r_0\}$. We will recall the functor on the category of based spaces ($Top_*$) associated to a right module. 

\begin{definition}
For $(X,a_0)$ a based space and $\mathcal R$ a right module over an operad $\mathcal O$, define a functor $R: Top_* \m Top_*$ by $ R X = \mathcal R \otimes X / \sim$. Here $\sim$ is the relation that if $r \in \mathcal R$ and $x_i \in X$, then $(r;x_1, \ldots, a_0, \ldots, x_n) \sim (r';x_1, \ldots,  x_n)$ with $r'$ the composition of $r$ with $(id,\ldots, id,0, id, \ldots, id)$.
\end{definition}

We will use the convention that calligraphy style fonts are for symmetric sequences and regular fonts are for the associated functor. 

\begin{definition}
For an operad $\mathcal O$, let $O : Top_* \m Top_*$ be the associated functor. We define an $\mathcal O$-functor (called right $\mathcal O$-functors elsewhere) to be a functor $R: Top_* \m Top_*$ and natural transformation $\xi : R O \m R$ making the following diagrams commute for every based space $X$:

$$
\begin{array}{ccccccccl}
RX  &\overset{R \eta }{\m} &ROX & & ROOX &\overset{R \mu}{\m} &ROX  \\
 & id \searrow & \xi \downarrow  &         &  \xi \downarrow   &  & \xi \downarrow   \\

  & & RX & &  ROX & \overset{\xi}{\m} & RX.\\

\end{array}$$ Here $\eta$ is the natural transformation induced by the unit of the operad and $\mu$ is the natural transformation induced by operad composition. 

\end{definition}

Note that right $\mathcal O$-modules induce $O$-functors. We now recall May's two-sided bar construction. We give a slight modification that works for partial algebras over an operad. Let $\mathcal P$ be a partial algebra over an operad $\mathcal O$. Let $Comp^k \subset \mathcal O ^{\otimes k} \otimes \mathcal P$ be the subspace of composable morphisms. Let $QComp^k$ be the image of $Comp^k$ in $O^k \mathcal P$. In other words, $QComp^k$ is the quotient of $Comp^k$ under base point relations. 

\begin{definition}
For $\mathcal O$ an operad, $\mathcal P$ a partial $\mathcal O$-algebra and $R$ an $\mathcal O$-functor, define $B_\bullet(R,\mathcal O,\mathcal P)$ to be the following simplicial space. The space of $k$ simplices of $B_\bullet(R,\mathcal O,\mathcal P)$ will be $R (QComp^k)$. The partial algebra composition map $Comp \m \mathcal P$ induces a map $QComp^k \m QComp^{k-1}$ and hence a map $d_0 : B_k(R, \mathcal O,\mathcal P) \m B_{k-1}(R, \mathcal O,\mathcal P)$. The operad composition map $\mathcal O \otimes \mathcal O \m \mathcal O$ gives $k-1$ maps $d_i : B_k(R, \mathcal O,\mathcal P) \m B_{k-1}(R, \mathcal O,\mathcal P)$ for $i=1, \ldots, k-1$ and the $\mathcal O$-functor composition map $R O \m R$ gives another map, $d_k : B_k(R, \mathcal O,\mathcal P) \m B_{k-1}(R, \mathcal O,\mathcal P)$. These maps will be the face map of $B_\bullet(R, \mathcal O,\mathcal P)$. The degeneracies, $s_i :  B_k(R, \mathcal O,\mathcal P) \m B_{k+1}(R, \mathcal O,\mathcal P)$ are induced by the unit of the operad $\iota \m \mathcal O$.

\end{definition}

See \cite{M} for a proof that these maps satisfy the axioms of face and degeneracy maps of a simplicial space. Let $B(R,\mathcal O,\mathcal P)$ denote the geometric realization. 

When $R$ is a functor coming from a right module $\mathcal R$, we now explain how to view elements $B_k(R,\mathcal O, \mathcal P)$ as decorated rooted directed trees. See Figure \ref{figure:RootedTree} for an illustration. One special vertex is called the root and  is labeled by an element of $\mathcal R(m)$ with $m$ being the valence of the root. The univalent verticies are called leaves and are labeled by elements of the partial $\mathcal O$-algebra $\mathcal P$. We say that a vertex is at level $k+1-i$ if it is $i$ edges away from the root. We require that all of the leaves are at level $0$. Internal verticies of valence $m+1$ are labeled by elements of $\mathcal O(m)$. The tree is directed and we require that the directions point away from the root. We order the edges exiting internal vertices and the root. Given a vertex with $m$ outgoing edges, the symetric group $\Sigma_m$ acts in two ways on the tree. It acts on the label since $\mathcal R(m)$ and $\mathcal O(m)$ have $\Sigma_m$ actions. It also acts by permuting the ordering of the $m$ outgoing labels. We mod out by the induced diagonal action. Given a decorated tree $T$ of this type, let $\{ T_i \}$ be the subtrees of $T$ obtained by deleting the root (see Figure \ref{figure:RootedTree}). These determine elements of $O^k \mathcal P$. We require that these elements are in fact in $Comp^k$. We do not label internal verticies by $0 \in \mathcal O(0)$. Indeed the base-point relation precisely states that we can ignore such ``dead ends'' which would have leaves at level $>0$. The face maps can be described by collapsing edges connecting two levels and composing the labels. The degeneracies can be described as inserting a level where all elements are labeled by the unit of the operad.

Given a partial $\mathcal O$-algebra $\mathcal P$, there is a natural way of completing $\mathcal P$ to get an $\mathcal O$-algebra.

\begin{definition}
Let $\mathcal O$ be an operad and $\mathcal P$ be a partial $\mathcal O$-algebra. We call $B(O,\mathcal O, \mathcal P)$ the $\mathcal O$-algebra completion of $\mathcal P$. 
\end{definition}

 Let $A$ be an $\mathcal O$-algebra. The arguments in \cite{M} showing that $B(O, \mathcal O, A)$ is a $\mathcal O$-algebra and that the natural $\mathcal O$-algebra map $B(O, \mathcal O, A) \m A$ is a homotopy equivalence immediately generalize to give the following propositions. 

\begin{proposition}
If $\mathcal P$ is a partial $\mathcal O$-algebra over an operad $\mathcal O$, $B(O,\mathcal O, \mathcal P)$ is an $\mathcal O$-algebra.
\end{proposition}

\begin{proposition}
If $\mathcal P$ is a partial $\mathcal O$-algebra over an operad $\mathcal O$ and $R$ is an $\mathcal O$-functor, then the natural map $B(R,\mathcal O, B(O,\mathcal O, \mathcal P)) \m B(R,\mathcal O, \mathcal P) $ is a weak equivalence.
\label{extradeg}
\end{proposition}

\begin{figure}[!ht]
\begin{center}\scalebox{.4}{\includegraphics{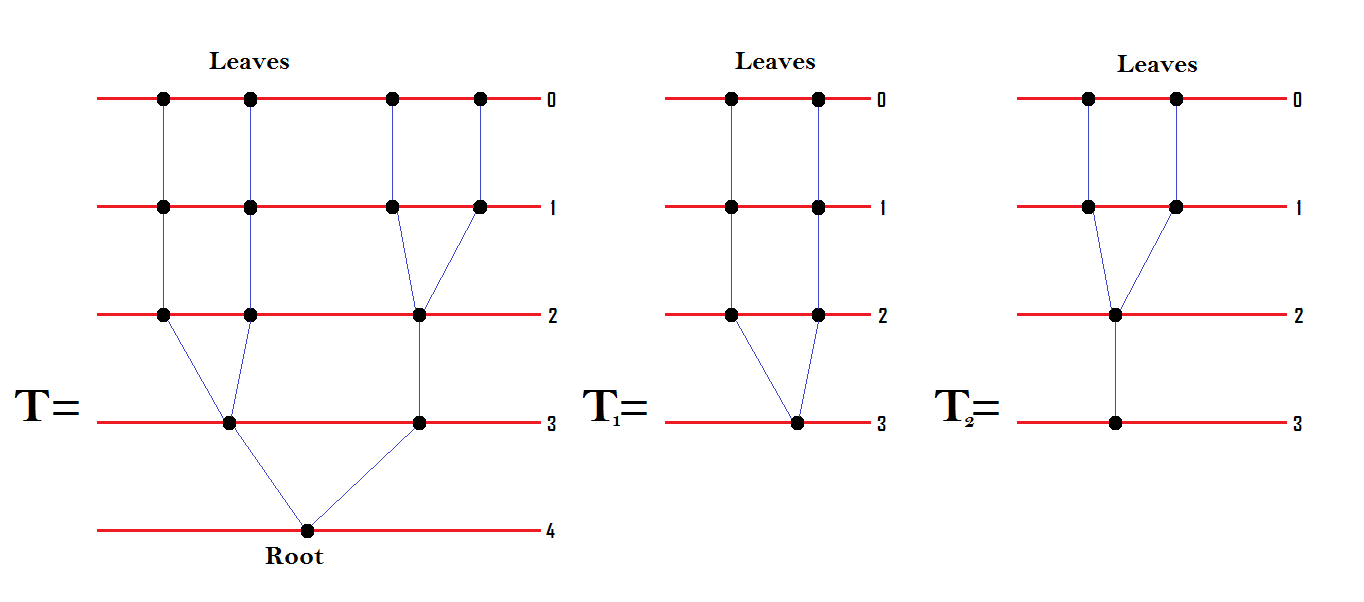}}\end{center}

\caption{A rooted tree and subtrees}
\label{figure:RootedTree}
\end{figure}

\begin{remark}
\label{reedy}
In later sections, we will need to compare simplicial spaces. It is not true that a map of simplicial spaces that induces a weak homotopy equivalence on each simplical level will necessarily induce a weak homotopy equivalence between the geometric realizations. A well known cofibrancy condition called ``properness'' or ``Reedy cofibrancy'' gives a sufficient condition for level-wise weak equivalences to realize to weak equivalences. Using the techniques of Appendix A of \cite{M} (also see Section 2.4 of \cite{Mi2}), it is not hard to show that all simplicial spaces considered in this paper are proper. Essentially the only thing to check is that the inclusion of basepoints are cofibrations. Hence we will be able to deduce that a map of simplicial spaces induces a weak homotopy equivalence if it does level-wise.

\end{remark}

\subsection{Topological chiral homology}

In this subsection we describe a model of topological chiral homology introduced by Andrade in \cite{An} based on May's two-sided bar construction. We also describe a natural map from topological chiral homology to a mapping space called the scanning map and discuss its properties. See \cite{Mi2} for a more detailed discussion of the topics of this subsection.

\begin{definition}
For a partial $D_n$-algebra $\mathcal P$ and parallelizable $n$-manifold $M$, we define the topological chiral homology of $M$ with coefficients in $\mathcal P$ to be $B(D(M),D_n,\mathcal P)$. We denote this by $\int_M \mathcal P$. 
\end{definition}

In this subsection, we only consider $D_n$-algebras as opposed to partial $D_n$-algebras. The goal of the remainder of the subsection is to define a map $s: \int_M A \m \Map^c(M,B^n A)$ and recall its properties. This map is called the scanning map $s$. We use the following model of $B^n A$ introduced in \cite{M}. Recall that the functor $Y \m \Sigma^n Y$ is a $D_n$-functor.

\begin{definition}
For a $D_n$-algebra $A$, let $B^n A =B(\Sigma^n,D_n,A)$.
\end{definition}

If $A$ is a $D_n$-algebra, then $\pi_0(A)$ has a natural monoid structure induced by the map $D_n(2) \times A^2 \m A$.

\begin{definition}
A $D_n$-algebra is called group-like if $\pi_0(A)$ is a group.
\end{definition}

The above definition of $B^n A$ is reasonable because of the following theorem due to May in \cite{M}.

\begin{theorem}
If $A$ is a group-like $D_n$-algebra, then $\Omega^n B^n A$ is weakly homotopy equivalent to $A$. 
\end{theorem}

\begin{definition}
For an $n$-manifold $M$, let $\Map^c(M, \Sigma^n \cdot)$ be the functor which sends a based space $Y$ to the space of compactly supported maps from $M$ to the $n$-fold suspension of $Y$.
\end{definition}

This is a $D_n$-functor because the functor $Y \m \Sigma^n Y$ is a $D_n$-functor \cite{M}. For $M$ a parallelized $n$-manifold, there is a scanning natural transformation of $D_n$-functors $s : D(M) \m \Map^c(M,\Sigma^n \cdot )$ defined as follows. For a based space $Y$, we define a map $D(M)Y \m \Map^c(M,\Simga^n Y)$. Map a collection of embeddings $e_i: \D^n \m M$ to the function $f:M \m \Sigma^n Y$ defined as follows. Let $f$ be constant outside of the images of the embeddings $e_i$. For $m \in im(e_i)$, the  $(e_i^{-1}(m),y_i)$ defines a point inside $\Simga^n Y$. Define $f(m)$ to be that point. Here we view $\Sigma^n Y$ as $Y$ smashed with the one point compactification of $\D^n$. This is a natural transformation of $D_n$-functors. 

The scanning natural transformation $s$ induces the scanning map $s:B(D(M),D_n,A) \m \Map^c(M,B^n A)$ as follows. The natural transformation induces a map $B(D(M),D_n,A) \m B(\Map^c(M,\Sigma^n \cdot ),D_n,A)$. There is a natural map $B(\Map^c(M,\Sigma^n \cdot),D_n,A) \m  \Map^c(M,B( \Sigma^n \cdot ,D_n,A))$ (for example see Section 12 of \cite{M}). Combining these, we get the scanning map.

\begin{definition}
Let $A$ be a $D_n$-algebra and $M$ a parallelized manifold. Let $s: \int_M A \m \Map^c(M,B^n A)$ be the composition of: $$ \int_M A = B(D(M),D_n,A) \overset{s}{\m} B(\Map^c(M,\Sigma^n \cdot ),D_n,A)$$ $$ \m  \Map^c(M, B(\Sigma^n,D_n,A))= \Map^c(M,B^n A).$$
\end{definition}

In \cite{Lu5} (Theorem 3.8.6), Lurie proves the following theorem.

\begin{theorem}
If $A$ is a group-like $D_n$-algebra, then $s: \int_M A \m \Map^c(M,B^n A)$ is a homotopy equivalence.
\end{theorem}

For our purposes, we will need properties of the scanning map when $A$ is not group-like. In this case, the scanning map is not a homotopy equivalence. However, when $M$ is open, the scanning map is what we call a stable homology equivalence \cite{Mi2}. The word stable is not in the sense of stable homotopy theory but instead means after inverting the following types of maps which we call stabilization maps.

For simplicity, assume that $M$ is a connected manifold which is the interior of a manifold with connected boundary $\partial M$. Fix an embedding $e: \D^n \m (0,1) \times \partial M$ and diffeomorphism $d: M \cup_{\partial M} [0,1) \times \partial M \m M$ such that $d^{-1}$ is isotopic to the inclusion $M \hookrightarrow  M \cup_{\partial M} [0,1) \times \partial M$. For $a \in A$, we shall define a stabilization map $t_a : B(D(M),D_n,A) \m B(D(M),D_n,A)$ as follows. The pair $(e;a)$ defines an element of $D(\partial M \times (0,1))A$. Let $a^k \in D(\partial M \times (0,1)) D_n^k A$ be the image of $(e;a)$  under $k$ degeneracy maps. Sending a configuration of disks to one union $a^k$ gives a map $B_k(D(M),D_n,A) \m B_k(D(M \cup_{\partial M} (0,1) \times \partial M),D_n,A)$. The diffeomorphism $d: M \cup_{\partial M} [0,1) \times \partial M \m M$ induces a map $d_k : B_k( M \cup_{\partial M} [0,1) \times \partial M , D_n, A) \m B_k(D(M),D_n,A) $. Let $t^a_k$ be the composition of these two maps.

\begin{definition}
Let $t_a : \int_M A \m \int_M A$ be the geometric realization of $t^a_\bullet : B_\bullet(D(M),D_n,A) \m  B_\bullet(D(M),D_n,A).$
\label{stab}
\end{definition}

For $a \in A$, let $f_a : (0,1) \times \partial M \m B^n A$ be the image of $(e,a)$ under the scanning map.

\begin{definition}
For $m \in M$ and $f \in \Map^c(M,B^n A)$, let $T_a : \Map^c(M,B^n A) \m \Map^c(M,B^n A)$ be defined by the following formula with:
$$ T_a(f)(m)=
\begin{cases}
f(d^{-1}(m))  & \mbox{if } d^{-1}(m) \in M \\
f_a(d^{-1}(m))  & \mbox{if } d^{-1}(m) \notin M . \\
\end{cases}
$$
\end{definition}

In \cite{Mi2}, the author proved the following theorem regarding the scanning after inverting the maps $t_a$ and $T_a$.

\begin{theorem}
Let $A$ be a $D_n$-algebra and let $M$ be a parallelized $n$-manifold with $\pi_0 (\partial M) \m \pi_0(M)$ onto. Let $a_i$ be representatives of generators of $\pi_0(A)$. The scanning map induces a homology equivalence $s:hocolim_{t_{a_i}} \int_M A \m hocolim_{T_{a_i}} \Map^c(M,B^n A)$.
\label{scanTCH}
\end{theorem}

Note that the maps $T_a :  \Map^c(M,B^n A) \m  \Map^c(M,B^n A) $ are homotopy equivalences. Thus $hocolim_{T_{a_i}} \Map^c(M,B^n A) \simeq \Map^c(M,B^n A)$. Note that if $\pi_0(A) = \N$ and $M$ is connected, then $\pi_0( \int_M A) =\N$ and $\pi_0(\Map^c(M,B^n A)) =\Z$. In this situation, we denote the $k$th connected component of these spaces by $(\int_M A)_k$ and $\Map_k^*(M,B^n A)$ respectively. Informally, the above theorem can be viewed as saying that the homology of $(\int_M A)_k$ converges to the homology of a component of $\Map^*(M,B^n A))$ in the limit as $k$ tends to infinity.




\section{The partial $D_2$-algebra $R_2$}\label{KY}

The goal of this section is to define a partial $D_2$-algebra called $R_2$ and describe the homotopy type of its topological chiral homology. The desired properties of $R_2$ are that $B^2 R_2 \simeq \C P^2$ and that $R_2$ admits a map of partial $D_2$-algebras to $\Omega^2 \p$ that lands in a subspace of ``approximately'' $J$-holomorphic maps. This map will be described in Section \ref{secR2J}. The subscript $2$ indicates that $R_2$ will map to the components of degree less than or equal to $2$ and the $R$ stands for rational. We begin by recalling some theorems of Kallel from \cite{K2} and Yamaguchi from \cite{Y} regarding configuration space models for $\Map^*(\Sigma_g,\C P^m)$.

\subsection{A configuration space model for $\Map^*(\Sigma_g,\C P^2)$}

The scanning map of topological chiral homology was partially inspired by similar constructions for various configuration spaces. We are interested in the following construction which gives a configuration space model of the space of maps from a surface into a complex projective space. We only consider the case relevant for $\p$ even though basically everything in this section works for $\C P^m$ for all $m$. Let $\D$ be the open unit disk in $\C$ and let $\bar \D$ denote its closure. 

\begin{definition}
For a space $X$, let $Sym^k_{\leq 2}(X)$ denote the subspace of the $k$-fold symmetric product of $X$ of configurations where at most two points occupy the same point in $X$. Let $Sym_{\leq 2}(X) = \bigsqcup\limits_{k=0}^\infty Sym^k_{\leq 2}(X)$ and let $A_2=\bigsqcup_{k \leq 2} Sym^k_{\leq 2}(\D)$. The space $Sym_{\leq 2}(\D)$ has a natural $D_2$-algebra structure and this induces a partial  $D_2$-algebra structure on $A_2$.
 \end{definition}

Kallel proved the following theorem (Theorem 1.7 of \cite{K2}).

\begin{theorem}
Let $S$ be a connected parallelizable surface admitting boundary. There exist stabilization maps $t:Sym_{\leq 2}(S) \m Sym_{\leq 2}(S)$ and a scanning map  $s:Sym_{\leq 2}(S) \m \Map^c(S,\p)$ such that $s$ induces a homology equivalence between $hocolim_t Sym_{\leq 2}(S)$ and $\Map^c(S,\p)$.
\end{theorem}

The stabilization maps are constructed in an analogous manner to those of Definition \ref{stab}. The above theorem was improved by Yamaguchi (Theorem 1.3 and Theorem 3.1 of \cite{Y}).

\begin{theorem} \label{Ytheorem}
Let $S$ be a connected parallelizable surface admitting boundary. The stabilization map $t: Sym^k_{\leq 2}(S) \m  Sym^{k+1}_{\leq 2}(S)$ and scanning map $s: Sym^k_{\leq 2}(S) \m  \Map^{c}_k(S,\p)$ induce isomorphisms on homology groups $H_i$ for $i \leq k/3$.
\end{theorem}

\begin{remark}
It seems to be unknown if the above isomorphism on homology groups can be upgraded to an isomorphism on homotopy groups. In \cite{KS}, it is proven that $s$ and $t$ induce isomorphism on fundamental groups. This is not enough to conclude that $s$ and $t$ are highly connected; for that one would need to know that these maps induce homology equivalences with coefficients in the local coefficient systems given by the higher homotopy groups. In \cite{Se}, Segal does this in the case of $\Hol^*(\s,\s)$ by proving that the fundamental group acts trivially on higher homotopy groups so the relevant local coefficients are in fact untwisted.
\end{remark}

Implicit in the work of Kallel and Yamaguchi is the following proposition. It is well known but a proof does not seem to appear in the literature so we sketch one here. 

\begin{proposition}
The $2$-fold delooping of $Sym_{\leq 2}(\D)$ is homotopy equivalent to $\p$. 

\label{deloop}
\end{proposition}

\begin{proof}
It is well known that a model for $B^2 Sym_{\leq 2}(\D)$ is given by $Sym_{\leq 2}(\bar \D, \partial \bar \D)$ (see Page 2 of \cite{Sa}). Here $Sym_{\leq 2}(\bar \D, \partial \bar \D)$ is the quotient of the space $Sym_{\leq 2}(\bar \D)$ by the relation that two configurations are identified if they agree on the open disk $\D$. There is a natural filtration $F_k(Sym_{\leq 2}(\bar \D, \partial \bar \D))$ defined as the image of the natural map $Sym^k_{\leq 2}(\bar \D) \m Sym_{\leq 2}(\bar \D, \partial \bar \D)$. Observe that $F_2(Sym_{\leq 2}(\bar \D, \partial \bar \D))$ is homeomorphic to $\p$. An argument similar to that used in the proof of Proposition 3.1 of \cite{Se} shows that $Sym_{\leq 2}(\bar \D, \partial \bar \D)$ deformation retracts onto $F_2(Sym_{\leq 2}(\bar \D,\partial \bar \D))$. This deformation retraction is constructed by radial expansion around the origin until all but at most two particles in the configuration have been pushed into boundary.

\end{proof}

\subsection{Definition of $R_2$ and map to bounded symmetric products}

In this subsection, we define the partial $D_2$-algebra $R_2$ and compare it to bounded symmetric products. 

\begin{definition}
Let $R_2=\{r_0,r_1\} \cup Cone(D_2(2)/{\Sigma_2})$. The space of composable element in $D_2(2) \otimes  R_2$ is $Comp=\{0\} \cup \{id\} \times R_2 \cup D_2(2) \times_{\Sigma_2} \{r_1\}^2$ with $0 \in D_2(0)$ and $id \in D_2(1)$. The map $p:Comp \m R_2$ is defined as follows. For $\kappa \in R_2$, define $p(id,\kappa)=\kappa$. For $\vec e \in D_2(2)/{\Sigma_2}$, let $p(\vec e;r_1,r_1)=\vec e$ viewed as an element of $D_2(2)/{\Sigma_2}  \subset Cone(D_2(2)/\Sigma_2)$.
\label{defR2}

 \end{definition}

The partial abelian monoid $\{0,1,2\}$ with addition is too small to admit a map of partial $D_2$-algebras  to $\Omega^2 \p$. However, we can replace it with a homotopic partial $D_2$-algebra ($R_2$) which does map. This is the purpose of $R_2$. When you glue two degree one rational maps to $\p$ using an element of $D_2(2)$, the element of $D_2(2)$ matters. However, it does not matter up to homotopy so we can extend this map to the cone.

We now define a map of partial $D_2$-algebras $a:R_2 \m A_2$. Let $a(r_0)$ be the empty configuration. Pick any configuration $a_1 \in Sym^1_{\leq 2}(\D)$ and let $a(r_1)=a_1$. The partial $D_2$-algebra structure of $A_2$ gives a natural map $D_2(2)/{\Sigma_2} \m Sym^2_{\leq 2}(\D)$. Since $Sym^2_{\leq 2}(\D)$ is contractible, this extends to the cone and gives a map of partial $D_2$-algebras $a:R_2 \m A_2$. This map is clearly a homotopy equivalence. Moreover it induces homotopy equivalences between the spaces of composable elements and higher spaces of composable elements. Thus $a: B(D(S),D_2,R_2) \m B(D(S),D_2,A_2)$ is a weak homotopy equivalence for any parallelizeable surface $S$ (see Remark \ref{reedy}). 

There is a natural map $\epsilon : D(S) A_2 \m Sym_{\leq 2}(S)$ constructed by using the embeddings of disks into the surface to map points in the disk to points in the surface. In Figure \ref{figure:Epsilon}, the left-hand side depicts an element of $D(S)A_2$ and the right-hand side row depicts its image under $\epsilon$ in $Sym_{\leq 2}(S)$. It is easy to see that $\epsilon$ is an augmentation of the simplicial space $B_{\bullet}(D(S),D_2,A_2)$ and so it induces a map $\epsilon : B(D(S),D_2, A_2)  \m Sym_{\leq 2}(S)$. The goal of the next section is to prove that $\epsilon$ is a weak homotopy equivalence.

\begin{figure}[!ht]
\begin{center}\scalebox{.2}{\includegraphics{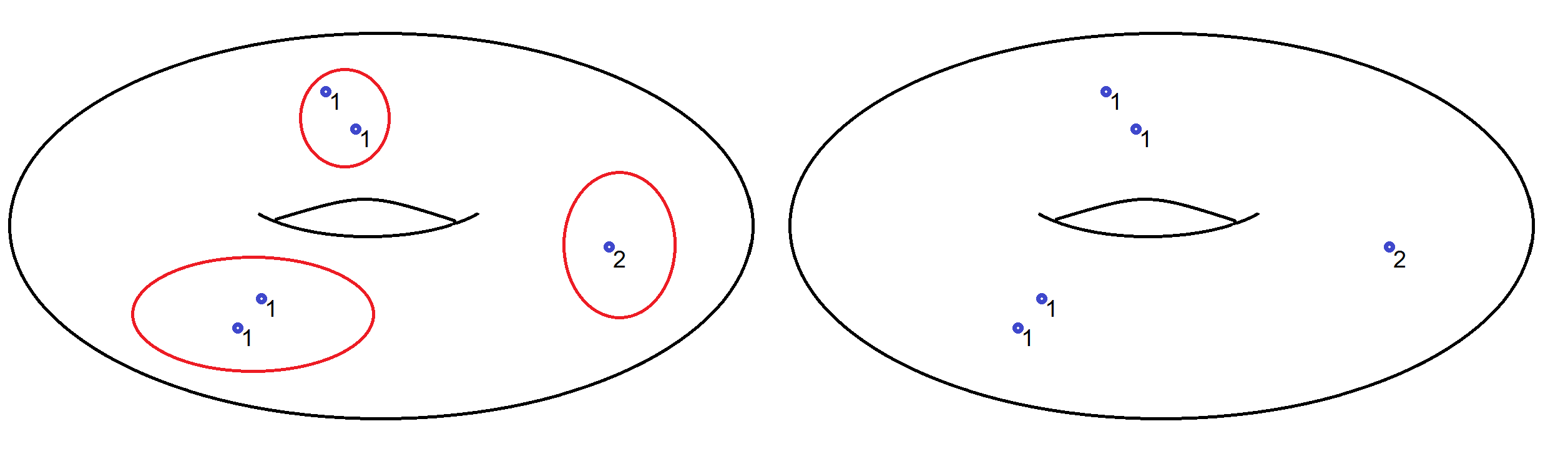}}\end{center}
\caption{The map $\epsilon$}
\label{figure:Epsilon}
\end{figure}

\subsection{Micro-fibrations}

In order to prove that the map $\epsilon : B(D(S),D_2, A_2)  \m Sym_{\leq 2}(S)$ is a weak homotopy equivalence, we first recall the notion of microfibration.

\begin{definition}
A map $\pi: E \m B$ is called a microfibration if the following condition holds. For every one parameter family of maps  $H_t: \bar \D^n \m B$ and lift $\tilde H_0 : \bar \D^n \m E$ of $H_0$, there exists a $t_0>0$ such that we can continuously lift $H_t$ for $t \in [0,t_0)$. 
\end{definition}

This definition is relevant because of the following theorem which follows from Lemma 2.2 of \cite{W}.

\begin{theorem}
Let $\pi: E \m B$ be be a microfibration with weakly contractible fibers. Then $\pi$ is a Serre fibration and hence a weak homotopy equivalence.
\label{microW}
\end{theorem}

We now prove that the map $\epsilon : B(D(S),D_2, A_2)  \m Sym_{\leq 2}(S)$ defined in the previous section is a microfibration. This will allow us to prove that it is a homotopy equivalence by proving that its fibers are weakly contractible. 

\begin{proposition}
The map $\epsilon : B(D(S),D_2, A_2)  \m Sym_{\leq 2}(S)$ is a microfibration.
\label{ismicro}
\end{proposition}

\begin{proof}
Let $H_t: \bar \D^n \m Sym_{\leq 2}(S)$ be a one parameter family of maps and $\tilde H_0 : \bar \D^n \m B(D(S),D_2, A_2)$ be a lift of $H_0$. To define the lifts $\tilde H_t$, we need to specify a collection of embeddings and a collection of elements of $Sym_{\leq 2}(\D)$.  Let $pt$ denote the one point space viewed as $D_2$ algebra and let $E : B(D(S),D_2,A_2) \m B(D(S),D_2,pt)$ be the map which forgets the labels but remembers the embeddings. Given an element $\kappa \in B(D(S),D_2,pt)$ 
and a configuration $\xi \in Sym_{\leq 2}(S)$, there may not exist $\xi_1,\xi_2, \ldots \in A_2$ such that  $(\kappa;\xi_1, \xi_2 \ldots) \in B(D(S),D_2,A_2)$ is mapped to $\xi$ by $\epsilon$. However, if there do exist $\xi_1,\xi_2, \ldots \in A_2$, they will be unique. In this situation, we say that $\kappa$ and $\xi$ are compatible and define $L(\kappa; \xi)=(\kappa; \xi_1, \xi_2, \ldots)$.

Note that if $\xi$ and $\kappa$ are compatible, so will $\xi'$ and $\kappa$ if $\xi'$ is sufficiently close to $\xi$. For $a \in \bar \D^n$, we define $\tilde H_t(a)=L(E(\tilde H_0(a));H_t(a))$. In other words, we use the embedded disks and simplicial coordinates associated to $\tilde H_0(a)$ and the elements of the symmetric product associated to $H_t(a)$. For $t$ sufficiently small, $L(E(\tilde H_0(a)))$ and $H_t(a)$ will be compatible since $L(E(\tilde H_0(a)))$ and $H_0(a)$ are compatible. By compactness of $\bar \D^n$, we can find a $t_0>0$ such that $\tilde H_t$ is defined for all $t<t_0$.

\end{proof}

This lift is also depicted in Figure \ref{figure:Lift}. On the bottom row is a one parameter family of points in $Sym_{\leq 2}(S)$. This leftmost edge is at time zero and the top left depicts a lift. The middle row depicts a short time in the future where the disks from the upper left are still compatible. The rightmost edge depicts a time well in the future where the disks are no longer compatible with the configuration. 

\begin{figure}[!ht]
\begin{center}\scalebox{.15}{\includegraphics{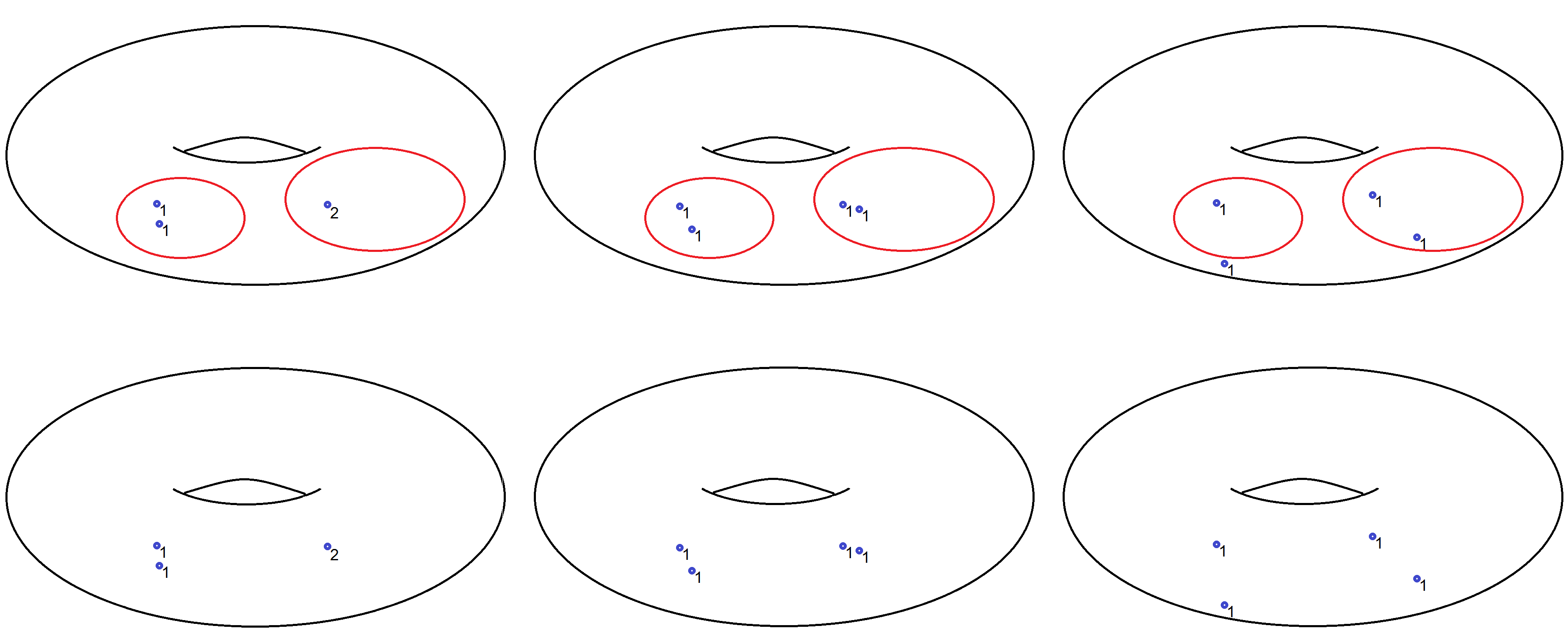}}\end{center}
\caption{A one parameter family and a partial lift}
\label{figure:Lift}
\end{figure}

We now prove that the fibers of $\epsilon$ are contractible. By Theorem \ref{microW} and  Proposition \ref{ismicro}, this will imply that $\epsilon$ is a homotopy equivalence.

\begin{proposition}
The fibers of $\epsilon : B(D(S),D_2, A_2)  \m Sym_{\leq 2}(S)$ are weakly contractible.
\end{proposition}

\begin{proof}

Fix a configuration $\xi \in Sym_{\leq 2}(S)$. The space $\epsilon^{-1}(\xi)$ is naturally the geometric realization of a subsimplicial space of $B_{\bullet}(D(S),D_2, A_2)$. Call this simplicial space $X(\xi)$. This simplicial space can be described as a space of nested disks such that the points in the configuration $\xi$ lie in innermost disks. Let $(f_1,f_2, \ldots) = \vec f \in D(S)$ be a collection of embeddings such that the image of each $f_i$ contains exactly one point (not counted with multiplicity) of $\xi$ and every point of $\xi$ is in the image of one of the embeddings. Let $Y(\vec f)$ denote the subsimplicial space of $X(\xi)$ where the images of the maps $f_i$ are contained in innermost disks and the paths of matrices are compatible. Let $\vec \xi =(\xi_1, \xi_2, \ldots) $ be the elements of $A_2$ associated to the pair $\vec f$ and $\xi$. More precisely but possibly less clearly, a point $\alpha \in X(\xi)_k$ is in  $Y(\vec f)_k$ if there exists $\beta \in X(\xi)_{k+1}$ with $d_0(\beta)=\alpha$ and $d_{k+1}(d_{k}(\ldots d_1(\beta)\ldots)) = (\vec f; \vec \xi)$. Since the space of disks containing some particular disk is homotopy equivalent to the space of disks containing some particular point, the inclusion $Y(\vec f)_k \m X(\xi)_k$ is a homotopy equivalence and so $|Y(\vec f)_\bullet| \m |X(\xi)_\bullet|$ is a weak homotopy equivalence. 

Therefore it suffices to prove that $|Y(\vec f)_\bullet|$ is contractible to prove the theorem. Define $Y_{-1}(\vec f)$ to be the set containing $ \xi$. We can view $Y_{-1}(\vec f)$ as an agumentation of $Y(\vec f)_\bullet$. Let $D:Y(\vec f)_k \m Y(\vec f)_{k+1}$ be the map which inserts the embeddings $\vec f$. That is, $D(\alpha)=\beta$ as above and $D(\xi)= (\vec f; \vec \xi)$. The map $D$ is an extra degeneracy of the augmented simplicial space $Y_{\bullet}(\vec f)$. Thus $|Y_{\bullet}(\vec f)| \simeq Y_{-1}(\vec f)$. Since, $Y_{-1}(\vec f)$ is a point, the claim follows.  

\end{proof}

This shows that $\int_S A_2 =B(D(S),D_2,A_2)$ is homotopy equivalent to $Sym_{\leq 2}(S)$. Using a similar proof, we get the following proposition.

\begin{proposition}
The $D_2$-algebra completion of $A_2$, $ B(D_2,D_2, A_2)$, is homotopy equivalent to $Sym_{\leq 2}(\D)$. Hence, $ B(\Sigma^2,D_2, A_2) \simeq \p$.
\end{proposition}

\begin{remark}
Using ideas similar to those of this subsection, one can prove that Salvatore's model of topological chiral homology from \cite{Sa} agrees with the one used in this paper.
\end{remark}

\subsection{Conclusions about scanning maps}

Using Kallel's scanning map, we can define a map  $(\int_S R_2 )_k \m \Map_k^c(\s,\p)$ inducing a homology equivalence in the range $* \leq k/3$. In this section we show that another map also induces such an equivalence. This other map will be more convenient later in the paper when we consider gluing $J$-holomorphic curves. Picking a map of partial $D_2$-algebras $h: R_2 \m \Omega^2 \p$ induces a map $h':(\int_S R_2 )_k \m (\int_S \Omega^2 \p )_k$. In this subsection we prove that the composition of $h'$ with $s: (\int_S \Omega^2 \p )_k \m  \Map^c_k(S,\p)$ is a homology equivalence in the above range under the assumption that $h(r_1) \in \Omega_1^2 \p$. Technically, the scanning map of Section \ref{secTCH} is a map $s: \int_S \Omega^2 \p  \m \Map^c(S,B^2 \Omega^2 \p)$ but there is a natural homotopy equivalence $B^2 \Omega^2 \p \m \p$ and so we also denote the induced map $\int_S \Omega^2 \p \m  \Map^c(S,\p)$ by $s$. This map $\int_S \Omega^2 \p \m \Map^c(S,\p)$ agrees with the one induced by the obvious augmentation $D(S) \Omega^2 \p \m \Map^c(S,\p)$. Note that such a map $h$ exists by a similar argument to the one used to construct the map $R_2 \m A_2$.

\begin{lemma}
\label{lemUnique}
Suppose $h_0,h_1:R_2 \m \Omega^2  \p$  are maps of partial $D_2$-algebras with $h_0(r_1)$ and $h_1(r_1)$ in $\Omega^2_1 \p$. Then $h_1$ and $h_2$ are homotopic through maps of partial $D_2$-algebras.
\end{lemma}

\begin{proof}
We will define a family of maps of partial $D_2$-algebras $h_t:R_2 \m \Omega^2  \p$ interpolating between $h_0$ and $h_1$. By considering the Hopf fibration, one sees that $\pi_i(\Omega^2_k  \p)=\pi_{i+2}(S^5)$. Since $\Omega^2_1  \p$ is connected, we can pick a path from $h_0(r_1)$ to $h_1(r_1)$. Use this path to define $h_t(r_1)$. The partial algebra structure forces the definition of $h_t$ on $D_2(2)/\Sigma_2$. Note this agrees with $h_0$ and $h_1$ since they are also maps of partial $D_2$-algebras.  The maps $h_t$ restricted to $D_2/\Sigma_2$ as well as $h_0$ and $h_1$ restricted to $Cone(D_2/\Sigma_2)$ assemble to form a map: $$H:[0,1] \times D_2(2)/\Sigma_2 \cup \{0,1\} \times Cone(D_2(2)/\Sigma_2) \m \Omega^2_2 \p.$$ Completing the definition of $h_t$ amounts to extending this to a map $$H:[0,1] \times Cone(D_2(2)/\Sigma_2) \m \Omega^2_2 \p.$$ Since $ [0,1] \times D_2(2)/\Sigma_2 \cup \{0,1\} \times Cone(D_2(2)/\Sigma_2) \simeq S^2$ and $\pi_2(\Omega^2_k  \p)=\pi_4( S^5)=0$, this is possible.
\end{proof}

\begin{lemma}
Let $S$ be a connected parallelizable surface admitting boundary. The scanning map (of Section \ref{secTCH}) $s:(\int_S R_2)_k \m  \Map^c_k(S,B^2 R_2)$ induces isomorphisms on homology groups $H_i$ for $i \leq k/3$.
\end{lemma}

\begin{proof}

Although Kallel and Yamaguchi's scanning map is not obviously comparable to the scanning map of Section \ref{secTCH}, their stabilization map is very similar to the stabilization map considered in Section \ref{secTCH}. In fact Kallel and Yamaguchi's stabilization map can be defined as the unique map making the following diagram commute: $$
\begin{array}{cccccccl}
(\int_S A_2)_k    &\overset{t}{\m} & (\int_S A_2)_{k+1}               \\
\downarrow \epsilon & & \downarrow \epsilon                \\

Sym^k_{\leq 2}(S) &\overset{t}{\m} &Sym^{k+1}_{\leq 2}(S). 
\end{array}$$ 

Thus $t: (\int_S A_2)_k \m (\int_S A_2)_{k+1} $ and hence $t: (\int_S R_2)_k \m (\int_S R_2)_{k+1}$  induce isomorphisms on homology groups $H_i$ for $i \leq k/3$. From this we can conclude that the natural inclusions $(\int_S R_2)_{k} \m hocolim_t (\int_S R_2)_{j}$  induce isomorphisms on homology groups $H_i$ for $i \leq k/3$. Since the stabilization maps for the spaces of maps are homotopy equivalences, the natural inclusion $\Map^c_k(S,B^2 R_2) \m hocolim_T \Map^c_j(S,B^2 R_2)$ is a weak homotopy equivalence. Consider the following commuting diagram: $$
\begin{array}{cccccccl}
(\int_S R_2)_k    &\m &  hocolim_t  (\int_S R_2)_{j}          \\
\downarrow s & & \downarrow s                \\

\Map^c_k(S,B^2 R_2)  & \m & hocolim_T \Map^c_j(S,B^2 R_2).
\end{array}$$  By Theorem \ref{scanTCH}, $s:  hocolim_t  (\int_S R_2)_{j} \m hocolim_T \Map^c_j(S,B^2 R_2)$ is a homology equivalence. Thus, $s:(\int_S R_2)_k \m  \Map^c_k(S,B^2 R_2)$ induces isomorphisms on homology groups $H_i$ for $i \leq k/3$.

\end{proof}

We now deduce the main result of this section.

\begin{corollary}\label{KYpackage} 
Let $S$ be a connected parallelizable surface admitting boundary and $h: R_2 \m \Omega^2 \p$ be a map of partial $D_2$-algebras with $h(r_1) \in \Omega^2_1 \p$. Let $h':(\int_S R_2 )_k \m (\int_S \Omega^2 \p )_k$ be the induced map. The composition $s \circ h': (\int_S R_2 )_k \m \Map^c_k(S,\p)$ induce isomorphisms on homology groups $H_i$ for $i \leq k/3$.

\end{corollary}

\begin{proof}
Let $g:B^2 R_2 \m \p$ be any homotopy equivalence such that the map on $H_2$ extends the map $\pi_0(R_2) \m \pi_0(\Omega^2 \p)$. Such a map exists by Proposition \ref{deloop}. Let $\hat g:\Omega^2 B^2 R_2 \m \Omega^2 \p$ be the induced map of iterated loop spaces and $g'':\Map^c(S,B^2 R_2) \m \Map^c(S,\p)$ be the induced map on the space of compactly supported maps. Note that these maps are weak homotopy equivalences. Therefore $g'' \circ s: (\int_S R_2 )_k \m \Map^c_k(S,\p)$ induces isomorphisms on homology groups $H_i$ for $i \leq k/3$. The scanning map gives a map of partial $D_2$-algebras $\hat s: R_2 \m \Omega^2 B^2 R_2$. By Lemma \ref{lemUnique}, $\hat g \circ \hat s$ is homotopic as maps of partial algebras to $h$. Therefore, they induce homotopic maps on topological chiral homology and so the following diagram homotopy commutes: $$
\begin{array}{cccccccl}
(\int_S R_2)_k    &\overset{s }{\m} &(\int_S \Omega^2 B^2 R_2)_{k}               \\
\downarrow h' & & \downarrow g''               \\

(\int_S \Omega^2 \p)_{k}  &\overset{s}{\m} &\Map_k^c(S,\p). 
\end{array}$$ Since traversing the diagram clockwise induces a homology equivalence in a range, so does traversing the diagram counterclockwise.

\end{proof}

\section{Orbifolds}\label{secorb}

In this section, we will review facts about (smooth) orbifolds which will be needed in later sections since the moduli space of $J$-holomorphic curves can have orbifold singularities.  See \cite{KL} or \cite{BK} for more information. 

\subsection{Definitions}

\begin{definition}
An $n$-dimensional orbifold atlas on a paracompact Hausdorff space $X$ is a collection of charts $\mathcal U = \{ \tilde U_i,G_i,\phi_i\}$. Here $G_i$ is a finite group, $\U_i$ is an open subset of a $n$-dimensional $G_i$-representation that contains the origin, and $\phi_i:\tilde U_i \m X$ is a $G_i$ invariant smooth map that descends to a local homeomorphism, $\U_i/G_i \m X$. We also require that $\{U_i/G_i\}$ covers $X$ and the following condition regarding overlaps. An injection $\lambda: (\tilde U_i,G_i,\phi_i) \m (\tilde U_j,G_j,\phi_j)$ is an injective group homomorphism $G_i \m G_j$ and a smooth equivariant embedding $\lambda: \U_i \m \U_j$. We require that for points in two charts, there exists a third chart containing the point injecting into both charts.

\end{definition}

The space $X$ is called the underlying space of the orbifold. 

\begin{definition}
A  map between orbifolds $f:(X,\mathcal U) \m (Y,\mathcal V)$ is a map of spaces $f:X \m Y$ such that for all points $x \in X$, there are charts $(\U,G, \phi)$ and $(\V,H,\psi)$ and a local lift of $f$ to $\tilde f:\U \m \V$ and a group homomorphism $G \m H$ with the following properties. The point $x$ is in $\phi(U)$ and $f(x) \in \psi(V)$, $f$ is equivariant, and $\psi \circ \tilde f= f \circ \phi$.
\end{definition}

We call the map of spaces $f:X \m Y$ the underlying map.

\begin{definition}
Two orbifolds are called equivalent if there are orbifold maps between them whose compositions are the identity map.
\end{definition}

\subsection{Bundles and flows}

In this subsection, we recall the theory of vector and principle bundles, vector fields and flows on orbifolds. Flows on orbifolds will be an important tool for proving that some moduli spaces of $J$-holomorphic curves are independent of $J$. The notion of vector bundle generalizes to orbifolds as follows.

\begin{definition}
A $k$-dimensional vector orbibundle over an orbifold $X$ is a map of orbifolds $\pi:E \m X$ satisfying the following properties. There exist charts  $\mathcal U = \{ \tilde U_i,G_i,\phi_i\}$ on $X$ such that  $\{ \tilde U_i \times \R^k,G_i,\phi_i'\}$ give charts on $E$. The map $\pi$ is the standard projection $\tilde U_i \times \R^k \m \tilde U_i$ in these charts. We also require the groups $G_i$ to act linearly on the second component. If $\lambda: (\tilde U_i,G_i,\phi_i) \m (\tilde U_j,G_j,\phi_j)$ is an injection, we require that there is an injection $\lambda': (\tilde U_i \times \R^k,G_i,\phi_i') \m (\tilde U_j \times \R^k,G_j,\phi_j')$ such that $\lambda'$ is linear on the second factor.
\end{definition}

\begin{definition}
If $\pi:E \m X$ and $\pi':E' \m X$ are vector orbibundle over an orbifold $X$, then $f:E \m E'$ is an equivalence of vector orbibundles if $f$ is an equivalence of orbifolds, $\pi' \circ f =\pi$ and $f$ satisfies the following condition. On charts $\{ \tilde U_i \times \R^k,G_i,\phi_i\}$ on $E$ and $\{ \tilde U'_i \times \R^k,G'_i,\phi'_i\}$ on $E'$ used to define the vector orbibundle structure, $f$ is required to induce a linear map on the second factor.
\end{definition}

We likewise define principle $G$ orbibundles by replacing $\R^k$ with $G$ and the word linear with the phrase multiplication by elements of $G$.

\begin{definition}
A section of a vector or principle bundle $\pi: E \m X$ is a map of orbifolds $\sigma: X \m E$ such that $\pi \circ \sigma = id$.

\end{definition}

\begin{proposition}
Let $X$ be an orbifold with charts  $\mathcal U = \{ \tilde U_i,G_i,\phi_i\}$.  There exists an orbifold $TX$ with charts $\{ T\tilde U_i,G_i,\phi_i'\}$ and a vector bundle map $\pi:TX \m X$ induced by $T\U_i \m \U_i$. Equivalent atlases on $X$ give equivalent atlases on $TX$.

\end{proposition}

See \cite{BK} for a proof. We call this vector orbibundle the tangent bundle and we call sections of the tangent bundle, vector fields. Just as in the case of manifolds, compactly supported vector fields induce one parameter families of self maps of orbifolds.

\begin{definition}

Let $\sigma$ be a vector field on an orbifold $X$. Let $\sigma_0$ be the zero vector field. Let $supp(\sigma) \subset X$ be the closure of the set of points where the underlying maps of $\sigma$ and $\sigma_0$ disagree.
\end{definition}

\begin{definition}
Let $\sigma$ be a vector field on an orbifold $X$ and $F:\R \times X \m$ be a map of orbifolds. Let $F_t:X \m X$ denote the map induced by restricting to the point $t$ in the first component. The map $F$ is a called a flow if $ F_t \circ F_s=F_{s+t}$. For $x \in X$, let $\gamma_x: \R \m X$ be the map obtained by restricting $F$ to the point $x$. Let $( \tilde U_i,G_i,\phi_i)$ be a chart containing $x$ and $\tilde \sigma$ and $\tilde \gamma_x$ be lifts of $\sigma$ and $\gamma_x$. If for some choice of lifts, we have that the derivative of $\tilde \gamma_x$ at $t=0$ equals the value of $\tilde \sigma$ at $x$ for all points $x \in X$, then we say that $F$ is the flow generated by $\sigma$.

\end{definition}

In \cite{KL}, they make the following observation.

\begin{proposition}
If $\sigma$ is a vector field with compact support, $\sigma$ generates a unique flow.
\end{proposition}

\begin{definition}
Let $X$ be an orbifold and let $U_\alpha$ be an open cover of the underlying space of $X$. A collection of orbifold maps $f_{\alpha}:X \m \R $ are called an orbifold partition of unity subordinate to the open cover if they induce a partition of unity  subordinate to the open cover on the underlying space.
\end{definition}

To see that orbifold partitions of unity always exist, see \cite{KL}. The theory of classifying spaces for principle bundles also carries over to orbifolds after one replaces the orbifold with its homotopy type. 

\begin{theorem}
To each orbifold $X$, there is a space $BX$ called the classifying space or homotopy type of $X$ such that the following properties hold:

i) If $f:X \m Y$ is an equivalence of orbifolds, there is a natural map $f_*:BX \m BY$ that is a homotopy equivalence.

ii) Equivalence classes of principle $G$ orbibundles over $X$ are in bijection with homotopy classes of maps from $BX$ to $BG$.

\label{theorem:BG}

\end{theorem}

See \cite{BK} for a proof and the construction.

\section{$J$-Holomorphic curves and automatic transversality}\label{secauto}

In this section, we review the basic theory of $J$-holomorphic curves. We recall the definitions of several $\dbar$-operators and relate these operators to tangent spaces of various moduli spaces and mapping spaces. We state Gromov's compactness theorem. See \cite{MS} for more details, especially regarding the genus zero case. We also review the theory of automatic transversality as developed by \cite{G}, \cite{HLS}, \cite{IS} and \cite{S}. This is a tool which allows one to prove that various moduli spaces are smooth when the target symplectic manifold is a real $4$-manifold. Automatic transversality is the primary reason our results only apply to $\p$ as opposed to higher dimensional complex projective spaces. We also recall McDuff's result from \cite{Mc2} that the adjunction formula from algebraic geometry applies to $J$-holomorphic curves in symplectic $4$-manifolds. This theorem allows one to bound the number of singularities that $J$-holomorphic curves can have in terms of topological data.

\subsection{Linearized $\bar \partial$-operators}

In this subsection we review several linearized $\dbar$-operators. These will be relevant since, when onto, the null spaces of these operators will be naturally isomorphic to the tangent spaces of various $J$-holomorphic mapping spaces and moduli spaces. Surjectivity of these operators will also be relevant for the existence of gluing maps in later sections.

\begin{definition}

An almost complex structure $J$ is a section of $Hom_\R(TM,TM) \m M$ such that $J^2=-id$. If $(M,\omega)$ is a symplectic manifold, an almost complex structure $J$ is called compatible with $\omega$ if $\omega$ and $J$ pair to form a Riemannian  metric. Let $\mathcal{J}$ denote the space of smooth almost complex structures compatible with $\omega$.

\end{definition}

\begin{proposition}
The space $\mathcal{J}$ is a contractible infinite dimensional \ft manifold.

\end{proposition}

See \cite{G} or \cite{MS} for a proof. Since $\mathcal{J}$ is contractible, it is path connected. Throughout the paper, all almost complex structures are required to be compatible with the symplectic form and of class $C^{\infty}$. In this paper, the symbol $\s$ will denote the pair $(S^2,j_0)$ with $j_0$ the standard almost complex structure on the space $S^2$. However, for higher genus surfaces, there is no preferred almost complex structure. Eventually we will want to consider all almost complex structures on the domain, but for now, fix an almost complex structure $j$ on a surface $\Sigma_g$.

\begin{definition}
Let $C^{\infty}(\Sigma_g,M)$ denote the space of smooth maps between $\Sigma_g$ and $M$. Consider $C^{\infty}(\Sigma_g,M)$ as an infinite dimensional \ft manifold topologized with the $C^{\infty}$ topology. If $E \m \Sigma_g$ is a smooth vector bundle, let $\Gamma(E)$ denote the \ft manifold of smooth sections of $E$.
\end{definition}

\begin{definition}
Fix an almost complex structure $J$ on $M$. Let $\Upsilon \m C^{\infty}(\Sigma_g,M)$ be the infinite dimensional \ft bundle whose fiber over a map $u$ is $\Omega^{0,1}(u^*TM)$. Here $\Omega^{0,1}(u^*TM)$ denotes the space of smooth anti-complex linear one forms with values in the pullback of the tangent bundle of $M$.
\end{definition}

\begin{definition}

Let $u: \Sigma_g \m M$ be smooth. The non-linear $\bar \partial$-operator $\dbar^{nl}$ is defined by the formula $\dbar ^{nl}(u)=\frac{1}{2}(Du+J \circ Du \circ j) \in \Omega^{0,1}(u^*TM)$. The map $u$ is said to be $J$-holomorphic if $\dbar^{nl}(u)=0$.

\end{definition}

The operator $\dbar^{nl}$ can be viewed as a section of $\Upsilon$.

\begin{definition}
For $A\in H_2(M)$, let $\Hol_A((\Sigma_g,j),(M,J))$ denote the subspace of $C^{\infty}(\Sigma_g,M)$ of maps $u$ with $\dbar^{nl} u=0$ and $u_* [\Sigma_g ]=A$. Let $Z=\{z_i\} \subset \Sigma_g$ be a finite set of points and let $\z:Z \m M$ be an injection. Let $\Hol^\mathfrak{z}((\Sigma_g,j),(M,J))$ denote the subspace of $\Hol((\Sigma_g,j),(M,J))$ of maps with $u(z_i)=\mathfrak{z}(z_i)$.

\end{definition}

When $|Z|=1$, we denote $\Hol^\mathfrak{z}((\Sigma_g,j),(M,J))$ by $\Hol^*((\Sigma_g,j),(M,J))$. We recall the definition of two linearizations of the equation $\dbar^{nl} =0$ at a map $u: \Sigma_g \m M$. The first will correspond to solving the equation by only varying $u$ while the other will correspond to solving the equation by varying $u$ and the almost complex structure on the domain. We shall define: $$D_u: \Gamma(u^*TM) \m \Omega^{0,1}(u^*TM)$$ $$\text{and}$$ $$\tilde D_u: \Omega^{0,1}(T\Sigma_g) \times \Gamma(u^*TM) \m \Omega^{0,1}(u^*TM).$$

The operator $\dbar^{nl}$ defines a section of the bundle $\Upsilon$ defined above. Thus the derivative $D \dbar^{nl}$ is a bundle map $D \dbar ^{nl}:TC^{\infty}(\Sigma_g,M) \m T\Upsilon$. The tangent bundle to the \ft  manifold $C^\infty({\Sigma_g,M})$ is a bundle whose fiber at a map $u$ is $\Gamma(u^*TM)$. Let $V \in \Upsilon_u$. There is a non-canonical splitting of $T\Upsilon_{V}$ into $\Upsilon_u$ and $T_u C^{\infty}(\Sigma_g,M)$. In \cite{MS}, McDuff and Salamon use a Hermitian connection to construct a continuous family of projections $\pi_V: T\Upsilon_{V} \m \Upsilon_u$. This construction allows us to define the following linearized $\dbar$-operator. Also see Formula 3.2 of \cite{MS}.

\begin{definition}
For $u:\Sigma_g \m M$ smooth, let $D_u: \Gamma(u^*TM) \m \Omega^{0,1}(u^*TM)$ be given by the formula $D_u = \pi_{\dbar^{nl} u} \circ D \dbar ^{nl}_u$.

\end{definition}

\begin{definition}

For $u:\Sigma_g \m M$ smooth, let $\tilde D_u: \Omega^{0,1}(T\Sigma_g) \times \Gamma(u^*TM) \m \Omega^{0,1}(u^*TM)$ be given by the formula $\tilde D_u(v, \xi)=J du(v) + D_u(\xi)$ for $v \in  \Omega^{0,1}(T\Sigma_g)$ and $\xi \in \Gamma(u^*TM) $.
\end{definition}

If the map $u$ is a non-constant $J$-holomorphic map to a symplectic four manifold, Ivashkovich and Shevchishin in \cite{IS} defined a divisor $D$ which is the divisor of zeros of $du$ counted with multiplicity. In particular, if $u$ is an immersion, $D$ is the empty divisor. Let $\mathcal L(D)$ be the holomorphic line bundle associated to $D$. That is, $\mathcal L(D)$ is a vector bundle such that its associated sheaf of holomorphic sections is naturally isomorphic to the sheaf of meromorphic sections of the trivial complex line bundle with poles only on $D$ with order less than or equal to the order of the divisor at that point. In \cite{IS}, Ivashkovich and Shevchishin also constructed a normal bundle (generalizing that of \cite{HLS}) called $N$ and proved the following propositions (also see \cite{S}).

\begin{proposition}

There is a linear operator $\n: \Gamma(N) \m \Omega^{0,1}(N)$ making the following diagram of short exact sequences commute.

$$
\begin{array}{lcccccccl}
0  &\m &\Gamma(T\Sigma_g \otimes \mathcal L(D)) &\m  & \Gamma(u^* TM) & \m & \Gamma(N) & \m & 0                \\
\downarrow  & & \bar{\partial}_\tau \downarrow  &    & D_u \downarrow &  & \n \downarrow &  & \downarrow              \\

0 & \m  &\Omega^{0,1}(T\Sigma_g \otimes \mathcal L(D)) &\m & \Omega^{0,1}(u^*TM) & \m &\Omega^{0,1}(N) &\m & 0
\end{array}$$ Here $ \bar{\partial}_\tau =D_{id}$ is the linearized  $\bar{\partial}$-operator at the identity map $id: \Sigma_g \m \Sigma_g$. 

\label{theorem:commuteOperators}

\end{proposition}

Although $\n$ is Fredholm and $\tilde D_u$ is not, the two operators are closely related. In \cite{S}, Sikorav proved the following theorem.

\begin{proposition}
If $u$ is a non-constant $J$-holomorphic map to a symplectic 4-manifold, the operator $\n: \Gamma(N) \m \Omega^{0,1}(N)$ is surjective if and only if $\tilde D_u: \Omega^{0,1}(T\Sigma_g) \times \Gamma(u^*TM) \m \Omega^{0,1}(u^*TM)$ is onto. Moreover, $coker(\n) = coker(\tilde D_u)$.
\label{theorem:1}
\end{proposition}

\subsection{Moduli spaces, holomorphic mapping spaces and their tangent spaces}

In this subsection, we discuss various relevant moduli spaces and holomorphic mapping spaces. We recall the relationship between the kernels of various $\dbar$-operators and the tangent spaces of holomorphic mapping spaces and moduli spaces. Let $Z=\{z_i\} \subset \Sigma_g$ be a finite set of points and $\mathfrak{z}:Z \m M$ an injection. For a bundle $E$ over $\Sigma_g$, let $\Gamma(E)_Z \subset \Gamma(E)$ denote the subspace of sections which vanish on $Z$. Let $D^{Z}_u: \Gamma(u^*TM)_Z \m \Omega^{0,1}(u^*TM)$ be the restriction of $D_u$ to $\Gamma(u^*TM)_Z$. Likewise define $\n^z$ and $\tilde D_u^Z$. Recall that $\Hol^\mathfrak{z}((\Sigma_g,j),(M,J))$ denotes the subspace of $\Hol((\Sigma_g,j),(M,J))$ of maps with $u(z_i)=\mathfrak{z}(z_i)$. 

\begin{theorem}
Let $u \in \Hol^\mathfrak{z}((\Sigma_g,j),(M,J))$. If $D^Z_u: \Gamma(u^*TM) \m \Omega^{0,1}(u^*TM)$ is surjective, then a neighborhood of $u$ in $\Hol^\mathfrak{z}((\Sigma_g,j),(M,J))$ is homeomorphic to a neighborhood of zero in $ker(D_u^Z)$. Thus the tangent space of $\Hol^\mathfrak{z}((\Sigma_g,j),(M,J))$ at the point $u$ can be identified with $ker D_u^Z$. If $D_u^Z$ is onto for all $u\in \Hol^\mathfrak{z}((\Sigma_g,j),(M,J))$, then $\Hol^\mathfrak{z}((\Sigma_g,j),(M,J))$ is a smooth manifold. 
\end{theorem}

See \cite{MS} for a proof. Recall that a sufficient condition for a map between manifolds to have manifolds as fibers is that the map is a submersion. The above theorem regarding a space being smooth can be strengthened to a statement about a map being a submersion. Fix a path of compatible almost complex structures $J_t$. See \cite{MS} for a description of a natural topology on $\bigcup_t \Hol^\mathfrak{z}((\Sigma_g,j),(M,J_t))$. There is a natural map $$\pi:\bigcup_t \Hol^\mathfrak{z}((\Sigma_g,j),(M, J_t)) \m \R$$ such that the fiber over $t \in \R$ is $\Hol^\mathfrak{z}((\Sigma_g,j),(M,J_t))$. The above theorem can be strengthened to the fact that if $D_u$ is onto, then there is a neighborhood of $u$ in $\bigcup_t \Hol^\mathfrak{z}((\Sigma_g,j),(M, J_t))$ that is a smooth manifold and $\pi$ is a submersion at $u$ \cite{MS} \cite{S}. All of the theorems about smoothness of moduli spaces or mapping spaces in this section can be rephrased in terms of the corresponding projection maps being submersions.

The operator $D_u$ will be primarily of interest in genus zero. In higher genus, $D_u$ is generally not onto. However, $\tilde D_u$ is often onto. This operator is not Fredholm as its kernel is infinite dimensional. Therefore,  $ker \tilde D_u$ will not correspond to the tangent space of a finite dimensional mapping space. Let $\mathcal{J}_g$ denote the space of smooth almost complex structures on $\Sigma_g$ compatible with a fixed orientation form. Consider the following infinite dimensional $J$-holomorphic mapping space.

\begin{definition}
Let $\Hol^{\mathfrak z}(\Sigma_g,(M,J)) \subset C^{\infty}(\Sigma_g,M) \times \mathcal{J}_g$ be the subset of maps and almost complex structures $(u,j)$ such that $u$ is $(j,J)$-holomorphic and $u(z_i)=\mathfrak z(z_i)$. For $A \in H_2(M)$, let $\Hol_A^{\mathfrak z}(\Sigma_g,(M,J))$ be the subspace of maps such that $u_*([\Sigma_g])=A$.
\end{definition}

This is the space of all maps that are $(j,J)$-holomorphic for some $j$ on $\Sigma_g$.

\begin{definition}
The moduli space of maps, denoted $\mathcal M^{\z}_g(M,J,A)$, is the quotient of $\Hol_A^{\z}(\Sigma_g,(M,J))$ by $Diff_Z(\Sigma)$. Here $Diff_Z(\Sigma_g)$ is the group of diffeomorphisms point-wise fixing $Z$ acting on functions by precomposition and on almost complex structures by pullback.
\end{definition}

Note that for $g=0$ and $Z=\varnothing$, we could also  define $\mathcal M_0(M,J,A)$ as $\Hol_A(\s) (M,J))/PSL_2(\C)$. Here $PSL_2(\C)=\Hol_1((\s,\s)$ acts by precomposition. Also note that in genus zero, $ker \dbar_{\tau}$ is a model for the tangent space of $PSL_2(\C)$ at the identity. The operator $ \dbar_{\tau}$ depends on a choice of divisor and here we choose that divisor to be the empty divisor since every automorphism is an immersion.  In general, the action by $Diff_Z(\Sigma_g)$ is not free. If $2g\geq 3-|Z|$, however, the action has finite stabilizers since the action of $Diff_Z(\Sigma_g)$ on $\mathcal J_g$ has finite stabilizers. The action also has finite stabilizers whenever $u$ is not constant \cite{MS}. In fact, the only non-constant maps with non-trivial stabilizer groups are multiply covered maps. 

\begin{definition}
A non-constant $J$-holomorphic map $u: \Sigma_g \m M$ is called a multiple cover if there is a holomorphic branched cover $c:\Sigma_g \m \Sigma_{g'}$ of degree at least two and a $J$-holomorphic map $v:\Sigma_{g'} \m M$ such that $u=c \circ v$.
\end{definition}

\subsection{Nodal curves and Gromov compactness}

None of the mapping spaces and few of the moduli spaces described in the previous section are compact. However, the moduli space of maps has a useful compactification introduced by Gromov in \cite{G}. We will only need the case of genus zero. The proofs of all of the statements in this subsection are contained in \cite{MS}.  One compactifies this moduli space by adding what are called nodal curves.

\begin{definition}
A nodal surface $\Sigma$ is a finite collection of surfaces $\{ \Sigma_{g_i} \}$ and finite collection of pairs of distinct points $\{(x_k,y_k)\}$ on the surfaces such that the quotient space, $\bigsqcup \Sigma_{g_i} / \sim$ is connected. Here $\sim$ is the relation $x_k \sim y_k$. The surfaces $\Sigma_{g_i}$ are called the irreducible components of $\Sigma$. A complex nodal curve is a nodal surface with choice of almost complex structure $j_i$ on each irreducible component.
\end{definition}

The pairs of points $x_k$ and $y_k$ are called nodes. A $J$-holomorphic map $u$ from a nodal curve $\Sigma$ is a collection of $(j_i,J)$-holomorphic maps $u_i:\Sigma_{g_i} \m M$ such that $u_i(x_k)=u_j(y_k)$. Let the symbol $u_*[\Sigma]$ denote the sum $\sum_i u_{i*}[\Sigma_{g_i}]$. We fix a finite subset of $Z \subset \Sigma$ of size $N$ that does not include any nodes. Let $Diff_Z(\Sigma)$ denote the subgroup of $Diff(\bigsqcup  \Sigma_i)$ that sends pairs of nodes to pairs of nodes and point-wise fixes $Z$. A map $u:\Sigma \m M$ is called stable if for each $i$ with $u_i$ constant, we have $2 g_i \geq 3 - |(\{\text{nodes}\} \cup Z) \cap \Sigma_{g_i} | $. A nodal curve is called rational if the genus of every surface is zero and if the quotient space $\bigsqcup \Sigma_{g_i} / \sim$ is simply connected.

\begin{definition}
As a set, let $\bar{ \mathcal M}^{\z}_0 (M,J,A)$ denote the set of rational stable nodal curves $(\Sigma,j)$ and $J$-holomorphic maps $u:\Sigma \m M$ such that $u_*([\Sigma])=A$ and $u(z_i)=\z(z_i)$, modulo the action of the groups $Diff_Z(\Sigma)$.
\end{definition}

\begin{definition}
As a set, let $\overline {\mathfrak{M} }^{\z}_0(M,A) = \bigcup_J \bar{ \mathcal M_0^{\z}}(M,J,A)$.
\end{definition}

Initially, these two sets are stratified by sets with a natural topology. To topologize these sets, see \cite{G} or \cite{MS}. In \cite{G}, Gomov proved the following theorem known as Gromov's compactness theorem.

\begin{theorem}
The natural map $\pi: \overline {\mathfrak{M} }^{\z}_0(M,A) \m \mathcal{J}$ is proper. In particular, each $\bar{ \mathcal M}^{\z}_0(M,J,A)$ is compact.
\label{theorem:gromovcompactness}
\end{theorem}

The $\dbar$-operators from the previous section generalize to the case of nodal curves.

\begin{definition}
Let $(\Sigma,j)$ be a nodal curve. If $u:\Sigma \m M$ is a $J$-holomorphic map, let $\Gamma(u^*TM)$ denote the subspace of $\prod_i \Gamma(u_i^*(TM))$ of tuples of sections $(\xi_1, \xi_2, \ldots )$ that agree at the nodes. Let $\Omega^{0,1}(u_i^*(TM)) = \prod_i \Omega^{0,1}(u_i^*(TM))$ and let $\Omega^{0,1}(T \Sigma) = \prod_i \Omega^{0,1}(T\Sigma_{g_i})$. Define $\Gamma(u^*TM)_Z$ to be the subspace of sections vanishing on $Z$.
\end{definition}

\begin{definition}
If $u:\Sigma \m M$ is a $J$-holomorphic map from a nodal curve, let $D^Z_u: \Gamma(u^*TM)_Z \m \Omega^{0,1}(u_i^*(TM))$ denote the restriction of $D_{u_1} \times D_{u_2} \times \ldots$ to sections vanishing on $Z$ and agreeing on the nodes. Likewise define $\tilde D^Z_u:\Omega^{0,1}(T\Sigma) \times \Gamma(u^*TM)_Z \m \Omega^{0,1}(u_i^*(TM))$.
\end{definition}

The part of the following theorem regarding the operator $D_u$ was proved in \cite{RT1}, \cite{RT2}, \cite{FO}, \cite{LT} and \cite{Si} and the part regarding the operator $\tilde D_u$ was proved in \cite{S}.

\begin{theorem}
Let $u \in \overline {\mathfrak{M} }^{\z}_g(M,A)$, with $K$ pairs of nodes and $u_*[\Sigma]=A$. Let $I=<c_1(TM),A>+(dim_{\R}(M)-6)(1-g-|Z|)$. Let $G \subset Diff_Z(\Sigma)$ be the stabilizer group of the map $u$. Assume that $D^Z_u$ or $\tilde D^Z_u$ is onto. Then there is a local homeomorphism: $$\phi : \overline {\mathfrak{M} }^{\z}_g(M,A) \m (\C^K \times C^{I-K}/G) \times \mathcal{J}.$$ Moreover, the group $G$ is finite and acts linearly and the map $\phi$ commutes with the two natural maps to $\mathcal J$. Let $\mu$ be $\phi$ composed with the projection onto $\C^K$. Then for $v \in \overline {\mathfrak{M} }_g(M,A)$ near $u$, the number of components of the vector $\mu(v)$ which are zero equals the number of pairs of nodes of the domain of $v$. Additionally, these coordinate charts induce orbifold charts on the spaces $\bar{\mathcal M}_g(M,J,A)$.
\label{theorem:zero}
\end{theorem}

\subsection{Automatic transversality in dimension 4}

Often one is interested in proving that linearlized $\bar \partial$-operators are surjective. For example, we noted that surjectivity of linearized $\dbar$-operators is related to whether or not various moduli and mapping spaces are manifolds/orbifolds. Although linearlized $\dbar$-operators are not always surjective, one can often prove that for a generic choice of almost complex structure, certain linearlized $\bar \partial$-operators are surjective. However, for linearlized $\bar \partial$-operators acting on sections of complex line bundles, there is a topological criterion for proving surjectivity. This is the phenomena called automatic transversality. It has applications to the study of $J$-holomorphic maps to $4$-manifolds since $u^*TM$ often splits into the direct sum of two complex line bundles (normal and tangent bundles; see for example Proposition \ref{theorem:commuteOperators}). In \cite{HLS}, they introduced the following definition and proved the following theorem in the case of $Z=\emptyset$.

\begin{definition}

Let $E$ be a holomorphic line bundle. Let $\nabla^{0,1}: \Gamma(E) \m \Omega^{0,1}(E)$ be the anti-holomorphic part of $\nabla$, a Hermitian connection. We call a first order differential operator $L: \Gamma(E) \m  \Omega^{0,1}(E)$ a generalized $\bar \partial$-operator if $L=\nabla^{0,1} + a$ with $a \in \Omega^{0,1}(End_{\R} E)$.

\end{definition}

\begin{theorem}

If $E$ is a holomorphic line bundle on $\Sigma_g$, $L: \Gamma(E) \m  \Omega^{0,1}(E)$ a generalized $\bar \partial$-operator, and $c_1(E) > 2g-2$, then $L$ is surjective. For any finite collection points $Z$, let $L^Z$ denote the restriction of $L$ to $\Gamma(E)_Z$. If $c_1(E) > 2g -2 + |Z|$, then $L^Z$ is surjective.
\label{theorem:auto}
\end{theorem}

Generalizing this to the case where $Z \neq \emptyset$ is straightforward and is implicit in the work of \cite{S}.

Theorem \ref{theorem:auto} is what is referred to as automatic transversality. It is automatic in the sense that one can prove that a linearized $\bar \partial$-operator is onto without perturbing the almost complex structure. It pertains to transversality since it can be used to prove that $\bar \partial^{nl}$ is transverse to the zero section of $\Upsilon$.

In \cite{HLS}, they also proved that $\n$ and $\bar \partial_{\tau}$ are generalized $\bar \partial$-operators. Using Proposition \ref{theorem:1}, this gives a topological criterion for proving that $\tilde D_u$ is surjective. Since $c_1(T\Sigma_g)=2-2g$, $\bar \partial_{\tau}$ associated to immersions is onto when $g=0$. Combining this with Theorem \ref{theorem:auto} and Theorem \ref{theorem:commuteOperators}, we get a topological criterion for proving that $D_u$ is surjective. More explicitly we have the following corollaries.

\begin{corollary}
If $u:\Sigma_g \m M$ is a non-constant $J$-holomorphic curve, $M$ a symplectic 4-manifold, then $\tilde D_u$ and $\n$ are onto if $c_1(u^*TM) >|D|$ where $D$ is the divisor of zeros of the derivative of $u$. These operators are onto when restricted to the subspace of sections vanishing on $Z$ if $c_1(u^*TM) >|D|+|Z|$.
\end{corollary}

\begin{proof}
Note that $c_1(u^*TM) = c_1(N) +2-2g+|D|$. Thus $c_1(N)> 2g-2 +|Z|$ and by Theorem \ref{theorem:auto}, $\n^Z$ is onto. By Theorem \ref{theorem:1}, $\tilde D^Z_u$ is also onto.
\end{proof}

\begin{corollary}
If $u:\Sigma_g \m M$ is a non-constant $J$-holomorphic curve, $M$ a symplectic 4-manifold, then $D^Z_u$ is onto if $4-4g+|D|-|Z|>0$ and $c_1(u^*TM) >|D|+|Z|$.

\label{corollary:DuOnto}

\end{corollary}

\begin{proof}
For a divisor $Q$, let $\mathcal L (Q)$ denote the associated line bundle. Let $D$ denote the divisor of zeros of the derivative of $u$ and view the set $Z$ as a divisor. By a generalization of Theorem \ref{theorem:commuteOperators} to the case $|Z|>0$, we have a commuting diagram of short exact sequences:

$$
\begin{array}{lcccccccl}
0  &\m &\Gamma(T\Sigma_g \otimes  \mathcal L (D-Z)) &\m  & \Gamma(u^* TM)_Z & \m & \Gamma(N)_Z & \m & 0                \\
\downarrow  & & \bar{d}^Z_\tau \downarrow  &    & \bar{d}^Z_u \downarrow &  & \n^Z \downarrow &  & \downarrow              \\

0 & \m  &\Omega^{0,1}(T\Sigma_g \otimes \mathcal L (D)) &\m & \Omega^{0,1}(u^*TM) & \m &\Omega^{0,1}(N) &\m & 0.
\end{array}$$

Here $N$ is the normal bundle of the curve $u$. By the snake lemma, we have the following exact sequence: $coker \dbar^Z_{\tau} \m coker D^Z_u \m coker \n^Z \m 0$. By Theorem \ref{theorem:auto}, $ \dbar^Z_{\tau}$ and $\n^Z$ are onto so $coker D^Z_u  =0$.

\end{proof}

In particular, the hypotheses of the above theorem are satisfied if $u$ is an immersion, $g=0$, $c_1(u^*TM) >|Z|$ and $|Z|<4$. This allows us to prove many relevant holomorphic mapping spaces are smooth manifolds and relevant moduli spaces are orbifolds.

In \cite{Mc2}, McDuff proved that the adjunction formula from algebraic geometry applies to $J$-holomorphic curves in symplectic 4-manifolds. This allows one to get a topological bound for the size of $D$, the divisor of zeros of a $J$-holomorphic map to a symplectic 4-manifold. This is relevant since the size of this divisor has appeared in several formulas in this subsection regarding when linearized $\dbar$-operators are surjective. Since we are only interested in the case of $\p$, we only state the formula in that case.

\begin{theorem}
If $u: \Sigma_g \m \C P^2$ is a non-multiply-covered $J$-holomorphic map of degree $d$, then the number of points where $u$ is not an embedding (counted with multiplicity) is equal to $(d-1)(d-2)/2-g$. Hence $|D| \leq (d-1)(d-2)/2-g$.
\end{theorem}

See \cite{Mc2} for a precise statement of how to count singularities with multiplicity. Note that $(d-1)(d-2)/2-g$ measures both points of self intersection and points of non-immersion. Therefore $|D|$ is not necessarily equal to $(d-1)(d-2)/2-g$.

\section{Degree one and two holomorphic spheres in $\C P^2$} \label{secdeg}

In this section we describe automatic transversality and Gromov compactness arguments which show that the topology of the space of degree one and two $J$-holomorphic maps to $\p$ is independent of the choice of almost complex structure $J$. The results about degree one maps are a small extension of the results in \cite{G} and appear in \cite{Mi1}. The case of degree two maps does not seem to appear in the literature in this form. The moduli space of degree two $J$-holomorphic curves is an orbifold. We use the theory of flows and principle bundles on orbifolds as reviewed in Section \ref{secorb}.

    \subsection{Degree one holomorphic spheres in $\C P^2$}\label{secdeg1}

In this subsection we recall the results of \cite{G} and \cite{Mi1} on the topology of the space of degree one $J$-holomorphic maps. In \cite{Mi1}, the author proved the following propositions.

\begin{proposition}
For any compatible almost complex structure $J$, $\Hol^*_1(\s,(\p,J))$ is diffeomorphic to $\Hol_1^*(\s,(\p,J_0))$. Moreover, there is such a diffeomorphism $\Phi$ making the following diagram homotopy commute.

$$
\begin{array}{ccccccccl}
\Hol_1^*(\s, (\C P^2,J_0))  &\overset{\Phi}{\m} & \Hol_1^*(\s, (\C P^2,J)) \\
    &  \searrow i & \downarrow i    &            \\

&  & \Omega^2 \p \\

\end{array}$$

\label{theorem:deg1}
\end{proposition}

\begin{definition}
Let $ev:\Hol_1(\s,(\p,J)) \m \p$ be the map defined by $ev(u)=u(\infty)$.
\end{definition}

The spaces $\Hol_1^*(\s,(\p,J))$ are the fibers of the evaluation at $\infty$ map $ev:\Hol_1(\s,(\p,J)) \m \p$. The previous proposition shows that the fibers are all diffeomorphic because they are for the standard almost complex structure on $\p$. In fact, in \cite{Mi1} it was shown that the evaluation map  $ev:\Hol_1^*(\s,(\p,J)) \m \p$ is a smooth fiber bundle.

\begin{proposition}
For any compatible almost complex structure $J$, the map $ev:\Hol_1(\s,(\p,J)) \m \p$ is fiber bundle.
\label{theorem:evaluation1}
\end{proposition}

    \subsection{Degree two holomorphic spheres in $\C P^2$} \label{secdeg2}

We will now prove that the topology of degree two $J$-holomorphic mapping spaces is independent of $J$. Some degree two curves are multiply covered and hence the moduli spaces have orbifold points. Also, because degree two curves can degenerate into two degree one curves, the moduli space and the compactification are not equal. Let $Z=\{\infty\} \subset \s$ and $\z(\infty)=p_0 \in \p$. There are four different types of elements of $\bar \M^{\z}_0(\p,J,2)$ that we will treat separately.

\paragraph{Type 1}

These curves are represented by a map $u: \s \m \p$ with $u(\infty) = p_0$. The map $u$ is an embedding and non-multiply covered. These curves are elements of $\M^{\z}_0(\p,J,2)$ and have trivial automorphism groups.

\paragraph{Type 2a}

These curves are represented by a map $u: \s \m \p$ with $u(\infty) = p_0$. Here the map $u$ is of the form $c \circ v$. Here $c: \s \m \s$ is a degree two branched cover and $v:\s \m \p$ is a degree one embedding. We assume that $\infty$ is not a branch point of $c$. These curves are elements of $\M^{\z}_0(\p,J,2)$ and have trivial automorphism groups.

\paragraph{Type 2b}

These curves are identical to those of type 2a but we assume that $\infty$ is a branch point of $c$. The automorphism groups of these curves are $\Z_2$. These curves are elements of $\M^{\z}_0(\p,J,2)$.

\paragraph{Type 3}

These curves are represented by two degree one maps $u_1,u_2: \s \m \p$ with $u_1(\infty) = p_0$ and $u_1(0)=u_2(\infty)$. We require that $u_1$ and $u_2$ have distinct images. The automorphism groups of these curves are trivial. These curves are elements of $\partial \bar \M^{\z}_0(\p,J,2)$.

\paragraph{Type 4a}

These curves are represented by two degree one maps $u_1,u_2: \s \m \p$ and a degree zero (constant) map $u_3: \s \m \p$ with the constraint that $u_1(\infty)=u_3(0)$, $u_2(\infty)=u_3(1)$ and $u_3(\infty)=p_0$. Note that this condition is really just that $u_1(\infty)=u_2(\infty)=p_0$; however, we add in the extra genus zero curve to make it a stable map. We require that the images of $u_1$ and $u_2$ are different.  These curves are elements of $\partial \bar \M^{\z}_0(\p,J,2)$ and have trivial automorphism groups.

\paragraph{Type 4b}

These curves are identical to those of type 2a but we instead require that the images of $u_1$ and $u_2$ are the same.  These curves are elements of $\partial \bar \M^{\z}_0(\p,J,2)$ and have automorphism groups $\Z_2$.

$$ $$ By Theorem \ref{theorem:zero}, we can show that a moduli space is a smooth orbifold if we can prove that the relevant linearized $\bar \partial$-operators are surjective. This is the case for rational degree two curves in $\p$.

\begin{proposition}
Let $J$ be any compatible almost complex structure on $\p$. If $u \in \bar \M^{\z}_0(\p,J,2)$, then $D^Z_u$ is onto. In particular, this gives a smooth orbifold structure on $\bar \M^{\z}_0(\p,J,2)$.
\label{theorem:deg2orbifold}
\end{proposition}

\begin{proof}

The proof involves analyzing all four cases. In all cases $|Z|=1$. The symmetry groups do not affect transversality so we can address  Type 2a and Type 2b curves or Type 4a and Type 4b curves at the same time. 

\textbf{Type 1:} Since the map $u$ is an embedding, $|D|=0$. Note that $c_1(u^*T\p)=3deg(u)=6>|D|+|Z|=1$ and $4-4g+|D|-|Z|=3>0$ so Corollary \ref{corollary:DuOnto} implies that $D^Z_u$ is onto.

\textbf{Type 2:} The map $u$ is the composition of an embedding and a degree two branched cover. By the Riemann Hurwitz formula, degree two branched covers have two simple branch points. Thus, the map $u$ is an immersion except at $2$ points and $|D|=2$.  Note that $c_1(u^*T\p)=6>|D|+|Z|=3$ and $4-4g+|D|-|Z|=5>0$ so Corollary \ref{corollary:DuOnto} implies that $D^Z_u$ is onto.

\textbf{Type 3:} Note that all degree one curves are embeddings. In this case, $\Gamma(u^*T\p)_Z$ is the subspace of $$\Gamma(u_1^* T\p) \times \Gamma(u_2^* T\p)$$ consisting of pairs of sections $(\xi_1,\xi_2)$ with $\xi_1(\infty) = 0$ and $\xi_2(\infty) = \xi_1(0)$. Also recall that $$\Omega^{0,1}(u^*T\p) =\Omega^{0,1}(u_1^* T\p) \times \Omega^{0,1}(u_2^* T\p) $$ and that $ D_u^Z: \Gamma(u^*T\p)_Z \m \Omega^{0,1}(u^*T\p)$ is the restriction of $D_{u_1} \times D_{u_2}$ to $\Gamma(u^*T\p)_Z $. 

Let $Z_1=\{0,\infty\}$ viewed as a divisor on the domain of $u_1$ and let $Z_2=\{\infty\}$ viewed as a divisor on the domain of $u_2$. The map $D_{u_1}:\Gamma(u_1^* T\p)_{Z_1} \m \Omega^{0,1}(u_1^* T\p) $ is surjective by Corollary \ref{corollary:DuOnto} since $c_1(u_1^*T\p)=3>|D|+|Z_1|=2$ and $4-4g+|D|-|Z_1|=2>0$. The map $D_{u_2}:\Gamma(u_2^* T\p)_{Z_1} \m \Omega^{0,1}(u_2^* T\p) $ is surjective by Corollary \ref{corollary:DuOnto} since $c_1(u_2^*T\p)=3>|D|+|Z_1|=1$ and $4-4g+|D|-|Z_1|=3>0$. Let $V \subset \Gamma(u^*T\p)_Z$ be the subspace of sections where $\xi_2(\infty) = \xi_1(0)=0$. The subspace $V$ is isomorphic to $\Gamma(u_2^* T\p)_{Z_1} \times \Gamma(u_2^* T\p)_{Z_2}$ and the restriction of $D_u^Z$ to $V$ is $D_{u_1}^{Z_1} \times D_{u_2}^{Z_2}$. Since $D_u^Z$ is surjective when restricted to $V$, it is surjective.

\textbf{Type 4:}  In these cases, $\Gamma(u^*T\p)_Z$ is the subspace of $$\Gamma(u_1^* T\p) \times \Gamma(u_2^* T\p) \times \Gamma(u_3^* T\p)$$ consisting of sections $(\xi_1,\xi_2,\xi_3)$ with $\xi_1(\infty) = \xi_3(z_1)$ and $\xi_2(\infty) = \xi_3(z_2)$ and $\xi_3(\infty)=0$. We have $$\Omega^{0,1}(u^*T\p) =\Omega^{0,1}(u_1^* T\p) \times \Omega^{0,1}(u_2^* T\p) \times \Omega^{0,1}(u_1^* T\p)$$ and that $ D_u^Z: \Gamma(u^*T\p)_Z \m \Omega^{0,1}(u^*T\p)$ is the restriction of $D_{u_1} \times D_{u_2} \times D_{u_3}$ to $\Gamma(u^*T\p)_Z $.

Let $(\eta_1,\eta_2,\eta_3) \in \Omega^{0,1}(u^*T\p)$ be arbitrary. Since $u_3$ is constant, $u_3 ^* T\p$ is the trivial complex plane bundle $\mathcal O \oplus \mathcal O$. Let $\mathcal L(-\infty)$ be the line bundle associated to the sheaf of sections of $\mathcal O$ vanishing at $\infty$. Since $H^1((\mathcal O \oplus \mathcal O) \otimes \mathcal L (-\infty))=0$, we can find a section $\xi_3 \in \Gamma(u_3^*T\p)$ with $D_{u_3} (\xi_3) = \eta_3$ and $\xi_3(\infty)=0$. For $i=1$ or $2$, let $\alpha_i$ be a section of $u_i T\p$ with $\alpha_i(\infty) = \xi_3(z_i)$. For $i=1$ or $2$, let $Z_i=\{\infty\}$ viewed as a subset of the domain of $u_i$. As we observed when considering Type 3 curves, the $D^{Z_i}_{u_i}$ are surjective and so we can find a section $\xi_i \in \Gamma(u_i^*T\p)_{Z_i}$ with $D_{u_i} \xi = \eta_i +D_{u_i} \alpha_i$. Thus $D_{u_i}(\xi_i+\alpha_i) =\eta_i$ and $(\xi_i+\alpha_i)(\infty) =  \xi_3(z_i)$. Hence $(\xi_1+\alpha_1, \xi_2+\alpha_2, \xi_3) \in \Gamma(u^* T \p)_Z$ and $D_u^Z(\xi_1+\alpha_1, \xi_2+\alpha_2, \xi_3)= (\eta_1,\eta_2,\eta_3)$. This completes the proof that $D_u^Z$ is onto for $u$ of Type $4$.

\end{proof}

Let $J_t$ be a family of almost complex structures on a manifold $\p$. Let $\bar{\mathfrak{M}} = \bigcup_t \bar \M^{\z}_0(\p,J_t,2)$ and $\mathfrak{M} = \bigcup_t  \M^{\z}_0(\p,J_t,2)$.  There is a natural map $\pi: \bar{ \mathfrak{M} }\m \R$ which sends elements of $\bar \M^{\z}_0(\p,J_t,2)$ to $t$. Theorem \ref{theorem:zero} and Proposition \ref{theorem:deg2orbifold} combine to show that $\pi$ is a submersion. Additionally, Theorem \ref{theorem:zero} gives special charts on $\bar \MM $ and describes the local structure of the inclusion $\MM \subset \bar \MM$ in these charts. Using these charts, we can prove the following theorem.

\begin{theorem}
The orbifold equivalence class of $\M^{\z}_0(\p,J,2)$ is independent of $J$.
\label{flow}

\end{theorem}

\begin{proof}
Assume that we are interested in comparing the moduli spaces $J_0$ and $J_1$ holomorphic curves. Let $J_t$ be a path connecting $J_0$ and $J_1$. In order not to discuss orbifolds with boundary, we will define $J_t$ for $t<0$ and $t>1$.

 Let $\{ (U_i,G_i,\phi_i) \}$ be the orbifold charts on $\bar{ \mathfrak M}$ coming from Theorem \ref{theorem:zero}. Let $\{ \rho_i \}$ be a partition of unity subordinate to $\{ (U_i,G_i,\phi_i) \}$. Let $\chi_i$ be the vector field on $U_i$ induced from the product of the zero vector field on  $\C^m \times \C^d/Aut(\Sigma)$ and the unit speed vector field on $\R$. Let $b:\bar \MM  \m \R$ be a function such that $b(v)=0$ if $\pi(v) \notin [-1,2]$ and  $b(v)=1$ if $\pi(v) \in [0,1]$. Let $\chi=b \Sigma \rho_i \chi_i$. By Gromov Compactness, $\chi$ has compact support.  Thus, $\xi$ induces a flow $\phi_t : \R \times \bar \MM \m \bar \MM$. Since $\chi$ projects to a unit length vector field on $[0,1]$, for $t \in [0,1]$, $\phi_t$ restricts to a map $\phi_t: \bar \M^{\z}_0(\p,J_0,2) \m \bar \M^{\z}_0(\p,J_t,2)$. The flow induces orbifold equivalences. Since each vector field $\chi_i$ is parallel to the loci of singular curves $\bar \MM -\MM$, so is $\chi$. Hence, $\phi_t$ restricts to an equivalence between $\M^{\z}_0(\p,J_0,2)$ and $\M^{\z}_0(\p,J_t,2)$.

\end{proof}

\begin{theorem}
The diffeomorphism type of $\Hol^*_2(\s,(\p,J))$ is independent of $J$. Moreover, there is a diffeomorphism $\Phi : \Hol^*_2(\s,(\p,J_0)) \m \Hol^*_2(\s,(\p,J)$ making the following diagram homotopy commute:

$$
\begin{array}{ccccccccl}
\Hol_2^*(\Sigma_g, (\C P^2,J_0))  &\overset{\Phi}{\m} & \Hol_2^*(\Sigma_g, (\C P^2,J)) \\
    &  \searrow i & \downarrow i    &            \\

&  & \Omega^2 \p. \\

\end{array}$$

\label{theorem:deg2}
\end{theorem}

\begin{proof}
Take a one parameter family of almost  complex structures $J_t$ as in the proof of the previous theorem. The space $\Hol^*_2(\s,(\p,J_t))$ is a principle $\Hol^*_1(\s, \s)$ orbibundle over $\MM = \bigcup_t \M^{\z}_0(\p,J_t,2)$. Pulling back along the equivalence $\phi_t$ gives a one parameter family of orbibundles over $\M^{\z}_0(\p,J_0,2)$. Equivalence classes of $G$-orbibundles over an orbifold $X$ are in bijection with homotopy classes of map from $BX$ to $BG$ (Theorem \ref{theorem:BG}). Since $\R$ is connected, the orbibundle type and hence the orbifold equivalence class of $\Hol^*_2(\s,(\p,J_t))$ is independent of $t$. Since the category of smooth manifolds is a full subcategory of the category of smooth orbifolds, and since the total spaces of these orbibundles are manifolds not orbifolds, we have that the diffeomorphism type of $\Hol^*_2(\s,(\p,J))$ is independent of $J$. 

To see that the above diagram homotopy commutes, consider the space $\bigcup_t \Hol^*_2(\s,(\p,J_t))$. The restriction to $t=0$ or $t=1$ are the two inclusion maps in the above diagram. The fact that the inclusion map is defined over all of $\bigcup_t \Hol^*_2(\s,(\p,J_t))$ gives the homotopy.

\end{proof}

\begin{proposition}
 The evaluation at infinity map $ev: \Hol_2(\s, (\p,J)) \m \p$ is a smooth fiber bundle map.
\label{theorem:evaluation2}
\end{proposition}

\begin{proof}

This proof is exactly the same as the proof of Theorem \ref{theorem:evaluation1} given in \cite{Mi1}, except one needs to use arguments similar to those of  Theorem \ref{flow} and Theorem \ref{theorem:deg2} to deal with the lack of properness and orbifold points. 

Since the group of holomorphic automorphisms of $(\p,J_0)$ acts transitively on $\p$, $ev: \Hol_2(\s, (\p,J_0)) \m \p$ is a smooth fiber bundle map. We will use automatic transversality arguments to show  $ev: \Hol_2(\s, (\p,J)) \m \p$ is fiberwise diffeomorphic to $ev: \Hol_2(\s, (\p,J_0)) \m \p$ and hence is also a smooth bundle map.

Let $\M_0^{(\infty)}(\p,J,2) = \Hol_2(\s,(\p,J))/\Hol^*(\s,\s)$ with the base point condition defining $\Hol^*(\s,\s)$ being $u(\infty)=\infty$ and let $\bar \M_0^{(\infty)}(\p,J,2)$ be its compactification by stable maps. Let $J_t$ be a path of almost complex structures connecting $J_0$ to $J$ and let: $$\mathfrak H=\cup_t \Hol_2(\s,(\p,J_t))$$ $$\mathfrak M= \cup_t \M_0^{(\infty)}(\p,J_t,2) \text{ and}$$ $$\bar {\mathfrak M}=\cup_t \bar \M_0^{(\infty)}(\p,J_t,2).$$ 

The map $\mathfrak H \m \mathfrak M$ is a one parameter family of principle orbibundles with fiber $\Hol_1^*(\s,\s)$. The transversality condition needed to establish this is the same as that needed to show $\Hol_1^*(\s,\s)$ is a manifold, namely that $c_1(T\s)=2>0$. By the same arguments as those in the proof of Theorem \ref{theorem:deg2}, we see that the isomorphism type of the bundle $\Hol_2(\s, (\p,J_t)) \m \M_0^{(\infty)}(\p,J_t,2)$ is independent of $t$. 

By Gromov compactness, the map $\bar{\mathfrak M} \m \p \times \R$ is proper. It is a submersion, by the automatic transversality calculation done in Theorem \ref{theorem:deg2orbifold}. Using the arguments of Theorem \ref{flow}, we conclude that the  isomorphism types of the bundles $\bar \M_0^{(\infty)}(\p,J_t,2) \m \p$ and $ \M_0^{(\infty)}(\p,J_t,2) \m \p$ are independent of $t$. 

The rest of the arguments follow those of the proof of Theorem 2.17 of \cite{Mi1}. We use the above results to show that the fiber diffeomorphism type of $\Hol_2(\s, (\p,J_t)) \m \p$ is independent of $t$ and use this to conclude that $ev:\Hol_2(\s, (\p,J)) \m \p$ is a fiber bundle.

\end{proof}

 \section{Construction of a holomorphic gluing map}\label{secgluing}

In Section \ref{secTCH}, we discussed a scanning map which, in particular, gives a map $s: \int_{\Sigma_g-pt} R_2 \m \Map^*(\Sigma_g, B^2 B(D_2,D_2,R_2)).$ In Section \ref{KY}, we noted that $B^2 B(D_2,D_2,R_2) \simeq \p$. Moreover, we showed that a map $h:R_2 \m  \Omega^2 \p$ satisfying certain conditions induces a map: $$ \int_{\Sigma_g-pt} R_2 \m \Map^*(\Sigma_g,\p)$$ which is a homology equivalence in an explicit range. In Section \ref{secR2J}, we will define a map $h:R_2 \m \Map^*(\s,\p)$ which lands in a neighborhood of $\Hol^*(\s,(\p,J))$. In Section \ref{seccompair} we describe how to interpret $h$ a map of partial $D_2$-algebras up to homotopy $h:R_2 \m \Omega^2 \p$. Given this, it is natural to ask, is there a gluing map: $$g : \int_{\Sigma_g-pt} R_2  \m \Hol^*(\Sigma_g,(\p,J)).$$  This is the goal of this section. We will need some modifications primarily because of two issues. Due to details involving analysis, we will not be able to construct such a gluing map on all of $\int_{\Sigma_g-pt} R_2 $ but only on compact subsets. Due to transversality issues, we will need to use a construction that involves bundles of $D_2$-algebras over a manifold instead of just one fixed $D_2$-algebra.

\subsection{Approximate gluing maps and the implicit function theorem}\label{secapp}

In this subsection, we discuss a gluing construction from \cite{S} which is a generalization of a construction which appears in \cite{MS}. Let $\Sigma \in \bar{ \M}_g$ be a nodal curve. Suppose that $K$ is a compact set and $f: K \m \Hol(\Sigma,(M,J))$. If certain linearized $\dbar$-operators are surjective, Sikorav in \cite{S} described how to construct a map $\hat f: K \m \Hol(\Sigma_g,(M,J))$ which can be thought of as a desingularization of the map $f: K \m \Hol(\Sigma,(M,J))$. This is done by first constructing a map $f_a: K \m \Map(\Sigma_g,M)$ such that $f_a(\kappa)$ is almost a $J$-holomorphic map in an appropriate sense. One proves that $coker D_{f_a(\kappa)}=coker D_{f(\kappa)}$ and $coker \tilde D_{f_a(\kappa)}=coker \tilde D_{f(\kappa)}$. In our situation, the surjectivity of these operators can be checked via automatic transversality arguments. Then one applies an implicit function theorem to correct $f_a$ and build the map $\hat f: K \m \Hol(\Sigma_g,(M,J))$. We start off by stating implicit function theorems for $D_u$ \cite{MS} and $\tilde D_u$ \cite{S}.

We will state a version of the implicit function theorem that not only describes when a map $u$ can be corrected to a $J$-holomorphic map $\hat u$, but also gives an estimate for the distance between $u$ and $\hat u$. Estimates for the distance between $\hat u$ and $u$ depend on bounds for the norm of the partial right inverse of $\tilde D_u$. Extend $\tilde D_u$ to a map $\tilde D_u : W^{1,p}(\Gamma(u^*TM)) \m L^p(\Omega^{0,1}(u^*TM))$. Here $W^{1,p}(\cdot)$ or $L^p(\cdot)$ denotes the completion of the vector space with respect to the indicated norm. In our situation, this requires picking a metric on the surface. Suppose $Q_u: L^p(\Omega^{0,1}(u^*TM)) \m W^{1,p}(\Gamma(u^*TM))$ is a right inverse to $D_u$. A bound on the norm of $Q_u$ measures how surjective $D_u$ is in the following sense. Given two Banach spaces $X$ and $Y$, let $\mathcal L(X,Y)$ denote the Banach space of bounded linear transformation with the operator norm. By a Neumann series argument, one can show that if $A \in\mathcal L(X,Y)$ and $B \in \mathcal L(Y,X)$ with $A \circ B =id_Y$, then if $\lVert C - A \rVert <  \lVert B\rVert$, then $C$ is also surjective. Also, note that if $D_u$ is surjective, then it will have a bounded partial right inverse. In \cite{MS} (Theorem 3.3.4), they prove the following theorem in the case that $Z = \emptyset$. The general case follows by similar arguments.

\begin{theorem}
Fix a finite set $Z \subset \Sigma_g$ and let $p>2$. For every $c_0>0$, there exists $\delta>0$ and $c>0$ (only depending on $c_0$) such that the following hold. Fix a metric on $\Sigma_g$ with volume less than $c_0$. Fix an almost complex structure $j$ on $\Sigma_g$  and let $u: \Sigma_g \m M$ be smooth. Assume that $D^Z_u: W^{1,p}( \Gamma(u^*TM)_Z) \m L^p (\Omega^{0,1}(u^*(TM))$ is surjective and let $Q_u:L^{1,p} (\Omega^{0,1}(u^*(TM)) \m  W^{1,p}(\Gamma(u^*TM)_Z) $ be a right inverse. Assume that $ \lVert Q_u \rVert \leq c_0$, $ \lVert du \rVert_{L^p} \leq c_0$ and $\lVert \dbar^{nl} u \rVert_{L^p} \leq \delta$.  Under these circumstances, there exists $\xi \in \Gamma(u^* TM)_Z$ such that the function $\hat u$ defined by the following formula is $(j,J)$-holomorphic: $$\hat u(z) = exp_u(z)(\xi_z).$$ Also, $ \lVert \xi \rVert_{W^{1,p}} \leq c \lVert \dbar^{nl} u \rVert_{L^p} \leq \delta c_0 $. 
 \label{theorem:implicitquantR}
\end{theorem}

The analogous theorem for the operator $\tilde D_u^Z$ is implicit in Step 4 of the proof of Theorem 1 of \cite{S}.

\begin{theorem}
Fix a finite set $Z \subset \Sigma_g$ and let $p>2$. For every $c_0>0$, there exists $\delta>0$ and $c>0$ (only depending on $c_0$) such that the following hold. Fix a metric on $\Sigma_g$ with volume less than $c_0$. Fix an almost complex structure $j$ on $\Sigma_g$  and let $u: \Sigma_g \m M$ be smooth. Assume that $\tilde D^Z_u: W^{1,p}( \Omega^{0,1}(T^*\Sigma_g)_Z \times \Gamma(u^*TM)_Z) \m L^p (\Omega^{0,1}(u^*(TM))$ is surjective and let $Q_u:L^{1,p} (\Omega^{0,1}(u^*(TM)) \m  W^{1,p}(\Omega^{0,1}(T^*\Sigma_g)_Z \times \Gamma(u^*TM)_Z) $ be a right inverse. Assume that $ \lVert Q_u \rVert \leq c_0$, $ \lVert du \rVert_{L^p} \leq c_0$ and $\lVert \dbar^{nl} u \rVert_{L^p} \leq \delta$.  Under these circumstances, there exists $\xi \in \Gamma(u^* TM)_Z$ and almost complex structure $j'$ on $\Sigma_g$ such that the function $\hat u$ defined by the following formula is $(j',J)$-holomorphic: $$\hat u(z) = exp_u(z)(\xi_z).$$ Also, $ \lVert \xi \rVert_{W^{1,p}} \leq c \lVert \dbar^{nl} u \rVert_{L^p} \leq \delta c_0 $. The almost complex structure $j'= exp_j (\digamma)$ for some $\digamma \in \Omega^{0,1}(T \Sigma_g)$.

 \label{theorem:implicitquant}
\end{theorem}

Next we turn towards constructing gluing maps. Consider the following situation. Let $C_i=(\Sigma_{g_i},j_{i})$ be complex curves for $i=1$ and $2$. Let $z_i \in \Sigma_{g_i}$ and $u_i:C_i \m M$ be $J$-holomorphic maps with $u_1(z_1)=u_2(z_2)$. Let $e_i: \D \m C_i$ be holomorphic embeddings with $e_i(0)=z_i$. Here $\D$ is the open unit disk in $\C$ with the standard complex structure. Fix metrics on each surface $\Sigma_{g_i}$. In \cite{S}, they describe how to construct a smooth complex curve $\hat C$ of genus $g_1+g_2$ with a $J$-holomorphic map $\hat u:\hat C \m M$. For the remainder of this subsection, we will describe their construction. We consider gluing just two curves for simplicity of notation and clarity but the construction generalizes to allow the gluing of an arbitrary finite number of curves in compact families. In future sections we will only be interested in the case where all but one of the curves are rational.

Let $\D_t \subset \C$ be the disk of radius $t$ centered at the origin. We will always assume $t \in (0,1)$. Let $A_t \subset \D^2$ be the subspace of points $(x,y)$ with $xy=t$. The space $A_t$ is topologically an annulus with a natural complex structure and metric induced by $\C^2$. Let $C_t = \left( ( C_1-e_1(\D_{\sqrt t}) )\cup ( C_2-e_2(\D_{\sqrt t}) ) \cup A_t \right)  / \sim$. Here $\sim$ is the relation that $(x,y) \sim e_1(x)$ if $|x|>|y|$ and $(x,y) \sim e_2(y)$ if $|y|>|x|$. Denote the inclusion map $A_t \m C_t$ by $\psi$. The space $C_t$ is a surface of genus $g_1+g_2$ and has a natural complex structure coming from the $C_i$ and $A_t$.

To construct a $J$-holomorphic map $u_t:C_t \m M$, we first construct an approximate $J$-holomorphic map $u^t_a:C_t \m M$ such that $lim_{t \m 0} \lVert \dbar u_a^t \rVert_{L^p} =0$. The function $u^t_a$ will agree with $u_i$ on the image of $C_i-e_i(\D_1)$ and we will use cutoff functions to define the function on the rest of the surface. Fix $p : [0,\infty) \m [0,1]$ a smooth function with $p(s)=0$ for $s \leq 1$ and $p(s) = 1$ for $s \geq 2$. Let $C$ be the nodal curve which is the wedge of $C_1$ and $C_2$ at the points $z_1=e_1(0)$ and $z_2=e_2(0)$ and let $u:C \m M$ be the map induced by the $u_i$'s. Let $P_t:C_t \m C$ be defined as follows. For points outside of $e_i(\D)$, define $P_t$ by the map $C_t- \bigcup e_i(\D) \m C_1 \cup C_2 \m C$. For $(x,y) \in A_t$, define $P_t(\psi(x,y))$ by the formula: $$P_t(\psi(x,y))=
\begin{cases}
e_1(p(t^{-1/4}|x|)x)  & \mbox{if } |x| \geq |y| \\
e_2(p(t^{-1/4}|y|)y)  & \mbox{if } |y| \geq |x|. \\
\end{cases}
$$ Define the map $u^t_a:C_t \m M$ by the formula $u^t_a =u \circ P_t$. The map $u^t_a$ is not $J$-holomorphic, but in Section 3.1 of \cite{S} they proved the following.

\begin{theorem}

Fix metrics on the curves $C_1$ and $C_2$ and consider $C_t$ to have these metric outside of the images of $e_i$ and on $A_t$ we use the metric induced by $\C^2$. For any maps $u_i:C_i \m M$ with $u_1(z_1)=u_2(z_2)$, we have $\lVert  \dbar^{nl}  u_a^t \rVert_{L^p} = \lVert  \dbar^{nl} u_1 \rVert_{L^p}+ \lVert  \dbar^{nl} u_2 \rVert_{L^p}+ O(t^{\frac{1}{2p}})$. Moreover, this estimate is uniform in families whose first derivate is uniformly point-wise bounded within the images of the embeddings. 

\label{theorem:smalldbar}
\end{theorem}

In particular, if $u_1$ and $u_2$ are $J$-holomorphic, then $\lVert \dbar^{nl} u_a^t \rVert_{L^p}$ can be made arbitrarily small. This is the sense in which $u^t_a$ is approximately $J$-holomorphic.

\begin{remark}
\label{iteratesmalldbar} 

Later in this section we will want to iterate this gluing construction. To apply the above theorem, we need to observe that in fact the first derivative of $u^t_a$ is bounded as $t$ goes to zero. The function $u_a^t$ is equal to one of the $u_i$'s or constant except on the annulus $t^{1/4} \leq |x| \leq 2t^{1/4}$ and on the annulus with $x$ replaced with $y$. On this annulus, $Du_a^t = O(t^{-1/4}|x|)=O(1)$. This bound depends on the cutoff function $p$ and on the $u_i$ but not on $t$.

\end{remark}

In order to apply the implicit function theorem, Theorem \ref{theorem:implicitquantR} or Theorem \ref{theorem:implicitquant}, we also need a way to compute $coker D_{u_a}$ or $coker \tilde D_{u_a}$. In \cite{S}, they proved the following theorem.

\begin{theorem}
For any finite subset $Z \subset C_t -A_t$, there is an isomorphism between the cokernels of $\tilde D^Z_{u_a}$ and $\tilde D^Z_{u}$ for sufficiently small $t$.
\label{theorem:aprxcoker}
\end{theorem}

When $\dim_{\R} M =4$, we can often use automatic transversality techniques (Theorem \ref{theorem:auto}) to prove that $\tilde D_u$ is onto. This then implies that $\tilde D_{u_a^t}$ is onto and thus, by Theorem \ref{theorem:implicitquant}, the map $u^t_a$ can be corrected to produce an actual $J$-holomorphic map. A similar theorem is proved in \cite{MS} relating $D_u$ and $D_{u^t_a}$. If $\tilde D_u$ is surjective, then in \cite{S}, they construct a uniform bound on the norms of partial right inverses of $\tilde D_{u_t}$.

\begin{theorem}
If $\tilde D_u$ is onto, then for sufficiently small $t$, there exists a number $c_0$ such that $\tilde D_{u_t}$ has a right inverse with norm less than $c_0$.
\label{theorem:uniformrightinverse}
\end{theorem}

 Combining Theorems \ref{theorem:implicitquant}, \ref{theorem:smalldbar}, \ref{theorem:aprxcoker} and \ref{theorem:uniformrightinverse}, we get the following theorem.

\begin{theorem}
If $C$ is a nodal curve, $u:C \m M$ is $J$-holomorphic and $\tilde D_u$ is onto, then for sufficiently small $t$, there exists an almost complex structure $j'$ on $C_t$ and section $\xi \in \Gamma(u_a^{t*} TM)$ such that $u_t : C_t \m M$ defined by $u_t(z) = exp_{u_a^t(z)}(\xi_z)$ is $(j',J)$-holomorphic.
\label{theorem:gluing}
\end{theorem}

In \cite{S}, Sikorav proved this theorem and generalized this theorem to compact families of maps and embeddings as well as to the case of the intersection of more than two curves. If $Z \subset C$ is a finite set and $\tilde D^Z_u$ is onto, we additionally require that the section $\xi$ be in $\Gamma(u_a^{t*} TM)_Z$ and thus $u_t(z)=u(z)=u_a(z)$ for all $z \in Z$.

Combining Theorems \ref{theorem:implicitquant}, \ref{theorem:aprxcoker} and \ref{theorem:uniformrightinverse}, we get the following corollary.

\begin{corollary}
If $u$ is a $J$-holomorphic map from a nodal curve and $\tilde D_u$ is onto, then there is a $t_0>0$ such that for all $t<t_0$ and for all $s \leq 1$, $\tilde D_{ exp_{u_a^t}(s\xi)}$ is onto. In particular, $\tilde D_{u_t}$ is onto.

\label{corollary:surjectiveopen}
\end{corollary}

\begin{proof}
By Theorem \ref{theorem:aprxcoker}, there is a $t_0>0$ and $R>0$ such that for all $t<t_0$, the operator $\tilde D_{u^t_a}$ is onto and $\tilde D_{u_t}$ has a right inverse of norm less than $R$. Here the norm is the operator norm with the domain of $\tilde D_{u^t_a}$ completed with respect to the $W^{1,p}$ norm and the range completed with respect to the $L^p$ norm and $p>2$. Using parallel transport, for $v$ near $u^t_a$, we can view $v$ and $u^t_a$ as linear transformations between the same Banach spaces \cite{MS}. The operator $\tilde D_v$ depends continuously on $v$ with respect to the $W^{1,p}$ topology. Thus, for $v$ sufficiently close to $u_t$, $\tilde D_v$ is onto. By Theorem \ref{theorem:implicitquant}, by choosing a possibly smaller $t_0$, we can make $\xi$ small enough so that $\lVert \tilde D_{u_a^t} - \tilde D_{ exp_{u_a^t}(s\xi)} \rVert < 1/R$ for any $t<t_0$ and $0 \leq s \leq 1$. Thus $\tilde D_{ exp_{u_a^t}(s\xi)}$ is also onto since $\tilde D_{u_a^t}$ has a right inverse with norm less than $R$.
 
\end{proof}

All of these theorems above are generalizations of theorems from \cite{MS}. To get the original theorem of \cite{MS}, replace the assumption that $\tilde D_u$ is onto with $D_u$. In this situation, one can use the implicit function theorem for $D_u$ which does not involve perturbing the almost complex structure. For example, Theorem \ref{theorem:gluing} is a reformulation of the following theorem from \cite{MS}.

\begin{theorem}
Let $C$ be a complex nodal curve and let $C_t$ be the associated smooth complex curve with complex structure. If $D^Z_u$ is onto for a finite set $Z \subset C$ disjoint from the nodes, then for sufficiently small $t$, there exists a section $\xi \in \Gamma(u_a^{t*} TM)_Z$ such that $u_t : C_t \m M$ defined by $u_t(z) = exp_{u_a^t(z)}(\xi_z)$ is $J$-holomorphic.
\label{theorem:gluingMS}
\end{theorem}

\subsection{A genus zero approximate gluing map}\label{secR2J}

The goal of this subsection is to construct a map  $h : R_2 \m \Map^*(\s,\p)  \times Met^{pc}(\s)$. Here $Met^{pc}(S)$ denotes the space (topologized with the trivial topology) of piecewise continuous (continuous off the image of a finite number of smoothly embedded arcs) metrics on a surface $S$.  Let $h^m$ denote the projection of $h$ onto $Met^{pc}(\s)$ and let $h^f$ be the projection of $h$ onto $\Map^*(\s,\p)$. The desired property of the map $h$ is that for all $\kappa \in R_2$, $h^f(\kappa)$ will be approximately $J$-holomorphic so that in Section \ref{correct} one can correct this and build a map that lands in the space of $J$-holomorpic maps to $\p$. We measure the failure of $h^f(\kappa)$ to be holomorphic using the metric $h^m(\kappa)$. Recall that the symbol $\s$ denotes the sphere with the standard complex structure $j_0$. Using an identification of $\Map^*(\s,\p)$ with $\Omega^2 \p$, in Section \ref{seccompair} we will see that $h^f$ is a map of partial $D_2$-algebras up to homotopy so Corollary \ref{KYpackage} will apply to $h^f$.

\paragraph{Modified definition of $R_2$:} Recall that $R_2$ was defined in Definition \ref{defR2}. We describe some modifications to $R_2$ which do not change the homotopy type but facilitate the construction of gluing maps. First, we will pick a subset $D'_2(2) \subset D_2(2)$ such that:

-$D'_2(2)$ is $\Sigma_2$-invariant

-the inclusion $D'_2(2) \m D_2(2)$ is an equivariant homotopy equivalence

-$D'_2(2)$ is compact.

\noindent  For concreteness, we take  $D'_2(2)$ to be the family of pairs of disks of radius $1/5$ with centers a distance of $1/2$ away from the origin and distance of $1$ away from each other. From now on, we slightly modify the definition of $R_2$ replacing $D_2(2)$ with $D_2'(2)$ both in the definition of the space and in the definition of the space of compositions. The compactness of $D_2'(2)$ will be relevant for analysis later. Since the inclusion $D_2'(2) \m D_2(2)$ is a homotopy equivalence, all of the results of Section \ref{KY} still apply.

\paragraph{Definition of $h$ on $\{r_0,r_1\}$:}

We send $r_0$ to the constant map with value the base point and round metric $\gamma_0$ on $\s$. Once and for all, fix an element $q_1 \in \Hol_1^*(\s,(\p,J)) \subset \Map^*_1(\s,\p)$ and let $h(r_1)=(q_1,\gamma_0)$.

\vspace{12pt}

View the cone $Cone(D_2'(2)/\Sigma_2)$ as the quotient of $D_2'(2)/\Sigma_2 \times [0,1]$ where we identify all points of the form $(\kappa,1)$. To define $h$ on the $Cone(D_2'(2)/\Sigma_2)$ we will define $h^f$ on $D'_2(2)/\Sigma_2 \times \{0\}$ to be an approximate gluing map using the formulas of the previous section. In the interval $[0,1/3]$, we will correct this approximate gluing map to an actual gluing map so that $D'_2(2)/\Sigma_2 \times \{1/3\}$ is mapped to the subspace of holomorphic functions. Since the space of degree two holomorphic functions is simply connected, the map $D'_2(2)/\Sigma_2 \times \{1/3\} \m \Hol_2^*(\s,(\p,J))$ can be extended to the rest of the cone.

\paragraph{Definition of $h$ on $D'_2(2)/\Sigma_2 \times \{0\}$:}

Let $p_0 \in \p$ be the base point. An element of $D'_2(2)/{\Sigma_2}$ is two embeddings $e_i:\D \m \D \subset \s$. We write these embeddings as if they were ordered; however, our construction will never use this and thus will be well defined on the quotient. Let $u^{const}: \s \m \p$ be the constant map. Let $C$ be the nodal curve which is a union of three copies of $\s$ attached at two points. Call the $\s$ with two copies of $\s$ attached to it the root and call the other copies of $\s$ the leaves. Attach $\infty$ in the $i$th leaf to $e_i(0)$ in the root. The maps $q_1,q_1,u^{const}$ assemble to form a map $u: C \m \p$. Let $e': \D \m \s$ be the map $z  \m 1/z$ and let $e'_i: \D \m \s$ be given by the same formula as $e'$ but we now view $\s$ as the $i$th root. The collection of embeddings $e_1,e_2,e'_1,e'_2$ give neighborhoods of the nodes. Equip the three spheres with the metric $\gamma_0$.

We can build a surface with complex structure and metric $C_t$ and a map $u_a: C_t \m \p$ as in the previous subsection. The curve $C_t$ is genus zero and hence is diffeomorphic to $\s$. In fact there is a natural holomorphic isomorphism $d(e_1,e_2):\s \m C_t$ which we will describe shorty. Then we will define the approximate gluing map $h^a:D_2'(2)/\Sigma_2 \m \Map_2^*(\s,\p)$ via the formula $h^a(e_1,e_2)=u_a \circ d(e_1,e_2)$.

\begin{figure}[!ht]
\begin{center}\scalebox{.4}{\includegraphics{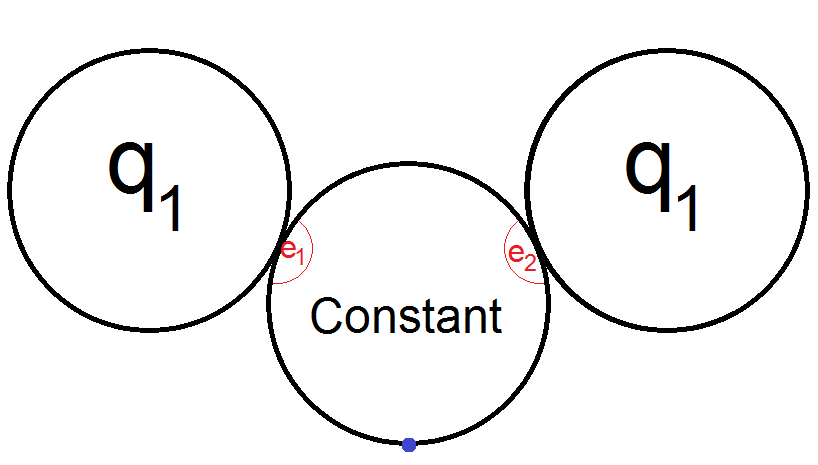}}\end{center}
\caption{A nodal curve associated to an element of $D_2'(2)/\Sigma_2$}
\label{figure:mickey}
\end{figure}

Recall that $A_t$ is the annulus in $\D^2$ of points $(x,y)$ with $xy=t$ and $\D_r$ is the disk of radius $r$ around the origin. The curve $C_t$ is defined to be the quotient of the space: $$(\s - e_1(\D_{\sqrt t}) -e_2(\D_{\sqrt t}) \cup ( \bigsqcup_{i=1}^2 \s -e'_i(\D_{\sqrt t})) \cup  \bigsqcup_{i=1}^2 A_t$$ by the following relation. For $(x,y)$ in the $i$th $A_t$, we identify $(x,y)$ with $e_i(x)$ in the root if $|x|>|y|$ and we identify $(x,y)$ with $e_i'(y)$ if $|y|>|x|$. 

Let $\alpha_i: \D_{\sqrt t} \m C_t$ be the map $z \m z/t$ composed with the inclusion of $\C$ into the $i$th leaf (that is, $\alpha_i$ has image the $i$th copy of $\s -e'_i(\D_{\sqrt t})$). Let $\iota:\s - e_1(\D_{\sqrt t}) -e_2(\D_{\sqrt t}) \m C_t$ be the inclusion of the root and define $d(e_1, e_2):\s \m C_t$ by the following formula:  $$d(e_1,e_2)(z) =
\begin{cases}
\iota(z),  & \mbox{if } z\notin e_i(\D_{\sqrt t}) \text{ for each } i \\

\alpha_i(e_i^{-1}(z)),  & \mbox{if } z \in e_i(\D_{\sqrt t}). \\

\end{cases}
$$ 

Note that this is a holomorphic isomorphism. The approximate gluing map $h^a(e_1,e_2)$ sends points outside the disks $e_i(\D_{t^{3/4}})$ to the base point and points inside the disks to rescaled versions of the map $q_1$. The image of $h^a(e_1,e_2)$ does not depend on $t$.

On $D'(2)/\Sigma_2 \times \{0\}$, define $h^f((e_1,e_2),0)$ to be $h^a(e_1,e_2)$. We define $h^m((e_1, e_2),0)$ to be the pullback of the metric on $C_t$ under the map $d(e_1, e_2)$. Here $C_t$ is equipped with the metric induced from $\C^2$ on the annuli and the metric $\gamma_0$ on the spheres minus the disks.

\paragraph{Definition of $h$ on $D'_2(2)/\Sigma_2 \times [0,1/3]$:}

For $s \leq 1/3$ and $\kappa \in D'_2(2)/\Sigma_2$, let $h^m(\kappa,s)=h^m(\kappa,0)$. The definition of $h^f(\kappa,s)$ involves correcting $h^a$ to land in the subspaces of holomorphic maps.

By Theorem \ref{theorem:gluingMS}, we can perturb $h_a$ to a map to $ \Hol^*(\s , (\p, J))$ if the operator $D^{Z}_u$ is onto.  Here $u$ is the map from the nodal curve induced by $q_1$ on the leaves and the constant map on the root and $Z$ is the set containing $\infty$ in the root $\s$. By Proposition \ref{theorem:deg2orbifold}, the operator $D^Z_u$ is onto for these types of curves. Thus, we have the following proposition.

\begin{proposition}
\label{correcth}
For $t$ sufficiently small, there is a map $\xi_t : D_2'(2)/\Sigma_2 \m \Gamma( T C^{\infty}(\s,\p))$ with the property that for $\kappa \in D_2'(2) / \Sigma_2 $, the following map is $J$-holomorphic: $$z \m \exp_{h^a(\kappa)(z)} (\xi_t(\kappa)(z)).$$
\end{proposition}

We define $h^f(\kappa,s)$ by $h^f(\kappa,s)=\exp_{h^a(\kappa)(z)} (3s\xi_t(\kappa)(z))$. This agrees with our previous definition when $s=0$. When $s=1/3$, $h^f(\kappa,s)$ is $J$-holomorphic by Proposition \ref{correcth}.

\paragraph{Definition of $h$ on $D'_2(2)/\Sigma_2 \times [1/3,2/3]$:}

We would like our construction to depend on continuously $t$. Moreover, we would like some explicit control as $t$ tends to $0$. Thus, we cannot just arbitrarily extend the function to the rest of the cone. Instead, we define $h$ on  $D'_2(2)/\Sigma_2 \times [1/3,2/3]$ so that it does not depend on $t$ on  $D'_2(2)/\Sigma_2 \times \{2/3 \}$ and then once and for all define $h$ on $D'_2(2)/\Sigma_2 \times [2/3,1]$ by extending arbitrarily using simply connectivity of $\Hol^*_2(\s,(\p,J))$. 

Pick $t_0$ sufficiently small so that the conclusions of Proposition \ref{correcth} apply. View $t_0$ as fixed once and for all as opposed to $t$ which we will send to zero later.  For $s$ in $[1/3,2/3]$ and $\kappa \in D'_2(2)/\Sigma_2 $, define $h^f(\kappa,s)$ by $h^f(\kappa,s)=\exp_{h^a(\kappa)(z)} (  \xi_{\tau(s)}(\kappa)(z))$ with $\tau(s)=(2-3s)t+(3s-1)t_0$. When $s=1/3$, this agrees with the definition already given. When $s=2/3$, this is identical to the case $s=1/3$ except with $t$ replaced by $t_0$.  

For $s \in [1/3,2/3]$, we define the metric $h^m(\kappa,s)$ to be identical to the metric $h^m(\kappa,0)$ except with the gluing parameter $t$ replaced with $\tau(s)$ in all formulas.

\paragraph{Definition of $h$ on the image of $D'_2(2)/\Sigma_2 \times [2/3,1] \}$:}

We now describe how to extend $h$ to the rest of the cone. The following proposition appears in the proof of Lemma 3.4 of \cite{Se}.

\begin{proposition}
The spaces $\Hol^*_k(\s,(\p,J_0))$ are simply connected for all $k$.
\end{proposition}

By Theorem \ref{theorem:deg2}, $\Hol^*_2(\s,(\p,J))$ is simply connected for all $J$. Since $D_2'(2 ) \cong S^1$,  we can extend extend $h^f$ over the rest of the cone in such a way that $h^f(\kappa,s)$ is $J$-holomorphic  for $s \geq 2/3$ (Note it is also $J$-holomorphic for $s \in [1/3,2/3]$). 

We extend $h^m$ to the rest of the cone as follows. We keep the metric constant outside the images of the embeddings $e_i$. Inside the images of the embeddings we change the metrics so that for $s=1$, the metric on $\s$ is $\gamma_0$. We require that this path of metrics consists of continuous metrics when restricted to the regions $e_i(\D)-e_i(\D_t)$ and  $e_i(\D_t)$. We also ask that the metrics vary continuously with respect to the usual topology on the space of metrics in these regions. Additionally, we require that this path of metrics stays commensurable to $\gamma_0$. In other words, there should be a constant $R>0$ such that for all tangent vectors $v,w$, $|h^m(\kappa,s)(v,w)| \leq R |\gamma_0(v,w)|$ and $|h^m(\kappa,s)(v,w)| \geq |\gamma_0(v,w)|/R$. We allow $R$ to depend on $t_0$. Note that such a family of metrics exist since $D_2'(2)$ is compact.

\paragraph{Bundle of partial algebras:}

Note that the map $h$ depends on a base point $p_0 \in \p$ and a choice of $q_1 \in \Hol^*_1(\s,\p)$. Due to transversality issues, we will not be able to glue rational maps onto the constant map from $\Sigma_g$ to $\p$. Instead, we will pick a $J$-holomorphic map $w:\Sigma_g \m \p$ and glue rational $J$-holomorphic maps to $w$. In order to glue two curves, they need to intersect. Thus, we need to let the base point vary. For $z \in \Sigma_g-pt$ and $\kappa \in R_2$, we denote the associated element of $\Map(\s,(\p,J)$ by $h^f_z(\kappa)$. The desired properties of $h$ are that $h^f_z(\kappa)(\infty)=w(z)$. That is, for each point $z$ in the curve, we want $R_2$ to map into the space of maps based at $w(z)$.

First, we define $h^f_z(r_0)$ to be the constant map at $w(z)$.

Note that all maps $w:\Sigma_g-pt \m \p$ are null-homotopic since $\Sigma_g-pt$ is homotopy equivalent to a wedge of circles and $\p$ is simply connected. Thus all bundles on $\Sigma_g-pt$ pulled back from $\p$ along $w$ are trivializable.

Given a map $u \in \Hol_1(\s,(\p,J))$, let $f_u:D_2'(2)/\Sigma_2 \m \Hol_2(\s,(\p,J))$ be the map which sends $(e_1,e_2) \in D_2'(2)/\Sigma_2$ to the element of $\Hol_2(\s,(\p,J))$ defined using the same procedure as was used to define $h^f((e_1,e_2),2/3)$ except with $q_1$ replaced with $u$. Let $ev:\Hol_1(\s,(\p,J)) \m \p$ be the bundle described in Theorem \ref{theorem:evaluation1} and let $k:K \m \p$ be a subbundle of $\Hol_1(\s,(\p,J)) \m \p$ such that $k$ is proper and the inclusion of the fibers of $k$ into the fibers of $ev$ are homotopy equivalences. Let $E \m \p$ be the bundle whose fiber over a point $p \in \p$ is the space of pairs $(u,g)$ with $u \in k^{-1}(p)$ and $g: Cone(D_2'(2)/\Sigma_2) \m \Hol_2(\s,(\p,J))$ satisfying:

-$g(e_1,e_2,s)(\infty)=p$

-$g(e_1,e_2,0)=f_u(e_1,e_2)$ 

\noindent for all $(e_1,e_2) \in D_2'(2)/\Sigma_2$ and $s \in [0,1]$. That is, the fiber at $p$ consists of degree one holomorphic maps $u$ based at $p$ as well as an extensions of $f_u$ to the cone. Note that the map $f_u$ depends on the gluing parameter $t_0$ and is only defined for $t_0$ sufficiently small. Since $\p$ is compact and $k$ is proper, we can pick one $t_0$ that works for all $u \in K$ and hence make sense of $E$. 

Fix a trivialization of $w^*E$ on $\Sigma_g-pt$ and use this trivialization to pick a section of $w^*E$. This in particular gives a section $\sigma$ of $w^*\Hol_1^*(\s,(\p,J))$. This allows us to define $h^f_z$ for $s \leq 2/3$ as was done in the non-parameterized setting. We have $f_{\sigma(z)}(e_1,e_2)=h^f_z(e_1,e_2,2/3)$. For $s >2/3$, we use the chosen extension of $f_{\sigma(z)}$ to the cone to define $h^f_z$. 

We do not need to have the metric depend on the point in $\Sigma_g-pt$.

\subsection{A holomorphic version of topological chiral homology}\label{secholTCH}

In this subsection, we slightly modify the space $\int_{\Sigma_g-pt} R_2$ to construct a new space $\int\limits^{hol}_{\Sigma_g -pt} (R_2,\bold j)$. We will refer to the space  $\int\limits^{hol}_{\Sigma_g -pt} (R_2,\bold j)$ as the holomorphic topological chiral homology of the manifold $\Sigma_g-pt$ with coefficients in $R_2$. Due to issues involving analysis, we will not be able to construct a gluing map on all of $\int\limits^{hol}_{\Sigma_g -pt} (R_2,\bold j)$ but only over any compact subset.

Since we are attempting to build a map to $\Hol^*(\Sigma_g,(\p,J))$, the space $D(\Sigma_g-pt)$ of smoothly embedded disks is not the correct space to consider. Instead we will be interested in a space of holomorphic embeddings of disks. Fix a complex structure $\bold j$ on $\Sigma_g$.

\begin{definition}
Let $\D(\Sigma_g -pt)$ be the sub-$\Sigma$-space of $D(\Sigma_g-pt)$ consisting of $\bold j$-holomorphic embeddings.
\end{definition}

Since elements of $D_2(k)$ are holomorphic, $\D(\Sigma_g -pt)$ has the structure of a right $D_2$-module. Note that for all $k$, the obvious map $\pi:\D(\Sigma_g -pt)(k) \m D(\Sigma_g-pt)(k)$ is a weak equivalence. We will define $\int\limits^{hol}_{\Sigma_g -pt} (R_2,\bold j)$ to be $B(\D(\Sigma_g -pt),D_2,R_2)$.

Just as in Section \ref{subsectwosided}, we can view elements of $B_k(\D(\Sigma_g -pt),D_2,R_2)$ as labeled rooted trees. This is done as follows:

1. The root of $T$ is labeled by an element of $\D(\Sigma_g-pt)$. In other words, the label of the root is a collection of $\bold j$-holomorphic embeddings $e_i: \D \m \Sigma_g-pt$ and paths of matrices.

2. The internal verticies are labeled by elements of $D_2$.

3. The leaves are labeled by elements of $R_2$.

Given a tree $T$, we will alway denote the components of $T$ minus the root by $T_i$. Each $T_i$ is a tree and has an associated embedding $e_i : \D \m \Sigma_g$. A labeled tree being an element of $B(\D(\Sigma_g -pt),D_2,R_2)$ is equivalent to the following condition. The labels of all of the internal verticies of $T_i$ along with the labels of the leaves form an element of $D_2^{\otimes k} \otimes R_2$ and we require that this element is in $Comp^k$. 

We see that one of two possible situations occurs for a subtree $T_i$. We could have that $T_i$ has one leaf, the valency of each internal vertex is two, and each internal vertex is labeled by $id \in D_2(1)$. In this case, the leaves are labeled by arbitrary elements of $R_2$. Otherwise, $T_i$ has two leaves. In this case, there is a unique trivalent internal vertex labeled by an element of $D_2'(2)$ and both leaves are labeled by $r_1$.

\vspace{12pt}

The space $\int\limits^{hol}_{\Sigma_g -pt} (R_2,\bold j)$ does not depend on the map $w:\Sigma_g \m \p$. However, using the map $w$ and the maps $h_z$, we can view the labels of the leaves of the trees of $\int\limits^{hol}_{\Sigma_g -pt} (R_2,\bold j)$ as elements of $\Map(\s,\p)$.

\subsection{Construction of an approximate gluing map for holomorphic topological chiral homology}\label{secJTCH}

In the previous subsection, we constructed a space $\int\limits^{hol}_{\Sigma_g -pt} (R_2,\bold j)$ which is the geometric realization of a simplicial space involving holomorphic embeddings of disks in $\Simga_g$ and elements of $R_2$ which we identify with degree one and two approximately $J$-holomorphic maps to $\p$ using the function $h$. In this subsection, we will describe a map: $$g_a : \int\limits^{hol}_{\Sigma_g -pt} (R_2,\bold j)\m \Map^*(\Sigma_g,(\p,J)) \times Met^{pc}(\Sigma_g)$$ which we call the approximate gluing map.  In the final subsection of this section, we will perturb this map using the implicit function theorem to produce a map to $\Hol^*(\Simga_g,(\p,J))$.
 
Before we proceed, we fix a metric $\gamma_g$ on $\Sigma_g$.

\begin{definition}
For $N \in \N$, let $F_N ( \int\limits^{hol}_{\Sigma_g -pt} (R_2,\bold j))$ denote the $N$-skeleton of $\int\limits^{hol}_{\Sigma_g -pt} (R_2,\bold j)$.

\end{definition}

As in the Section \ref{secR2J}, the approximate gluing map $g_a$ depends on a gluing parameter $t\in (0,1]$ which we suppress from the notation. On compact sets, we will be able to choose $t$ sufficiently small such that $g_a$ can be corrected to yield a map into the space of $J$-holomorphic maps. First we will define $g_a$ on the zero skeleton. For purposes of clarity we will define $g_a$ on the one skeleton before giving the general definition of the map. We will denote the restriction of $g_a$ to the $N$-skeleton by $g_N$. We denote the projection of $g_N$ to the space of maps by $g^f_N$ and the projection to $Met^{pc}(\Sigma_g)$ by $g^m_N$.

\textbf{Definition of $g_a$ on $F_0( \int\limits^{hol}_{\Sigma_g -pt} (R_2,\bold j))$}:

A point in the space $F_0(\int\limits^{hol}_{\Sigma_g -pt} (R_2,\bold j))$ can be viewed as a collection of holomorphically embedded disks labeled with rational degree 1 or 2 approximately $J$-holomorphic functions $u_i$ such that $w$ maps the center of the $i$th disk to $u_i(\infty)$ along with metrics on the domains of the maps $u_i$. To denote a collection of holomorphic embeddings $e_i: \D \m \Sigma_g$ and maps $u_i : \s \m \p$, we will use the shorthand notation $(\vec e, \vec u)$. Let $e':\D \m \s$ be the map $z \m 1/z$ and let $e'_i$ denote the map $e'$ where we view the range as the $i$th copy of $\s$. Given this data and the fixed metric $\gamma_g$ on $\Sigma_g$, in Section \ref{secapp} we described a construction of a surface of genus $g$ with complex structure and metric, $C_t$ and an approximate gluing map $u_a(\vec e, \vec u): C_t \m \p$. The diffeomorphism type of $C_t(e,\vec u)$ is always the same as $\Sigma_g$. Moreover, one can define a holomorphic isomorphism $d(\vec e): \Sigma_g \m C_{t}$ in the exact same way as in Section \ref{secR2J}. We use the obvious inclusion $\iota: \Sigma_g -\bigcup_{i \leq n} e_i(\D_{\sqrt t}) \m C_t$ outside the images of the disks (as well as a little into the images) and scale by a factor of $t$ to define it inside the disks (see Section \ref{secR2J}).

Let $g^f_0(\vec e,\vec u)=u_a(\vec e, \vec u) \circ d(\vec e)$. Define $g^m_0$ to be the pullback under $d(\vec e)$ of the metric on $C_t$ used in Section \ref{secapp}. This gives a map $g_0: F_0(\int\limits^{hol}_{\Sigma_g -pt} (R_2,\bold j)) \m \Map^*(\Sigma_g, \p) \times Met^{pc}(\Sigma_g)$. The reason that the map $u_a(\vec e, \vec u) \circ d(\vec e)$ is based is because the disks are embedded in $\Sigma_g -pt$ so $u_a(e, \vec u) (d(\vec e)(pt))=w(pt)$.

\textbf{Definition of $g_a$ on $F_1( \int\limits^{hol}_{\Sigma_g -pt} (R_2,\bold j))$}:

Let $A_N=B_N(\D(\Sigma_g -pt),D_2,R_2)$ be the space of $N$ simplicies of $ \int\limits^{hol}_{\Sigma_g -pt} (R_2,\bold j)$. The space $F_1( \int\limits^{hol}_{\Sigma_g -pt} (R_2,\bold j))$ is the homotopy coequalizer of two face maps $d_0, d_1 : A_1 \m A_0$. Extending $g_a$ to $F_1( \int\limits^{hol}_{\Sigma_g -pt} (R_2,\bold j))$ is the same data as a homotopy between $g_0 \circ d_0$ and $g_0 \circ d_1$. We will pick a homotopy through maps which are ``close'' to holomorphic maps. Therefore, in Section \ref{correct}, we will be able to perturb $g_a$ to a map to $\Hol^*(\Sigma_g,(\p,J))$. There are many other choices for homotopies; however, some choices may involve maps which cannot be corrected to $J$-holomorphic maps using the techniques of Section \ref{secapp}.

A point in $A_1$ is a rooted tree $T$ of depth $3$ with the following data. The root is labeled by disjoint embeddings $\vec{e} = (e_1, \ldots , e_k) \in \D(\Sigma_g-pt)$. The valence of the root is $k$. The internal verticies are labeled by elements of $\vec{e_i} = (e_{i1}, \ldots , e_{ik_i}) \in D_2'(k_i)$ and these have valency $k_i+1$. The leaves are labeled by elements of $R_2$ which can be thought of as maps $u_{ij}:\s \m \p$ with $u_{ij}(\infty) = w(e_i(0))$ and metrics $\gamma_{ij}$ on $\s$. We require that all subtrees of $T$ minus the root be labeled by composable elements. This translates to the following conditions. Let $\{T_i\}$ be the set of subtrees of $T$ minus the root. Either, the internal vertex of $T_{i}$ is labeled by the identity and there is only one leaf labeled by $(u_i,\gamma_i)$, or the internal vertex of $T_{i}$ is labeled by $(e_{i1},e_{i2}) \in D'_2(2)$ and there are two leaves labeled by $r_1$. Depending on the simplical coordinate, we will think of the label $r_1$ as being the degree one map $h^f_{e_i(0)}(r_1)$, $h^f_{e_i (e_{ij}(0))}(r_1)$, or $h^f_{z}(r_1)$ for some $z$ interpolating between $e_i(0)$ and $e_i (e_{ij}(0))$. Reorder so that $T_i$ has trivalent internal verticies for $i \leq l$ and has all internal verticies labeled by $id$ for $i >l$. Let $\vec e_{\leq l}=(e_1,\ldots e_l)$  and let $\vec e \circ \vec e_{**}=(e_1 \circ e_{11}, e_1 \circ e_{12}, \ldots , e_l \circ e_{l2}, e_{l+1}, \ldots e_k)$.


We now recall the effects of the face map $d_0:  A_1 \m A_0$. The map $d_0$ does not change the embeddings $e_i : \D \m \Sigma_g$. It replaces the trees $T_i$ with the composition of all of the elements labeling verticies of $T_i.$ Let $z_i=e_i(0) \in \Sigma_g$. In this notation, it replaces $T_i$ with a leaf labeled by $h_{z_i}(e_{i1},e_{i2},0)$ for $i \leq l$ and replaces $T_i$ with a leaf labeled by $(u_i,\gamma_i)$ for $i >l$. Here the number $0$ in $(e_{i1},e_{i2},0)$ is the cone coordinate and recall that $h(e_{i1},e_{i2},0)$ is equal to the approximate gluing map $h^a(e_{i1},e_{i2})$. Thus $d_0(T) =$ 
$$\vec e  \times (h_{z_1}(e_{11},e_{12},0), \ldots, h_{z_l}(e_{l1},e_{l2},0), (u_{l+1};\gamma_{l+1}), \ldots , (u_{k};\gamma_{k})).$$

In contrast, the face map $d_1: A_1 \m A_0$ does change the embeddings of disks into $\Sigma_g-pt$ by composing embeddings. In this notation, we have that $d_1(T)=$ $$\vec e \circ \vec e_{**} \times (h_{z_{11}}(r_1), \ldots,  h_{z_{l2}}(r_1), (u_{l+1},\gamma_{l+1}), \ldots, (u_{k},\gamma_{k}))$$ with $z_{ij}=e_i(e_{ij}(0))$.

To construct the homotopy between $g_0 \circ d_0$ and  $g_0 \circ d_1$, we will need to introduce some more notation. The map $g^f_0$ depends implicitly on a map $w : \Sigma_g \m \p$. For another map $w':  \Sigma_g \m \p$, collection of embeddings $e_i: \D \m \Sigma_g -pt$, and maps $u_i : \s \m \p$ with $u_i(\infty)$ equal to the image of the center of the $i$th disk under the map $w'$, let $G^f_{w'}(\vec e, \vec u): \Sigma_g \m \p$ be the function defined by the same formulas used to define $g^f_0$, after replacing $w$ by $w'$ in all of the formulas. We will add a subscript $t$ when we want to indicate the dependence on the gluing parameter.

Let $\gamma'$ be a metric on $\Sigma_g$ and $\vec \gamma=(\gamma_1, \ldots \gamma_k)$ be metrics on the $\s$'s. Likewise let $G^m_{\gamma'}(\vec e, \vec \gamma)$ be the metric on $\Sigma_g$ defined using the same formula as $g^m_0$, after replacing $\gamma_g$ by $\gamma'$ and using the metrics $\gamma_i$ on the $\s$'s. Given a metric $\gamma$ and a positive real number $R$, let $R \gamma$ denote the metric $\gamma$ scaled by $R$.

Given a collection of embeddings $e_i: \D \m \Sigma_g -pt$, let $\vec u^{const}$ denote the collection of constant maps with $u^{const}_i$ mapping $\s$ to the image under $w$ of the center of the $i$th disk. Let $w_{\vec e,s}=w$ for $s=0$ and let $w_{\vec e, s}=G_w (\vec e; \vec u_{const})_{st}$. Recall that the subscript $st$ indicates that we have replaced the gluing parameter $t$ with $st$. The maps $w_{\vec e,s}$ give a one parameter family of maps connecting $w$ with one which is constant inside of $e_i(\bigsqcup \D_{t^{1/4}})$ for $i \leq l$.

Let $C_t^{\leq l}$ be genus $g$ surface constructed by gluing $l$ copies of $\s$ to $\Sigma_g$ using the embeddings $\vec e_{\leq l}$. Recall that the first $l$ embeddings correspond to trees with trivalent verticies. The embeddings $e_{ij}$ for $i \leq l$ naturally can be viewed as maps to $C_t^{\leq l}$. Recall that there is natural holomorphic isomorphism $d(\vec e_{\leq l}): \Simga_g \m C_t^{\leq l}$. Let $E^t_{ij}: \D \m \Sigma_g$ be given by $E^t_{ij}=d(e_1, \ldots e_l)^{-1} \circ e_{ij}$ with $e_{ij}$ being viewed as embeddings of $\D$ into $C^{\leq l}_t$. Note that $E^1_{ij}=e_i \circ e_{ij}$.  Let $\vec E^t = (E^t_{11},E^t_{12},\ldots, E^t_{l2},e_{l+1},\ldots,e_k)$. Let $z^s_{ij} \in \Sigma_g -pt$ be points such that $w(z^s_{ij})=w_{e,s} (E^t_{ij}(0))$. These points can be picked continuously using the diffeomorphism $d(\vec e_{\leq l})$ and the bump function $p$. Since $w_{e,0}=w$, $z^0_{ij}=E^t_{ij}(0)$. Since $w_{e,1}$ is constant in a disk containing the images of each $E^t_{ij}$, $z^1_{ij}=z_i=e_i(0)$. Let $\vec h_s =$ $(h_{z^s_{11}}(r_1), \ldots,  h_{z^s_{l2}}(r_1), (u_{l+1};\gamma_{l+1}), \ldots, (u_{k};\gamma_{k})).$

We will describe the homotopy $H_s$ between $g_0 \circ d_0$ and $g_0 \circ d_1$ in two steps. As before, we let superscripts $f$ and $m$ denote respectively the projections onto the space of functions and metrics.  There are two main differences between $g_0^f \circ d_0$ and $g_0^f \circ d_1$. The map $g_0^f \circ d_0$ involves gluing degree one maps to a constant map and then gluing that to $w$ whereas $g_0^f \circ d_1$ involves gluing degree one maps directly to $w$. In the interval $[0,1/2]$ we will correct this difference. Thus, $H_{1/2}^m$ will not involve gluing in any constant maps. It will differ from $g^f_0 \circ d_1$ in that the degree one maps are glued in using the embeddings $E^t_{ij}$ instead of $e_i \circ e_{ij}$. In the interval $[1/2,1]$, this difference will be corrected. See Figure \ref{Hs} for a schematic depiction of this homotopy.

Let $C:[0,1]^2 \m [0,1]$ be a function with $C(t,0)=0$ and $C(t,1)=1$.  In the following subsection, we will put more requirements on the function $C$ to ensure that the approximate gluing map can be corrected.

For $s$ in the interval $[0,1/2]$, we define $H^f_s$ by the formula: $$H^f_s(T)= G_{w_{\vec e_{\leq l},1-2s}}(\vec E^t;\vec h^f_{1-2s})$$ with $\vec h^f_s$ the projection onto the space of functions. For $s$ in the interval $[0,1/2]$, we define $H^m_s$ as follows. Let $\gamma^s=$ $$G^m_{\gamma_g}(\vec e_{\leq l};C(t,s)\gamma_0,\ldots, C(t,s)\gamma_0)_{st}$$ for $s \in [0,1]$. As before, the subscript $st$ indicates that the gluing parameter $t$ has been replaced by $st$ in the formula for $G^m$ and the notation $C(t,s) \gamma_0$ means we have scaled the metric $\gamma_0$ by $C(t,s)$. As $s$ approaches $0$, the region where the function $w_{\vec e_{\leq l},s}$ is constant shrinks to a point because the gluing parameter is tending to zero. However, since we have modified the metric on the copies of $\s$, the areas of the disks $e_i(\D_{st})$ do not stay roughly constant but instead tend to zero. Let $H^m_s$ be given by $G_{\gamma_{1-2s}}(\vec E^t;\vec h^m_{1-2s})$.

Note that $H_0=g_0 \circ d_0$. It is not the case that $H^f_{1/2}$ equals $g_0 \circ d_1$. They would be equal if you were to replace the embeddings $e_1 \circ e_{ij}: \D \m \Sigma_g$ with $E^t_{ij}$ and replace the metric $\gamma_g$ with $\gamma^0$ in the definition of $d_1$. For $\tau \in (0,1]$, and $\gamma'$ a metric on $\Sigma_g$, let $d^{\tau,\gamma'}:A_1 \m A_0 $ be given by the same formula as $d_1:A_1 \m A_0 $ except with $e_1 \circ e_{ij}$ replaced with $E^\tau_{ij}$ and $\gamma_g$ replaced with $\gamma'$. Pick a path $\gamma'^s$ with $\gamma'^0=\gamma^0$ and $\gamma'^1=\gamma_g$. We require that the path $\gamma'^s$ satisfies identical continuity and commensurability requirements to the path used in Section \ref{secR2J}. For $s \in [1/2,1]$, define $H_s$ by the formula: $$H^f_s(T)=g_0 \circ d^{(2-2s)t+(2s-1),\gamma'^{2s-1}}_1.$$  Qualitatively this portion of the homotopy can be described as making the embeddings $E^t_{ij}$ larger and making the metric agree with $\gamma_g$ outside of the image of $\vec e \circ \vec e_{**}$. This completes the definition of $H_s$.

\begin{figure}[!ht]
\begin{center}\scalebox{.3}{\includegraphics{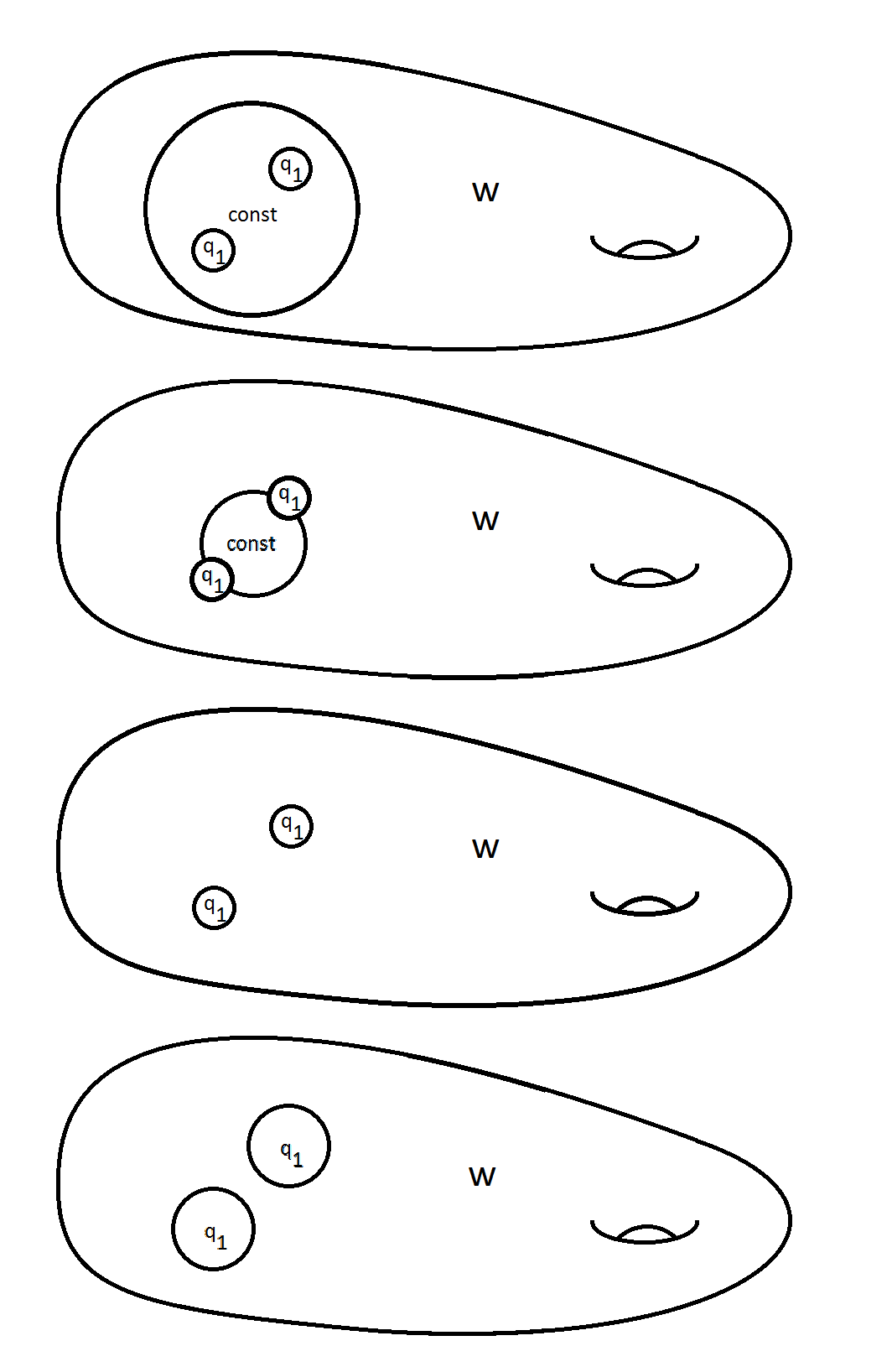}}\end{center}
\caption{The path $H_s$ at $s=0,1/4,1/2$ and $1$.}
\label{Hs}
\end{figure}

\textbf{Definition of $g_a$ on all of $ \int\limits^{hol}_{\Sigma_g -pt} (R_2,\bold j)$}:

We will now define the gluing map on $A_N \times \Delta^N$. View the $N$-simplex $\Delta^N$ as $\{(x_0,\ldots x_N) \in [0,1]^{N+1} \text{ with } \sum x_i=1\}$. Recall that for $T \in A_N$, $T$ can be described as a rooted tree with internal verticies labeled by elements of $D_2'$, root labeled by an element of $\D(\Sigma_g-pt)$ and leaves labeled by elements of $R_2$. Let $\{T_i\}$ be the subtrees of $T$ minus the root. Since we require that all of the labels of each $T_i$ be composable, the tree $T_i$ can have at most two leaves. Let $T_1, \ldots, T_l$ be the trees with two leaves each labeled by $r_1$. Let $\delta_i$ be the depth of the trivalent vertex. See Section \ref{secTCH} for the definition of depth or note that it is defined so that the facemap $d_N$ will split the subtree $T_i$ into two trees if and only if the depth of the trivalent vertex is $N$. Likewise, the composition $d_{N-1} \circ d_N$ will split $T_i$ into two trees if the depth is $N-1$ but will not if the depth is less than $N-1$. 


Note that the homotopy $H_s$ used to define the approximate gluing map only involved changing the metric and function inside the images of the embeddings $e_i$. For a collection of numbers $\vec s= (s_1, \dots, s_l)$ labeling the subtrees with trivalent verticies, let  $H^{\vec s} \in \Map^*(\Sigma_g, \p)  \times Met^{pc}(\Sigma_g)$ be the function and metric defined using the same formulas as those that were used to define $H^s$ except, in the $i$th disk, we replace $s$ with $s_i$. Let $\vec x=(x_0,\ldots x_N) \in \Delta^N$. Let $\delta(\vec x)=(x_{\delta_1}, \ldots,x_{\delta_l})$. We define $g$ on the space of $N$-simplicies by the formula: $$g_N(T,\vec{x})= H^{\delta(\vec x)}(T).$$ Note that these maps of spaces of simplicies assemble to a map on the geometric realization. The definition of $g_N$ is similar to the definition of $g^f_1$ except we have $N$ instead of $1$ variables parameterizing inserting in constant maps into disks and homotoping $E_{ij}^t$ to $e_i \circ e_{ij}$.  In particular, this definition agrees with the definition already given on the $1$-skeleton.

\subsection{Correcting the approximate gluing map}\label{correct}

The approximate gluing map $g^f_a : \int\limits^{hol}_{\Sigma_g -pt} (R_2,\bold j) \m \Map^*(\Sigma_g,\p) $ does not land in $\Hol^*(\Sigma_g,(\p,J))$. The purpose of this subsection is to correct this problem.  In this subsection, whenever we take a norm of $g^f_a(\kappa)$, the underlying metric on $\Simga_g$ will be $g^m_a(\kappa)$. To correct $g^f_a$ to a map to the space of $J$-holomorphic functions, we need to prove a few lemmas. First we need to show that $$lim_{t \m 0} \lVert \dbar^{nl}  g^f_a(\kappa) \rVert_{L^p}  = 0$$ for all $\kappa \in \int\limits^{hol}_{\Sigma_g -pt} (R_2,\bold j)$. Then we need to show that $$\tilde D^Z_{g_a(\kappa)} : \Gamma(g^f_a(\kappa)^*T \p) \m \Omega^{0,1}(g^f_a(\kappa)^*T\p)$$ is onto for $Z = \{ pt\}$. We will also need to show that for any compact set $K \subset \int\limits^{hol}_{\Sigma_g -pt} (R_2,\bold j)$, one can bound (independent of $t$) the volume of $g^m(\kappa)$, the norm of a right inverse of $\tilde D^Z_{g_a(\kappa)} :\Omega^{0,1}(T\Sigma_g) \times  \Gamma(g^f_a(\kappa)^*T \p) \m \Omega^{0,1}(g^f_a(\kappa)^*T\p)$ as well as $\lVert Dg^f_a(\kappa) \rVert_{L^p}$. Combining these results with Theorem \ref{theorem:implicitquant} will allow us to construct a map $g:K \m \Hol^*(\Sigma_g,(\p,J))$ for any compact subset $K \subset \int\limits^{hol}_{\Sigma_g -pt} (R_2,\bold j)$.

 Note that in Section \ref{secR2J}, the existence of these bounds followed from compactness and Theorem \ref{theorem:uniformrightinverse}; however, in this context, one needs to say more. 

In the previous subsection, we did not use that $w:\Sigma_g \m \p$ was $J$-holomorphic. However, in this subsection we always assume this.

\begin{lemma}
For any $p>2$ and $\kappa \in \int\limits^{hol}_{\Sigma_g -pt} (R_2,\bold j)$, $lim_{t \m 0} \lVert \dbar^{nl} g^f_a(\kappa) \rVert_{L^p}  = 0$.  
\label{lemma:small}
\end{lemma}

\begin{proof}

We will first show that $lim_{t \m 0} \lVert \dbar^{nl} h(r) \rVert_{L^p} =0$ for all $r \in R_2$. We have that $\dbar^{nl} r=0$ unless $r=(e_1,e_2,\sigma)$ with $\sigma \in [0,1/3]$ and $(e_1,e_2) \in D_2'(2)$. We have $lim_{t \m 0}  \lVert  \dbar^{nl} h(e_1,e_2,0) \rVert_{L^p}=0$ by Theorem \ref{theorem:smalldbar}. For $\sigma \in [0,1/3]$, by  Theorem \ref{theorem:implicitquantR}, we can estimate $ \lVert  \dbar^{nl} h(e_1,e_2,\sigma) \rVert_{L^p}$ by $ \lVert  \dbar^{nl} h(e_1,e_2,0) \rVert_{L^p}$. Thus, $ lim_{t \m 0} \lVert  \dbar^{nl} h(e_1,e_2,\sigma) \rVert_{L^p}=0$ for $\sigma \in [0,1/3]$ as well. 

Next we will show that, in the the image of $e'$, $h(e_1,e_2,\sigma)$ has bounded first derivate as $t$ goes to zero. For $\sigma \geq 2/3$, $h(e_1,e_2,\sigma)$ does not depend on $t$ so this is trivial. For $\sigma=0$, this is addressed in Remark \ref{iteratesmalldbar}. By construction, $h(e_1,e_2,1/3)$ is an approximate gluing map with gluing parameter going to zero. By the arguments of A.5 of \cite{MS}, $h(e_1,e_2,\sigma)$ converges in the sense of Section 4.4 of \cite{MS} to a nodal curve of Type 4. In particular, $h(e_1,e_2,1/3)$ converges to a constant map uniformly with all derivatives in any compact set not containing $e_1(0)$ and $e_2(0)$. This shows that the first derivatives of $h(e_1,e_2,1/3)$ tend to zero uniformly on the image of $e'$ and are thus bounded. For $\sigma \in [1/3,2/3]$, $h(e_1,e_2,\sigma)$ is defined in exactly the same way as $h(e_1,e_2,\sigma)$ except the gluing parameter $t$ is replaced by a number depending both on $t$ and $\sigma$. Thus, the first derivative of $h(e_1,e_2,\sigma)$ is also bounded as $t$ goes to $0$. For $\sigma \in (0,1/3)$, the derivatives are bounded because $h(e_1,e_2,\sigma)$ is defined as an interpolation between $h(e_1,e_2,0)$ and $h(e_1,e_2,1/3)$. 

The rest of the proof is just an iterated application of Theorem \ref{theorem:smalldbar} and Remark \ref{iteratesmalldbar}.  Let $w_{\vec e,\vec s}$ be the function defined in exactly the same way as $w_{\vec e,s}$ except we use the number $s_i$ to scale the gluing parameter in the $i$th disk. By Theorem \ref{theorem:smalldbar} and Remark \ref{iteratesmalldbar}, $w_{\vec e, \vec s}$ has bounded first derivative and $ lim_{t \m 0}  \lVert  \dbar^{nl} w_{\vec e,\vec s} \rVert_{L^p} =0$. Applying Theorem \ref{theorem:smalldbar} to maps constructed by gluing maps of the form $h(r)$ to $w_{\vec e, \vec s}$ shows $lim_{t \m 0} \lVert \dbar^{nl} g^f_a(\kappa) \rVert_{L^p}  = 0$.  

\end{proof}

At this point, we need to add more restrictions to our fixed $J$-holomorphic map $w:\Sigma_g \m \p$. For arbitrary $w:\Sigma_g \m \p$, it is not true that $\tilde D_{g_a(k)} : \Omega^{0,1}(T \Sigma_g) \times \Gamma(g^f_a(\kappa)^*T \p) \m \Omega^{0,1}(g^f_a(\kappa)^*T\p)$ is onto.  For example, we cannot take $w$ to be constant. Taking $w$ to be constant would have simplified many things. For example, we would not need to let the base point for $h_z$ vary. The property of $w$ that we need is the property that we prove often occurs in the following lemma.

\begin{lemma}
For all $g \in \N$, there exists $w:\Sigma_g \m \p$ and $pt \in \Sigma_g$ such that $w$ is $J$-holomorphic and $\tilde D^{pt}_w : \Omega^{0,1}(T\Sigma_g) \times \Gamma(w^*T \p)_{pt} \m \Omega^{0,1}(w^*T\p)$ is onto. Moreover, the degree of $w$ can be taken to be $g+1$. 

\label{gplusone}

\end{lemma}

\begin{proof}

Take a collection of rational curves $\bigsqcup \limits_{i=0}^g C_i$. Let $x_{i1},x_{i2} \in C_i$ for $i>0$ be distinct points. Likewise let $pt,y_{11},y_{12}, \ldots, y_{g1},y_{g2} \in C_0$ be distinct. Let $C$ be the nodal curve formed by identifying $x_{ij}$ with $y_{ij}$. This is depicted in Figure \ref{figure:2g1}. Note that $C \in \bar{ \mathcal{M}_g}$. Let $p_0 \in \p$ be the basepoint and pick a degree one $J$-holomorphic map $u_0:C_0 \m \p$ with $u_0(pt)=p_0$. For $i>0$, let $u_i:C_i \m \p$ be a degree one $J$-holomorphic map with $u_i(x_{ij})=u_0(y_{ij})$. By a positivity of intersections argument, the $u_i$ will have the same image. We can in fact construct such $u_i$ for $i>0$ by precomposition of $u_0$ with a degree one holomorphic map from $\s$ to itself. Observe that $c_1(u_i^*T\p)=3$. Since $3>1$ and $3>0$,  Corollary \ref{corollary:DuOnto} implies that $D^{pt}_{u_0}$ is onto. Since $3>2$ and $2>0$, Corollary \ref{corollary:DuOnto} implies that $D^{\{x_{i1},x_{i2}\}}_{u_i}$ is also onto. Let $u:C \m \p$ be the map induced by the maps $u_i$. Using the techniques of Section \ref{secdeg}, we can use this to show $D^{pt}_u$ is onto. This implies that $\tilde D^{pt}_u$ is also onto and so we can glue. 

Let $C_t$ be the associated smooth genus $g$ curve and $w : C_t \m \p$ be the $J$-holomorphic map resulting from correcting the approximate gluing map corresponding to $u$ (see Section \ref{secapp}). By Corollary \ref{corollary:surjectiveopen}, $D_{w}^{pt}$ and $\tilde D_{w}^{pt}$ are onto for sufficiently small $t$. Since each $u_i$ is degree $1$ and there are $g+1$ of these maps, $w$ will have degree $g+1$.

\end{proof}

\begin{figure}[!ht]
\begin{center}\scalebox{.4}{\includegraphics{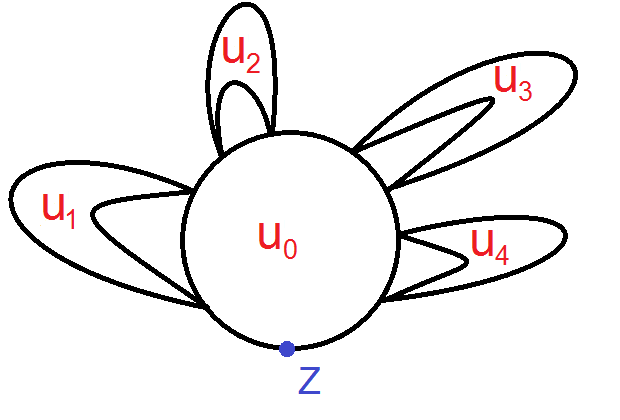}}\end{center}

\caption{The nodal curve $C$}
\label{figure:2g1}
\end{figure}

From now on, take $w: \Sigma_g \m \p$ as in the above lemma. Likewise, let $\bold j$ be the almost complex structure on $\Sigma_g$ resulting from the above gluing construction. If we could improve the degree bounds in the above lemma, we could accordingly improve the homological surjectivity range of Theorem \ref{theorem:big}.

\begin{lemma}
For any $\kappa \in \int\limits^{hol}_{\Sigma_g -pt} (R_2,\bold j)$, $$\tilde D^{pt}_{g_a(\kappa)} : \Omega^{0,1}(\Sigma_g) \times \Gamma(g^f_a(\kappa)^*T \p)_{pt} \m \Omega^{0,1}(g^f_a(\kappa)^*T\p)$$ is onto for $t$ sufficiently small.
\label{lemma:onto}
\end{lemma}

\begin{proof}

Let $\vec e$ be a collection of embeddings of disks into $\Sigma_g-pt$. Let $\vec s=\{s_1, \ldots, s_k\}$ with all $s_i \in (0,1)$. First we will prove that $\tilde D^{pt}_{w_{\vec e, \vec s}}$ is onto.  Let $C$ be a nodal curve which is $\Sigma_g$ with a finite number of $\s$ attached by attaching  $\infty \in \s$ to $z_i \in \Sigma_g$. Here the $z_i=e_i(0)$. Let $u$ be the map which is $w$ on $\Sigma_g$ and constant on all of the rational components. Call the constant maps $u_i : \s \m \p$.  We will show $\tilde D^{pt}_u$ is onto which will imply that  $\tilde D^{pt}_{w_{\vec e, \vec s}}$ is also onto for sufficiently small $t$ by Theorem \ref{theorem:aprxcoker}. Choose: $$(\eta_0,\eta_1, \ldots, \eta_k) \in \Omega^{0,1}( w^*T\p) \times  \prod \Omega^{0,1}( u_i^*T\p) =\Omega^{0,1}(u^*T\p).$$ Since $\tilde D_w^{pt}$ is onto, choose $(\beta_0,\xi_0) \in \Omega^{0,1}(\Sigma_g) \times \Gamma(w^*T \p)_{pt}$ with $\tilde D_w (\beta_0,\xi_0) = \eta_0$. Take $\alpha_i \in \Gamma(u_i^* T\p)$ with $\alpha(\infty) =\xi_0(z_i)$. In the proof of Theorem \ref{theorem:deg2orbifold}, we prove that $D^{\infty}_{u_i} :\Gamma(u_i^* T \p)_{\infty} \m \Omega^{0,1}(u_i^* T \p) $ is onto for $u_i$ constant. Since  $ D^{\infty}_{u_i}$ is a restriction of $\tilde D^{\infty}_{u_i} : \Omega^{0,1}(\s) \times \Gamma(u_i^* T \p)_{\infty} \m \Omega^{0,1}(u_i^* T \p)$, $\tilde D^{\infty}_{u_i}$ is also onto. Thus, we can find $(\beta_i,\xi_i) \in \Omega^{0,1}(\s) \times \Gamma(u_i^* T \p) $ with $\xi(\infty)=0$ and $\tilde D_{u_i} (\beta_i,\xi_i) = \eta_i - D_{u_i} \alpha_i$. Thus $\tilde D_{u_i} (\beta, \xi_i +\alpha_i) = \eta_i$ and $\xi_i(\infty) + \alpha_i(\infty) = \eta_i(\infty)$, $$((\beta_0,\xi_0), (\beta_1,\xi_1+ \alpha_1), \ldots , (\beta_k,\xi_k+ \alpha_k)) \in \Omega^{0,1}(T C) \times \Gamma(u^* T \p)$$ and $$\tilde D_u((\beta_0,\xi_0), (\beta_1,\xi_1+ \alpha_1), \ldots , (\beta_k,\xi_k+ \alpha_k))=(\eta_0,\eta_1, \ldots, \eta_k). $$ This proves that $\tilde D^{pt}_{w_{\vec e, \vec s}}$ is surjective for sufficiently small $t$.

Let $(e_i,e_j) \in D'_2(2)/\Sigma_2$, $s \in [0,1]$ and view $(e_i,e_j,s)$ as an element of $Cone(D'_2(s)/\Sigma_2)$. Next we observe that for sufficiently small $t$, $\tilde D^{\infty}_{h^f(e_i,e_j,s)}$ is onto. If $s\geq 1/3$ then $h^f(e_i,e_j,s)$ is a $J$-holomorphic degree two curve so $\tilde D^{\infty}_{h^fe_i,e_j,s)}$ is onto by Theorem \ref{theorem:deg2orbifold}. Let $u$ be a map from a nodal curve of Type 4 (see Section \ref{secdeg2}). These are maps $u_i: \s \m \p$ for $i=0,1,2$ with $u_0$ constant, $u_1$ and $u_2$ degree one. The two degree one curves are attached to the degree zero curve. By Theorem \ref{theorem:deg2orbifold}, $\tilde D_u^{\infty}$ is onto.  By Corollary \ref{corollary:surjectiveopen}, $D_{h^f(e_i,e_j,s)}$ is onto for $s \leq 1/3$ and $t$ sufficiently small.

There is a collection of embeddings $\hat e$ and holomorphic maps $\vec u$ such that the map $g^f_a(\kappa): \Sigma_g \m \p$ is of the form $G_{w_{\vec e,\vec s}} (\hat e; \vec u)$ with $u_i$ either degree one $J$-holomorphic, degree two $J$-holomorphic, or of the form $h^f(e_i,e_j,s)$. Thus $\tilde D_{u_i}^{\infty}$ is onto for small $t$. The same arguments as the ones used to prove $\tilde D^{pt}_{w_{\vec e,\vec s}}$ is onto for $t$ sufficiently small also can be used to show that $\tilde D^{pt}_{G_{w_{\vec e, \vec s}}(\hat e; \vec u)}$ is onto. Thus $$\tilde D^{pt}_{g^f_a(\kappa)} : \Omega^{0,1}(\Sigma_g) \times \Gamma(g^f_a(\kappa)^*T \p)_{pt} \m \Omega^{0,1}(g^f_a(\kappa)^*T\p)$$ is onto for $t$ sufficiently small.

\end{proof}

\begin{lemma}
Let $K \subset \int\limits^{hol}_{\Sigma_g -pt} (R_2,\bold j)$ be compact. There exists a bound independent of $t$ for the volume of $g^m(\kappa)$ for $\kappa \in K$. 

\end{lemma} 

\begin{proof}
The volume of $g^m(\kappa)$ can be bounded by a linear function of the number of subtrees of $\kappa$ plus the volume of $\gamma_0$. In compact families, the number of subtrees can be bounded. 
\end{proof}

Before we address a uniform bound for the right inverse, we need the following lemma. Recall that for $s$ a number and $\gamma$ a metric, $s \gamma$ denotes $\gamma$ scaled by $s$.

\begin{lemma}
\label{lemmascale}
Let $\gamma$ be a metric on $\Sigma_g$, $s \in (0,1]$, $M$  a symplectic manifold with choice of almost complex structure,  $u:\Sigma_g \m M$ a map and $Q: L^p(\Omega^{0,1}(u^*TM)) \m W^{1,p}(\Gamma(u^*TM)) $ a bounded linear transformation.  There is a bounded independent of $s$ for the norm of $Q$ with respect to the metric $s \gamma$. 

\end{lemma}

\begin{proof}
Here we use superscripts to indicate which metric we use. By definition $\lVert Q  \rVert^{s \gamma}= sup_{\eta \neq 0} \lVert Q \eta \rVert^{s \gamma}_{W^{1,p}}/ \lVert \eta \rVert^{s \gamma}_{L^p}$. Let $\nabla:W^{1,p}(\Gamma(u^*TM))  \m L^p(\Omega^{0,1}(u^*TM))$ denote the Levi-Civita connect induced by the symplectic and almost complex structure on $M$. We have: $$\lVert \eta \rVert^{s \gamma}_{L^p} = \left (\int_{\Sigma_g}  (s^{1-p/2} )|  \eta |^p   \right)^{1/p} \text{ and}$$ $$ \lVert Q \eta \rVert^{s \gamma}_{W^{1,p}} = \left (\int_{\Sigma_g}  s|Q \eta|^p+ (s^{1-p/2} )|\nabla Q \eta |^p   \right)^{1/p}.$$ Here the integrals are with respect to the metric $\gamma$ (see page 172 of \cite{MS}). Thus $ (\lVert Q \eta \rVert^{s \gamma}_{W^{1,p}}/ \lVert \eta \rVert^{s \gamma}_{L^p})^p=$ $$ (s^{p/2})(\int_{\Sigma_g} |Q \eta|^p)/(\int_{\Sigma_g} | \eta|^p)+(\int_{\Sigma_g} |\nabla Q \eta|^p)/(\int_{\Sigma_g} | \eta|^p).$$ Both terms are bounded by $(\lVert Q \eta \rVert^{\gamma})^p$ and so the claim follows. 
\end{proof}

Recall that the metric $g_a^m$ depended on a function $C:[0,1]^2 \m [0,1]$. We now require that $C(t,s)$ be a continuous function with $C(t,s)=(st)^2$ for $s \leq 1/2$. This means that for $s \leq 1/2$, as the disks in which the function $w_{\vec e,s}$ is constant shrink, their areas shrink proportionally to the square of their radii. With this added assumption, we can prove the following lemma.

\begin{lemma}
\label{rightbounded}
For any $\kappa \in  \int\limits^{hol}_{\Sigma_g -pt} (R_2,\bold j)$, there exists a bound independent of $t$ for the norm of a right inverse of $\tilde D^{pt}_{g_a(\kappa)} : \Gamma(g^f_a(\kappa)^*T \p) \m \Omega^{0,1}(g^f_a(\kappa)^*T\p)$.

\end{lemma} 

\begin{proof}

The only portion of this that does not follow from a straightforward adaptation of the methods of \cite{S} is constructing a continuously varying right inverse for $\tilde D_{w_{\vec e,s}}$ which is bounded as $t$ tends to zero. Using the construction from Remark A.4.1 of \cite{MS}, one can construct a continuously varying family of right inverses $Q_1^s$ for $\tilde D_{w_{\vec e,s}}$. Since there is a point-wise bound on the first derivative of $w_{\vec e,s}$ which is independent of $t$, the family $Q_1^s$ will have uniformly bounded norm with respect to the metric $\gamma_g$. Because of our added restriction on $C$, for $s \in [0,1/2]$, the metric $\gamma^s$ is commensurable with $\gamma_g$. Thus, on $[0,1/2]$, $Q_1^s$ is uniformly bounded with respect to $\gamma^s$ as well. Note that $\gamma^s$ depends on $t$. 

On the interval $(0,1]$, there is another procedure for producing a right inverse to $\tilde D_{w_{\vec e,s}}$. The construction of the right inverse in \cite{S}, does not use that the metrics on the irreducible components of the nodal curves are fixed. They can vary as long as the right inverses on each irreducible component stay bounded. In our case, the metrics on the leaves vary as we rescale them using the function $C$. However, by Lemma \ref{lemmascale}, this does not cause the right inverses to become unbounded. Thus, we can find a family $Q_2^s$ of right inverses of $\tilde D_{w_{\vec e,s}}$ for $s>0$ which stay bounded as $t$ goes to zero.

Let $b:[0,1] \m [0,1]$ be some function with $b(0)=1$ and $b(s)=0$ for $s \geq 1/2$. Let $Q^s=b(s)Q_1^s+(1-b(s))Q_2^s$. This gives a right for $\tilde D_{w_{\vec e,s}}$ which is defined for all $s$ and bounded as $t$ goes to zero. 

The rest of the proof involves iterating the results in \cite{S}.

\end{proof}

Using similar arguments to those of Lemma \ref{lemma:small}, we get the following. 

\begin{lemma}
\label{boundedderivlemma}
For any $\kappa \in  \int\limits^{hol}_{\Sigma_g -pt} (R_2,\bold j)$, there exists a bound independent of $t$ for the $L^p$ norm of $Dg^f_a(\kappa)$.

\end{lemma}

Combining these lemmas, we can construct the gluing map, the analogue of the scanning map for holomorphic topological chiral homology.

\begin{theorem}
For any compact set $K \subset \int\limits^{hol}_{\Sigma_g -pt} (R_2,\bold j)$, there exists vector fields $\xi_\kappa \in \Gamma(g^f_a(\kappa)^* T \p)$ continuously depending on $\kappa \in K$ such that $g(\kappa)=exp_{g^f_a(\kappa)}(\xi_k) \in \Hol^*(\Sigma_g,(\p,J))$.

\label{g}
\end{theorem}

\begin{proof}
By Lemmas \ref{lemma:small}, \ref{lemma:onto}, \ref{rightbounded} and \ref{boundedderivlemma}, the map $g^f_a(\kappa)$ satisfies the hypotheses of Theorem \ref{theorem:implicitquant} for sufficiently small $t$ depending on $\kappa \in \int\limits^{hol}_{\Sigma_g -pt} (R_2,\bold j)$. Since $K$ is compact, we can choose $t$ such that for all $\kappa \in K$, $g^f_a(\kappa)$ satisfies the hypotheses of Theorem \ref{theorem:implicitquant}. Thus we can find such a vector field $\xi_\kappa \in \Gamma(g^f_a(\kappa)^{*}T \p)$ such that $exp_{g^f_a(\kappa)}(\xi_{\kappa}) \in \Hol^*(\Sigma_g,(\p,J))$.

\end{proof}

\begin{corollary}
The image in homology of $g^f_a:\int\limits^{hol}_{\Sigma_g -pt} (R_2,\bold j) \m \Map^*(\Sigma_g,\p)$ is contained in the image in homology of $i:\Hol^*(\Sigma_g,(\p,J)) \m \Map^*(\Sigma_g,\p)$.

\label{theorem:actualgluingexists}
\end{corollary}

\begin{proof}
Let $K \subset \int\limits^{hol}_{\Sigma_g -pt} (R_2,\bold j)$ be compact and let $g:K \m \Hol^*(\Sigma_g,(\p,J))$ and $\xi_\kappa \in \Gamma(g^f_a(\kappa)^{*}T \p)$ be as in Theorem \ref{g}. Let $H: K \times [0,1] \m \Map^*(\Sigma_g,\p)$ be given by the formula $H(\kappa,s)=exp_{g^f_a(\kappa)}(s\xi_{\kappa})$. The map $H$ gives a homotopy between $i \circ g$ and $g^f_a$ restricted to $K$. Since the image of $i \circ g$ in homology is contained in the image of $i$ and every homology class is in the image of a compact set, the claim follows.

\end{proof}

 \section{Homotopy between gluing maps and scanning maps}\label{sechomotopy}

By Corollary \ref{theorem:actualgluingexists}, the image in homology of the approximate gluing map is contained in the image in homology of the inclusion of the space of $J$-holomorphic functions into the space of all continuous functions. The goal of this section is to study the image in homology of the approximate gluing map and compare it with the scanning and stabilization maps defined in Section \ref{secTCH} and further studied in Section \ref{KY}. In the first subsection, we investigate the dependence of  $g^f_a :  \int\limits^{hol}_{\Sigma_g -pt} (R_2,\bold j) \m \Map^*(\Sigma_g, (\p,J)) $ on $w$. In the second subsection, we show that $h:R_2 \m \Map^*(\s,\p)$ can be interpreted as a map of partial $D_2$-algebras up to homotopy. In the third subsection, we prove the main theorem of this paper, Theorem \ref{main}.

We denote the projection $g^f_a :  \int\limits^{hol}_{\Sigma_g -pt} (R_2,\bold j) \m \Map^*(\Sigma_g, \p) $ simply by $g_a$ as metrics will not be relevant in this section. We likewise denote $h^f:R_2 \m \Map^*(\s,\p)$ simply by $h$.

\subsection{Untwisting the approximate gluing map}\label{secuntwist}

The purpose of this subsection is to study how the map $g_a :  \int\limits^{hol}_{\Sigma_g -pt} (R_2,\bold j) \m \Map^*(\Sigma_g, \p) $ depends on the function $w : \Sigma_g \m M$.

First note that the formula for $g_a$ extends to all of $\int_{\Sigma_g -pt} R_2$ since the fact that the disks were embedded holomorphically was only relevant for correcting the approximate gluing map. Moreover, the definition of $g_a$ did not use that the map $w$ was holomorphic.

\begin{definition}

For $v:\Sigma_g \m \p$, let $g_a^v:\Sigma_g : \int_{\Sigma_g -pt} R_2 \m \Map^*(\Sigma_g, \p)$ be the function defined using the same formula as $g_a :  \int\limits^{hol}_{\Sigma_g -pt} (R_2,\bold j)  \m \Map^*(\Sigma_g, \p)$ with $w$ replaced with $v$.

\end{definition}

Recall that the name of the base point in $\Sigma_g$ is $pt$ and the name of the base point in $\p$ is $p_0$. Recall that $\pi_0(\Map^*(\Sigma_g, \p)) =\Z$ and this integer is called degree. In Section \ref{secTCH}, we defined stabilization maps which give maps $T_a : \Map^*(\Sigma_g, \p) \m \Map^*(\Sigma_g, \p)$ for every element $a \in \Map^c(B-pt,\p)$. Here $B$ is an open ball centered around $pt$. The map $T_a : \Map^*(\Sigma_g,\p) \m \Map^*(\Sigma_g, \p)$ can be defined as follows. Let $d : \Sigma_g -pt \m \Sigma_g-B$ be a diffeomorphism. For $m \in \Simga_g$ and $f\in \Map^*(\Sigma_g,\p)$, let $T_a : \Map^*(\Sigma_g,\p) \m \Map^*(\Sigma_g,\p)$ be the map defined by the following formula: $$
T_a(f)(m)=
\begin{cases}
p_0 & \mbox{if } m=pt \\
a(m)  & \mbox{if } m \in B-pt \\
f(d^{-1}(m))  & \mbox{otherwise.} \\
\end{cases}
$$ Note that $\pi_0(\Map^c(B-pt,\p)) =\Z$ and we also call this integer degree. Up to homotopy, the map $T_a$ only depends on the degree of $a$. Let $T_k : \Map_d^*(\Sigma_g, \p) \m \Map_{d+k}^*(\Sigma_g, \p)$ denote a stabilization map which raises degree by $k$. Note that topological chiral homology is functorial with respect to open embeddings. Let $e:\Sigma_g-B \m \Sigma_g -pt$ be the inclusion map.  

\begin{proposition}
Let $k$ be the degree of $w: \Sigma_g \m \p$ and let $const : \Sigma_g \m \p$ be the constant map. The following diagram commutes up to homotopy: $$
\begin{array}{ccccccccl}
 \int_{\Sigma_g -pt} R_2 &  \overset{g^w_a}{\m} &   \Map^*(\Sigma_g, \C P^2) \\

 \uparrow e_*  &&     &            \\

 \int_{\Sigma_g -B} R_2   & & \uparrow T_k \\

 \uparrow d_*  &           \\

 \int_{\Sigma_g -pt} R_2  &\overset{ g^{const}_a}{\m} & \Map^*(\Sigma_, \C P^2).\\

\end{array}$$


\label{theorem:untwisting}

\end{proposition}
\begin{proof}

 The map $g^v_a$ does depend on $v$. However, up to homotopy, it only depends on the class of $v$ in $\pi_0(\Map^*(\Sigma_g, \p))$. One can easily construct maps of every degree which are constant off of $B$. Thus the map $w$ is homotopic through based maps to a map $w' : \Sigma_g \m \p$ which is constant off of $B$. Let $a \in \Map^c(B -pt , \p)$ be the restriction $w'$ to $B$. The above diagram commutes on the nose if one replaces $w$ with $w'$ and interprets $T_k$ as $T_a$. Thus, the above diagram homotopy commutes.
\end{proof}


Note that all maps $T_k$ are homotopy equivalences. Additionally, the inclusion $\int\limits^{hol}_{\Sigma_g - pt} (R_2,\bold j) \m \int_{\Sigma_g -pt } R_2 $ is a weak homotopy equivalence. Thus, to prove the main theorem, it suffices to prove that $g^{const}_a: \int_{\Sigma_g-pt} R_2 \m \Map^*(\Sigma_g,\p)$ is a stable homology equivalence. In other words, we have the following lemma.

\begin{lemma}
The image in homology of $g^{w}_a:\int\limits^{hol}_{\Sigma_g -pt} (R_2,\bold j) \m \Map^*(\Sigma_g,\p)$ contains the image in homology of: $$T_{g+1} \circ g^{const}_a : \int_{\Sigma_g - pt} R_2 \m \Map^*(\Sigma_g,\p).$$
\label{lemma:imagecontained}
\end{lemma}

\subsection{Comparison of the $D_2$ actions on $R_2$ and $\Omega^2 \p$}\label{seccompair}

The goal of this subsection is to modify the map $h : R_2 \m \Map^*(\s,\p)$ so that we may interpret it as a map of partial $D_2$-algebras to $\Omega^2 \p$. Recall that we view $Cone(D_2'(2)/\Sigma_2)$ as the quotient of $[0,1] \times D_2'(2)/\Sigma_2$ identifying all points of the form $(\kappa,1)$. Since we do not let base points vary in this subsection, we drop the $z$ subscript on the function $h:R_2 \m \Map^*(\s,\p)$.  

\paragraph{Modifying $h$}

We will define a function $\hat h : R_2 \m \Map^*(\s,\p)$ which is homotopic to $h : R_2 \m \Map^*(\s,\p)$. Recall that the definition of $h$ required fixing a function $q_1 \in \Map^*_1(\s,\p)$. Let $q_1^s$ be a path of degree one maps with $q_1^0=q_1$ and $q_1^1$ constant on the image of $e'$ in $\s$. We define $\hat h(r_1)$ to be $q_1^1$. 

Let $f_s:\s \m \s$ be a path of degree one maps with $f_0$ the identity, $f_s(\infty)=\infty$ and $f_s(e'(\D)) \subset e'(\D)$ for all $s$ and $f_1(e'(\D))=\{\infty\}$. 

For $s \geq 2/3$, we define $\hat h(e_1,e_2,s)$ to be $h(e_1,e_2,3s-2) \circ f_1$. 

Recall that $h(e_1,e_2,0)$ was an approximate gluing map denoted $h_a$. This map implicitly depended on the gluing parameter $t$ and the degree one map $q_1$. Here we indicate this dependence with a superscript. Let $\tau(s)$ be the affine function with $\tau(1/3)=1$ and $\tau(2/3)=t$. For $s \in [1/3,2/3]$, we define $\hat h(e_1,e_2,s)$ to be $h_a^{\tau(s),q_1^{2-3s}}(e_1,e_2) \circ f_1 $.

For $s \in [0,1/3]$, we define $\hat h(e_1,e_2,s)$ to be $h_a^{1,q_1^1}(e_1,e_2) \circ f_{3s} $.

\vspace{12pt}

The key features of the function $\hat h$ are that it is constant on $e'(\D)$ and we have set the gluing parameter equal to $1$ in the approximate gluing map. Normally it does not make sense to let the gluing parameter go to $1$. However, it does make sense when both functions are constant on the images of the embeddings as is the case here. 
 
The usual model of $\Omega^2 X$ used to define the $D_2$-action is $\Omega^2 X =\Map((\bar \D,\partial \bar \D), (X,pt))$. Since $\hat h$ is constant outside of $\D \subset \s$, we can view $\hat h$ as a map $\hat h:R_2 \m \Omega^2 \p$. This allows us to make sense of the following lemma.

\begin{lemma}
The map $\hat h:R_2 \m \Omega^2 \p$ is a map of partial $D_2$-algebras.

\label{lemma:moruptohomotopy}
\end{lemma}

\begin{proof}

This follows from the definition of $\hat h$. This is similar to the fact that $E^t_{ij}=e \circ e_{ij}$ when $t=1$, an observation made in Section \ref{secJTCH}

\end{proof}

\subsection{The main theorem}

As was noted in Section \ref{secuntwist}, to prove that $i: \Hol^*(\Sigma_g ,(\p,J)) \m \Map^*(\Sigma_g , \p)$ is a stable homology surjection, it suffices to show that $g^{const}_a: \int_{\Simga_g} R_2 \m \Map^*(\Sigma_g,\p)$ is a stable homology equivalence. This is the goal of this section. This will follow from the following lemma. As in Corollary \ref{KYpackage}, let $h':\int_{\Sigma_g-pt} R_2 \m \int_{\Sigma_g-pt} \Omega^2 \p$ be the map induced by $\hat h:R_2 \m \Omega^2 \p$.

\begin{lemma}\label{scanningequalsgluing}
The map $g^{const}_a: \int_{\Sigma_g-pt} R_2 \m \Map^*(\Sigma_g,\p)$ is homotopic to the map $s \circ h':\int_{\Sigma_g -pt} R_2 \m \Map^*(\Sigma_g,\p)$.
\end{lemma}

\begin{proof}
The maps are equal if one replaces $h$ with $\hat h$ in the definition of $g^{const}_a$ and sets the gluing parameter equal to $1$ everywhere except in the definition of $\hat h$. Note that we can set the gluing parameter equal to $1$ since we are gluing maps which are constant in images of the embeddings. Since $h$ is homotopic to $\hat h$, $g^{const}_a$ and $s \circ \hat h'$ are homotopic. 

\end{proof}

We can now prove the main theorem. 

\begin{theorem}
Let $J$ be any compatible almost complex structure on $\p$. The inclusion map $i: \Hol_d^*(\Sigma_g,(\p,J)) \m \Map_d^*(\Sigma_g,\p)$ induces a surjection on homology groups $H_i$ for $i \leq (d-g-1)/3$.

\end{theorem}
By Corollary \ref{theorem:actualgluingexists},  the image in homology of $i:\Hol_d^*(\Sigma_g,(\p,J)) \m \Map_d^*(\Sigma_g,\p)$ contains the image in homology of $g^w_a:(\int\limits^{hol}_{\Sigma_g -pt} (R_2,\bold j))_d \m \Map_d^*(\Sigma_g,(\p,J))$. By Lemma \ref{lemma:imagecontained}, this contains the image of $T_{g+1} \circ g^{const}_a : \int_{\Sigma_g - pt} R_2 \m \Map^*(\Sigma_g,\p)$. Note that $T_{g+1}: \Map_{d-g-1}^*(\Sigma_g,\p)\m \Map_d^*(\Sigma_g,\p)$ is a homotopy equivalence. By Lemma \ref{scanningequalsgluing}, the image in homology of $g^{const}_a: \int_{\Sigma_g-pt} R_2 \m \Map^*(\Sigma_g,\p)$ and $s \circ h':\int_{\Sigma_g -pt}R_2 \m \Map^*(\Sigma_g,\p)$ agree since the two maps are homotopic. By Corollary \ref{KYpackage}, the map $s \circ h':(\int_{\Sigma_g -pt} R_2)_{d-g-1} \m \Map_{d-g-1}^*(\Sigma_g,\p)$ induces an isomorphism on $H_i$ for $i \leq (d-g-1)/3$. Thus, $i:\Hol_d^*(\Sigma_g,(\p,J)) \m \Map_d^*(\Sigma_g,\p)$ induces a surjection in that range.

\begin{remark}

\label{oneJ}

One can also use the techniques of this paper to establish a version of Theorem \ref{theorem:big} with fixed almost complex structures on the domain. If $\bold j$ is an almost complex structure satisfying the conclusions of Lemma \ref{gplusone} with $\tilde D^{pt}_w$ replaced with $D^{pt}_w$, then $i: \Hol_d^*((\Sigma_g,\bold j),(\p,J)) \m \Map_d^*(\Sigma_g,\p)$ is a stable homology surjection. Currently the author does not know how to show that $i: \Hol_d^*((\Sigma_g,\bold j),(\p,J)) \m \Map_d^*(\Sigma_g,\p)$ is a stable homology surjection for all $\bold j$ since it is unclear how to improve Lemma \ref{gplusone} to allow arbitrary almost complex structures on  $\Sigma_g$.

\end{remark}

\bibliography{thesis2}{}
\bibliographystyle{alpha}

\end{document}